%% file: Tran_prob.tex
\documentclass[11pt,reqno]{amsart}

\input{latexformat}
\input{latexnewmacros}

\numberwithin{equation}{section}
\usepackage{amssymb}
\usepackage{hyperref}
\usepackage{mathrsfs}
\usepackage{slashed}
\usepackage[usenames]{color}
\usepackage{url}
\usepackage{verbatim}
\usepackage{times}

\hypersetup{pdftitle={Transition probabilities for degenerate diffusions arising in population genetics}}

\hypersetup{pdfauthor={Charles L. Epstein and Camelia A. Pop}}

\begin{document}

\title[Degenerate diffusions]{Transition probabilities for degenerate diffusions arising in population genetics}
\author[C. L. Epstein]{Charles L. Epstein}
\address[CLE]{Department of Mathematics, University of Pennsylvania, 209 South 33rd Street, Philadelphia, PA 19104-6395}
\email{cle@math.upenn.edu}

\author[C. A. Pop]{Camelia A. Pop}
\address[CP]{School of Mathematics, University of Minnesota, 206 Church St. SE, Minneapolis, MN 55455}
\email{capop@umn.edu}
\thanks{CL Epstein's research is partially supported by NSF grant  DMS-1507396. CA Pop's research is partially supported by NSF grant  DMS-1714490.}
\date{\today{ }\hhmm}

\begin{abstract}
We provide a detailed description of the structure of the transition
probabilities and of the hitting distributions of boundary components of a
manifold with corners for a degenerate strong Markov process arising in
population genetics. The Markov processes that we study are a generalization of
the classical Wright-Fisher process. The main ingredients in our proofs are
based on the analysis of the regularity properties of solutions to a forward
Kolmogorov equation defined on a compact manifold with corners, which is
degenerate in the sense that it is not strictly elliptic and the coefficients
of the first order drift term have mild logarithmic singularities.
\end{abstract}

% AMS 2010 subject classifications (used in AMS journals)
% Primary
% 35H99   Close-to-elliptic equations and systems 
% 35J70  	Degenerate elliptic equations
% 35J86  	Linear elliptic unilateral problems and linear elliptic variational inequalities
% 49J40  	Variational methods including variational inequalities
% 35R45  	Partial differential inequalities
%
% Secondary
% 60G22  	Fractional processes, including fractional Brownian motion
% 35R35  	Free boundary problems
% 60J60  	Diffusion processes
% 49J20  	Optimal control problems involving partial differential equations

% AMS keywords (used in AMS journals)

% Acknowledge support

\maketitle

\tableofcontents

\section{Introduction}
\label{sec:Introduction}

We study generalized Kimura operators, which were introduced by C. Epstein and R. Mazzeo in \cite{Epstein_Mazzeo_annmathstudies}, as an extension of the Wright-Fisher operator, \cite{Fisher_1922, Fisher_1930, Wright_1931}, \cite[Equation (3.3)]{Shimakura_1981}; it is a basic model for the evolution of gene frequencies in population genetics. Related versions of generalized Kimura operators were previously studied in connection with superprocesses \cite{Athreya_Barlow_Bass_Perkins_2002} and with applications to the dynamics of populations \cite{Cerrai_Clement_2001, Cerrai_Clement_2003, Cerrai_Clement_2004, Cerrai_Clement_2007}.

Generalized Kimura operators act on functions defined on  compact manifolds
with corners, $P,$  see~\cite[Definition
2.2.1]{Epstein_Mazzeo_annmathstudies}. If  $P$ is a compact manifold with
corners of dimension $N$, then a point $p\in\partial P$ has the property that there
are non-negative integers, $n,m\in\NN$ such that $n+m=N$, and  a relatively
open neighborhood in $P,$ which is homeomorphic to a relatively open
neighborhood $V,$ of the origin in $\bar S_{n,m}.$ We can assume that $p$ maps
to $(0,0),$ where we let
$$
S_{n,m}:=\RR^n_+\times\RR^m\quad\hbox{ and }\quad \RR_+:=(0,\infty).
$$
The point $p$ belongs to a boundary stratum of co-dimension $n.$ The
  $l$-simplex in $\RR^l:$ \
$$\{x\in\RR^l:\: x_1+\cdots +x_l\leq 1\text{ and }0\leq
x_j\text{ for }1\leq j\leq l\}$$ is a typical manifold with corners. It has
boundary strata of co-dimensions 1 through $l.$

When written in an adapted system of local coordinates on $\bar S_{n,m}$, a generalized Kimura operators take the normalized form:
\begin{equation}
\label{eq:Operator}
\begin{aligned}
Lu &= \sum_{i=1}^n\left(x_i u_{x_ix_i} 
+b_i(z) u_{x_i}\right) +\sum_{i,j=1}^n x_ix_ja_{ij}(z)u_{x_ix_j}
+\sum_{l,k=1}^md_{lk}(z)u_{y_ly_k}
\\
&\quad
+\sum_{i=1}^n\sum_{l=1}^mx_ic_{il}(z)u_{x_iy_l} 
+ \sum_{l=1}^me_l(z)u_{y_l},
\end{aligned}
\end{equation}
where we denote $z=(x,y)\in\bar S_{n,m}$. In our article we study the transition probabilities of the Markov process with infinitesimal generator given by the generalized Kimura operator defined on a compact manifold with corners. We provide a precise description of the structure of the transition probabilities of Kimura diffusions,  the probabilistic distributions of the paths of the process before absorption on one of the boundary hypersurfaces of the manifold, and  the hitting distributions of suitable portions of the boundary of the manifold. These are equivalent to the study of the fundamental solution of the parabolic equation defined by generalized Kimura operator,  the Dirichlet heat kernel, and  the caloric measure, respectively. 

In order to prove the existence and uniqueness of a (strong) Markov process on $P$ with infinitesimal generator given by a generalized Kimura operator, we introduce the \emph{martingale problem} associated to this operator. The solution to the martingale problem is a probability measure on the space of continuous functions $C([0,\infty);P)$ endowed with the uniform convergence on compact sets in $[0,\infty)$. We let $\cB:=\cB(C([0,\infty);P))$ and $\cB_t:=\cB(C([0,t];P))$, for all $t\geq 0$, be the corresponding Borel $\sigma$-algebras. As usual, we call $\{\cB_t\}_{t\geq 0}$ the canonical filtration.

\begin{defn}[Martingale problem]
\label{defn:Martingale_problem}
Let $p\in P$. A probability measure $\QQ^p$ on the measurable space $(C([0,\infty);P), \cB)$ is a solution to the martingale problem associated to a generalized Kimura operator if
\begin{equation}
\label{eq:Mart_problem_initial_cond}
\QQ^p(\omega(0)=p) = 1,
%\quad\forall\,\omega\in C([0,\infty);P),
\end{equation}
and, for all $\varphi\in C^{\infty}([0,\infty)\times P)$, we have that
\begin{equation}
\label{eq:Mart_problem_martingales}
M^{\varphi}_t:=\varphi(t,\omega(t))-\varphi(0,\omega(0))-\int_0^t \left(\varphi_s(s,\omega(s))+L\varphi(s,\omega(s))\right)\, ds,\quad\forall\, t\geq 0,
\end{equation}
is a martingale with respect to the filtration $\{\cF_t\}_{t\geq 0}$, where $\cF_t:=\cap_{s>t} \cG_s$, and $\cG_t$ is the augmentation under $\QQ^p$ of the canonical filtration $\{\cB_t\}_{t\geq 0}$.
\end{defn}

For a point $z=(x,y)\in \bar S_{n,m}$ and $r>0$, we let
\begin{equation}
\label{eq:Ball_infty}
B^{\infty}_r(z) := \{z'=(x',y')\in S_{n,m}:|x'_i-x_i|<r, \quad\forall\, 1\leq i\leq n,\quad |y'_l-y_l|<r,\quad\forall\,1\leq l\leq m\}.
\end{equation}
When $z=0\in \bar S_{n,m}$, we write for brevity $B^{\infty}_r$ instead of
$B^{\infty}_r(0)$. To prove the well-posedness of the martingale problem
associated to generalized Kimura operators, we make the following assumptions
about the coefficients of the operator when written in a local system of
coordinates.  It is sufficient to state the conditions on the operator $L$
defined in \eqref{eq:Operator} in a neighborhood of the origin, say on
$B^{\infty}_1$:

\begin{assump}
\label{assump:Coeff}
The coefficients of the operator $L$ defined in \eqref{eq:Operator} satisfy:
\begin{enumerate}
\item[1.] The functions $a_{ij}(z)$, $b_i(z)$, $c_{il}(z)$, $d_{lk}(z)$, and $e_l(z)$ are smooth and bounded functions on $\bar B^{\infty}_1$, for all $1\leq i, j\leq n$ and $1\leq l, k \leq m$.
\item[2.] The drift coefficients $b_i(z)$ satisfy the \emph{non-negativity condition}: for all $1\leq i\leq n$ we have that that
\begin{equation}
\label{eq:Nonnegative}
b_i(z) \geq 0,\quad\forall\, z\in \bar B^{\infty}_1.
\end{equation}
\item[3.] The \emph{strict ellipticity} condition holds: there is a positive constant, $\Lambda$, such that for all $z \in \barB^{\infty}_1$, $\xi\in\RR^n$, and $\eta\in\RR^m$, we have 
\begin{equation}
\label{eq:Uniform_ellipticity}
\begin{aligned}
\sum_{i=1}^n \xi_i^2
+ \sum_{i,j=1}^n a_{ij}(z)\xi_i\xi_j
+\sum_{i=1}^n \sum_{l=1}^m c_{il}(z)\xi_i\eta_l 
+\sum_{l,k=1}^m d_{lk}(z)\eta_k\eta_l
\geq \Lambda\left(|\xi|^2+|\eta|^2\right).
\end{aligned}
\end{equation}
\end{enumerate}
\end{assump}
These are essentially the minimal assumptions needed to define a generalized
Kimura operator, and are the assumptions used
in~\cite{Epstein_Mazzeo_annmathstudies}. A slightly larger class can be
considered by allowing cross terms of the forms
$\sqrt{x_ix_j}a_{ij}\pa_{x_i}\pa_{x_j},$ and
$\sqrt{x_i}c_{il}\pa_{x_i}\pa_{y_l},$ along with smallness hypotheses on the
coefficients $a_{ij}$ and $c_{il}$ along the appropriate boundary
components. We leave these generalizations to the interested reader.

The main difficulties in studying questions related to generalized
Kimura operators are due to the fact that the operator $L$ in
\eqref{eq:Operator} is not strictly elliptic up to the boundary of
$S_{n,m}$. From \eqref{eq:Uniform_ellipticity}, we see that the coefficients of
the second order derivatives $\partial^2_{x_i}$ are linearly proportional to
the distance to the boundary, and so the terms $x_i\partial^2_{x_i}$ and
$b_i(z)\partial_{x_i}$ scale in the same way. To see this we consider the following rescaling. For $\lambda\in (0,1)$,  let 
\begin{equation}
\label{eq:Scaling}
x=\lambda x',\quad y=\sqrt{\lambda} y',\quad\hbox{ and }\quad u(x,y) = :v(x',y'),
\end{equation}
and  notice that when $u(z)$ is a solution to $Lu(z) = 0$, then the rescaled function $v(z')$ satisfies the equation $L' v(z') = 0$,  where the operator $L'$ is given by
\begin{align*}
L'v(z') 
&=
\sum_{i=1}^n\left(x'_i v_{x'_ix'_i} + b_i(\lambda x', \sqrt{\lambda} y') v_{x'_i}\right) 
+\sum_{i,j=1}^n \lambda x'_ix'_ja_{ij}(\lambda x', \sqrt{\lambda} y')v_{x'_ix'_j}
\\
&\quad
+\sum_{l,k=1}^md_{lk}(\lambda x', \sqrt{\lambda}y') v_{y'_ly'_k}
+\sum_{i=1}^n\sum_{l=1}^m \sqrt{\lambda} x'_i c_{il}(\lambda x', \sqrt{\lambda} y')v_{x'_iy'_l} 
+ \sum_{l=1}^m \sqrt{\lambda} e_l(\lambda x', \sqrt{\lambda} y')v_{y'_l}.
\end{align*}
Thus the rescaling \eqref{eq:Scaling} gives us solutions to an equation defined by a new operator, $L'$, that also satisfies Assumption \ref{assump:Coeff}. This scaling property implies that, at the boundary, the first order transverse derivatives $b_i(z)\partial_{x_i}$ are not of lower order, as in the case of strictly elliptic operators, and they play an important role in the analysis of the transition probabilities. 

We can now state our first main result:

\begin{thm}[Well-posedness of the martingale problem]
\label{thm:Wellposed_mart_problem}
Suppose that the generalized Kimura operator defined on a compact manifold  with corners,$P,$ satisfies Assumption \ref{assump:Coeff} when written in a local system of coordinates. Then the martingale problem in Definition \ref{defn:Martingale_problem} is well-posed.
\end{thm}

We call a \emph{(generalized) Kimura process} starting from a point $p\in P$ the canonical process, $\{\omega(t)\}_{t\geq 0}$ endowed with the probability distribution $\QQ^p$, where $\QQ^p$ is the unique solution to the martingale problem introduced in Definition \ref{defn:Martingale_problem}. As a consequence of the well-posedness of the martingale problem, we have:

\begin{cor}[The strong Markov property]
\label{cor:Strong_Markov}
Suppose that the generalized Kimura operator defined on a compact manifold $P$
with corners satisfies Assumption \ref{assump:Coeff} when written in an adapted
system of local coordinates. For $p\in P$, let $\QQ^p$ be the unique solution to the martingale problem in Definition \ref{defn:Martingale_problem} given by Theorem \ref{thm:Wellposed_mart_problem}. Then the family of probability measures $\{\QQ^p:\, p\in P\}$ defined on the filtered space 
$(C([0,\infty);P),\cB, \{\cB_t\}_{t\geq 0})$ satisfies the strong Markov property. That is, the following hold:
\begin{enumerate}
\item[(i)] For all $p\in P$, we have that $\QQ^p(\omega(0)=p)=1$.
\item[(ii)] For all sets, $A\in \cB$, the map $P \ni p\mapsto\QQ^p(A)$ is measurable.
\item[(iii)] For all stopping times $\tau$ relative to the filtration $\{\cB_t\}_{t\geq 0}$, $A\in \cB$, and $p\in P$, we have that
\begin{equation}
\label{eq:Strong_Markov}
\QQ^p(\theta_{\tau}^{-1} A|\cB_{\tau}) = \QQ^{\omega(\tau(\omega))}(A),\quad\, \QQ^p\hbox{-a.s.,}\quad\hbox{ on } \tau<\infty,
\end{equation}
where the shift operator is defined by $(\theta_s\omega)(t):=\omega(t+s)$, for all $t \geq 0$ and for all $\omega\in C([0,\infty)\times P)$.
\end{enumerate}
\end{cor}

We can now describe our results on the transition probability distributions of
Kimura diffusions and several related questions. The boundary of $P$ plays an
central role in this analysis. For this purpose, we review from \cite[\S
2.1]{Epstein_Mazzeo_annmathstudies} the definitions and introduce notation
related to the stratification of the boundary of $P$, which we use throughout
the article. Fixing $1\leq n\leq N$, the closure of the set of boundary points
on $\partial P$ that have a neighborhood homeomorphic to a neighborhood of the
origin in $\bar S_{n,m}$ can be written as the union of connected compact
$m$-dimensional manifolds with corners, where $m=N-n$. We denote
these components (strata) of the boundary by $P^m_1,\ldots,P^m_{\eta_m}$, and
call them corners of co-dimension $n$. Corners of co-dimension $1$ are
called the boundary hypersurfaces, or faces of the manifold $P$. When we refer to
the boundary hypersurfaces of the manifold $P,$ in a context where the lower
dimensional boundary components of the manifold are not important, we use the
simpler notation $H_1,\ldots,H_{\eta}$ to refer to them, instead of
$P^{N-1}_1,\ldots,P^{N-1}_{\eta_{N-1}}$.

Recall from \cite[\S 2.2]{Epstein_Mazzeo_annmathstudies} that the principal symbol of the generalized Kimura operator induces a Riemannian metric on the manifold $P$. 
For all $p\in P$ and $r>0$, we denote by $B_r(p)$ the relatively open ball
centered at $p$ and of radius $r$ computed with respect to the Riemannian
distance on $P$. The ball $B_r(p)\subseteq P$ should not be confused with the
sup-norm Euclidean balls $B^{\infty}_r(z) \subseteq S_{n,m}$ defined in \eqref{eq:Ball_infty}.
Letting $H_i$ be a boundary hypersurface and defining $\rho_i(p)$ to be the Riemannian distance from the point $p\in P$ to $H_i$, we recall from \cite[Proposition 2.1]{Epstein_Mazzeo_2016} that
\begin{equation}
\label{eq:Weight_b_i}
B_i\restrictedto_{H_i} := \frac{1}{4}L\rho^2_i\restrictedto_{H_i},\quad\forall\, 1\leq i\leq \eta,
\end{equation}
are coordinate-invariant quantities, which we call the \emph{weights} of the
generalized Kimura operator, \cite[Definition
2.1]{Epstein_Mazzeo_2016}. In population genetics the space $P$ is typically
  an $N$-simplex, $\Sigma_N;$ a point $(x_1,\dots,x_N)\in\Sigma_N$ gives the
  frequencies of individuals of different types in the population. The
  interiors of the faces of the simplex correspond to situations where a single
  type has gone extinct. If type $i$ has disappeared and the frequencies of the
  remaining types are $(y_1,\dots,y_{N-1}),$ then the weight
  $B_i(y_1,\dots,y_{N-1})$ is the instantaneous rate at which type $i$ 
  reappears in the population.

We also let $B_i$ denote a smooth extension from $H_i$ to $P$ of the
coefficients defined in \eqref{eq:Weight_b_i}. When we write the
  Kimura operator $L$ in an adapted system of local coordinates, it is natural
  to choose these extended weights $B_i$ to coincide with the coefficients
  $b_i$ of the first order derivatives $\partial_{x_i}$ in
  \eqref{eq:Operator}. We adopt this convention throughout our paper, except
  for the case when a weight $B_i\restrictedto_{H_i}$ is a constant, $C$, then
  the extension to $P$ is also taken to be the constant $C$\footnote{This
    assumption is not necessarily adopted in \cite{Epstein_Mazzeo_2016}.}. By
a slight abuse of terminology, we sometimes call the smooth extensions of the
quantities defined in \eqref{eq:Weight_b_i} weights of the generalized Kimura
operator. 

The next results in our article are derived under the following \emph{cleanness} condition:
\begin{assump}[The cleanness condition]
\cite[Definition 3.2.3]{Epstein_Mazzeo_annmathstudies}
\label{assump:Cleanness}
We say that the weights $\{B_i:1\leq i\leq \eta\}$ of the generalized Kimura
operator satisfy the \emph{cleanness condition} if there is a positive
constant, $\beta_0$, such that for all $1\leq i\leq \eta$ we have that either 
\begin{equation}
\label{eq:Cleanness}
B_i \equiv 0\quad\hbox{ on } H_i,
\quad\hbox{or }\quad
B_i \geq \beta_0 >0\quad\hbox{ on } H_i.
\end{equation}
\end{assump}

\noindent
In most of the results of this paper we make the following
  assumptions about the Kimura operator under study.
\begin{defn}[{\bf Standard assumptions}]
\label{defn:Standard_assump}
A generalized Kimura operator that satisfies Assumption
\ref{assump:Cleanness} and, when written in an adapted system of
coordinates, it takes the form of the operator $L$ in
\eqref{eq:Operator} that satisfies Assumption \ref{assump:Coeff} will
be said to satisfy the \emph{standard assumptions}.
\end{defn}

When the cleanness condition holds, we also say that the generalized Kimura
operator meets the boundary of the manifold $P$ cleanly. If the weight $B_i$
corresponding to the boundary hypersurface $H_i$ is $0$, then we say that the
operator is \emph{tangent} to $H_i$ or that $H_i$ is a tangent boundary
component. Otherwise, if the weight $B_i$ corresponding to the boundary
hypersurface $H_i$ is positive, then we say that the operator is
\emph{transverse} to $H_i$ or that $H_i$ is a transverse boundary
component.  The cleanness hypothesis implies a certain uniformity
  in the ``ellipticity'' of the operator $L$ along the faces of $\pa P,$ which
  holds up to the boundaries of the faces. Thus far, much less is known about
  the refined analytic properties of the heat kernel in the case where this
  hypothesis does not hold. It can be expected to have considerably more
  complicated singularities in the incoming variables, along the loci where
  cleanness fails. Even the structure of the null-space of $L$ is far from
  clear without this hypothesis, and dually the existence and uniqueness of a
  stationary measure has not been studied. 

We also introduce the notation:
\begin{align}
\label{eq:I_tangent}
I^T &:= \{i:\, 1\leq i\leq \eta,\quad B_i\restrictedto_{H_i} = 0\},\\
\label{eq:I_transverse}
I^{\pitchfork} &:= \{i:\, 1\leq i\leq \eta,\quad B_i\restrictedto_{H_i} > 0\},\\
\label{eq:Tangent_boundary}
\partial^T P &:= \cup_{i\in I^T} H_i,\\
\label{eq:Transverse_boundary}
\partial^{\pitchfork} P &:= \cup_{i\in I^{\pitchfork}} H_i.
\end{align}
As in \cite[\S 2]{Epstein_Mazzeo_2016}, we define a measure on the manifold $P$. Let $dV_P$ be a smooth positive density on $P$. Then we define the weighted measure $d\mu_P$ by
\begin{equation}
\label{eq:Weight}
d\mu_P(p) := \prod_{i=1}^{\eta} \rho_i(p)^{B_i(p)-1}dV_P.
\end{equation}
We often use $d\mu$ to denote $d\mu_P,$ when no confusion will arise.

In an adapted system of local coordinates, the weight measure $d\mu_P$ is a multiple by a smooth function of
$$
d\mu(z) = \prod_{i=1}^n x_i^{b_i(z)-1}dx_i\,\prod_{l=1}^m dy_l.
$$
The motivation to introduce the previous weight measure is most easily understood in the one-dimensional case when the operator $Lu=xu_{xx}+bu_x$ can be written in divergence form (or as a Sturm-Liouville operator) as
$$
Lu = \frac{1}{m(x)}\frac{d}{dx}\left(\frac{u(x)}{s(x)}\right),
$$
where $m(x)=x^{b-1}$ and $s(x)=x^{-b}$. In the theory of one-dimensional
diffusions, \cite[Chapter 16]{Breiman_book}, $m(x)$ and $s(x)$ are called the
speed measure and the scale function, respectively, and they completely
characterize the boundary behavior of the corresponding diffusion. The weight
measure $d\mu_P$ is our generalization to the multi-dimensional case of the
speed measure $m(x)\,dx$ in the one-dimensional case.  If all
weights are positive, then it is shown in~\cite{Epstein_Mazzeo_2016} that the
stationary measure for the process generated by $L$ is a bounded function times
$d\mu_P.$ 

The weighted $L^2$-space, $L^2(P;d\mu),$ consists of measurable functions, $u:P\rightarrow\RR$, with the property that the norm
\begin{equation}
\label{eq:L_2_norm}
\|u\|^2_{L^2(P;d\mu)} := \int_P |u|^2\, d\mu_P<\infty.
\end{equation}

We prove in \S \ref{subsec:Tangent_boundary} that, after hitting a tangent boundary component of the manifold $P$, the Kimura diffusion is absorbed. We have the following result concerning the marginal distributions of the paths of the Kimura diffusion that are not absorbed on the tangent boundary components $\partial^T P$ of the compact manifold $P$ with corners. 

\begin{thm}[Dirichlet heat kernel]
\label{thm:Distribution_absorbed}
Suppose that the generalized Kimura operator satisfies the standard
assumptions. For all $t>0$ and $p\in P\backslash \partial^T P$, there
is a non-negative measurable function, $k(t,p,\cdot) \in L^2(P;d\mu)
\cap L^1(P;d\mu)$, such that for all Borel measurable sets,
$B\subseteq P\backslash\partial^T P$, we have that
\begin{equation}
\label{eq:Density_Dirichlet_heat_kernel}
\QQ^p(\omega(t) \in B, t < \tau_{\partial^T P}) = \int_B k(t,p,p')\, d\mu_P(p'),
\end{equation}
where we let $\tau_{\partial^T P}$ be defined by
\begin{equation}
\label{eq:tau_tangent}
\tau_{\partial^T P} := \inf\{t \geq 0: \, \omega(t) \in\partial^T P\},
\end{equation}
and we let $\QQ^p$ be the unique solution to the martingale problem in Definition \ref{defn:Martingale_problem}. 
\end{thm} 

We call the integral kernel $k(\cdot,\cdot,\cdot)$ constructed in Theorem
\ref{thm:Distribution_absorbed} the \emph{Dirichlet heat kernel}.  In \S
\ref{sec:Dirichlet_heat_kernel}, we prove that $k(\cdot, \cdot, q)$ is a weak solution to the \emph{backward Kolmogorov equation} defined by the generalized Kimura operator, and
$k(\cdot,p,\cdot)$ is a weak solution to the \emph{forward Kolmogorov equation} or the \emph{Fokker-Planck equation} defined by the adjoint of the generalized Kimura operator. The adjoint operator is computed with respect to the weighted $L^2$-space
$L^2(P;d\mu)$ and, in general, it is not an operator of Kimura-type due to the
fact that the (possibly) non-constant weights of the generalized Kimura
operator can lead to logarithmic singularities in the coefficients of the first
order terms in the adjoint operator. Nevertheless, relying on results in
\cite{Epstein_Pop_2016}, we are able to prove suitable estimates for the Dirichlet heat
kernel that allow us to establish Theorem \ref{thm:Distribution_absorbed}. 
%Removed -- no longer needed
%Regularity properties of the Dirichlet heat kernel are proved in Theorems \ref{thm:Sup_est_tran_prob}, \ref{thm:Sup_est_tran_prob_tangent}, and \ref{thm:Tran_prob_backward}, and more detailed estimates can be deduced from \cite[Theorem 1.3]{Epstein_Pop_2016}.

We next give a description of the hitting distributions of the Kimura diffusion of the tangent portion of the boundary 
$\partial^T P$ of the manifold $P$. For this purpose, we need to introduce
additional notation. Let $\Sigma$ be a boundary component to which the
generalized Kimura operator is tangent, i.e. there are indices
$i_1,\ldots,i_k\in I^T$ such that $\Sigma$ is a connected component of
\begin{equation}
\label{eq:Sigma}
\bigcap_{j=1}^k H_{i_j}.
\end{equation}
We have the following observation due to Sato.

\begin{rmk}[Restriction operator to tangent boundary components]
\label{rmk:Sato}
\cite{Sato_1978}, \cite[Lemma 2.4]{Shimakura_1981}
Let $\Sigma$ be a tangent boundary component. There is a well defined Kimura diffusion operator $L_{\Sigma}$ on
$\Sigma,$ so that for any smooth function $u$ on $\Sigma$, we have that
\begin{equation}
\label{eq:Restriction_operator}
  L_{\Sigma}u=(LU)\restrictedto_{\Sigma},
\end{equation}
where $U$ is any smooth extension of $u$ to a full neighborhood of $\Sigma$ in $P.$ The restriction operator $L_{\Sigma}$ also satisfies Assumptions \ref{assump:Cleanness} and \ref{assump:Coeff}, when written in a local system of coordinates. We denote by 
$d\mu_{\Sigma}$ the weighted measure defined by analogy to the measure $d\mu$ in \eqref{eq:Weight}, but where we replace the role of the operator $L$ by that of $L_{\Sigma}$ and the role of the manifold $P$ by that of $\Sigma$, i.e. we set
\begin{equation}
\label{eq:mu_Sigma}
d\mu_{\Sigma}(p) := \prod\limits_{\{j\notin \{i_1,\dots,i_k\}\}}
\rho_{j}(p)^{B_{j}(p)-1}\, dV_{\Sigma}(p).
\end{equation}
\end{rmk}

We can now state

\begin{thm}[Hitting distribution of the interior of tangent hypersurfaces]
\label{thm:Hitting_distribution_hypers}
Suppose that the generalized Kimura operator satisfies the standard assumptions.
Then, for all $i\in I^T$ and for all $p\in P\backslash\partial^T P$, there is a Borel measurable function,
\begin{equation}
\label{eq:Hitting_density_H_i}
h_i(\cdot,p,\cdot):(0,\infty)\times H_i \rightarrow [0,\infty),
\end{equation}
with the property that, for all Borel measurable sets, $B_i\subseteq\Int(H_i)$, we have that
\begin{equation}
\label{eq:Hitting_H_i}
\QQ^p\left(\tau_{\partial^T P}\in(s,s+ds], \omega(\tau_{\partial^T P})\in B_i \right)
 = \int_{B_i} h_i(s,p,q)\, d\mu_{H_i}(q)\, ds,
\end{equation}
where the stopping time $\tau_{\partial^T P}$ is defined in \eqref{eq:tau_tangent} and $\QQ^p$ is the unique solution to the martingale problem in Definition \ref{defn:Martingale_problem}. Moreover, we have that
\begin{equation}
\label{eq:h_i_reg}
h_i(\cdot,p,\cdot) \in C^{\infty}((0,\infty)\times (\Int(H_i)\cup (\partial H_i \cap \Int(\partial^T P)))),\quad\forall\, p\in P\backslash\partial^T P,
\end{equation}
where $\Int(\partial^T P)$ is taken with respect to $\partial P$.
\end{thm}

The densities $h_i$, for $i\in I^T$, are also called the \emph{caloric measures} corresponding to the boundary hypersurface $H_i$. In 
\S \ref{prop:Integral_representation}, we prove that the caloric measures $h_i$ are the normal derivatives on $H_i$ of the Dirichlet heat kernel, and we use this property to prove in Theorem \ref{thm:Doubling_property} that the hitting distributions satisfy a \emph{doubling property}. Thus, the regularity properties of the Dirichlet heat kernel play a significant role in our article. We next state the result that the probability of hitting a boundary component of co-dimension larger than $1$ is zero.

\begin{thm}[Hitting probability of two hypersurfaces]
\label{thm:Hitting_corner}
Suppose that the generalized Kimura operator satisfies the standard assumptions.
Then, for all $p\in P\backslash\partial^T P$ and for all $i\in I^T$ and $j\in I^T\cup I^{\pitchfork}$, we have that
\begin{equation}
\label{eq:Hitting_corner}
\QQ^p(\omega(\tau_{\partial^T P}) \in H_i\cap H_j) = 0,
\end{equation}
where the stopping time $\tau_{\partial^T P}$ is defined in \eqref{eq:tau_tangent} and $\QQ^p$ is the unique solution to the martingale problem in Definition \ref{defn:Martingale_problem}.
\end{thm}

Theorems \ref{thm:Distribution_absorbed}, \ref{thm:Hitting_distribution_hypers}, and \ref{thm:Hitting_corner} are the main ingredients in the proof of the structure of the transition probabilities of Kimura diffusions. The rough idea is that the probability distribution of the Kimura process at a time $t$ can be decomposed into the distribution of the paths of the process that have not been absorbed on one of the tangent boundary components of the manifold up to time $t$, which are described by the Dirichlet heat kernel in Theorem \ref{thm:Distribution_absorbed}, and by the distributions of the paths on each tangent boundary component, of dimension $0$ to $N-1$. To describe the latter probability distribution a more careful analysis is needed and it is based on the following observations:
\begin{itemize}
\item[1.]
When the Kimura process hits the tangent boundary component, $\partial^T P$, it always hits it in the interior of a hypersurface component, with probability $1$. This is proved in Theorem \ref{thm:Hitting_corner}.
\item[2.]
The hitting distribution of a tangent hypersurface is described in Theorem \ref{thm:Hitting_distribution_hypers}.
\item[3.]  After hitting a tangent boundary hypersurface, a path remains there,
  a.s., and the Kimura process performs a diffusion of the same type, but in this
  lower dimensional compact manifold with corners. This statement is a
  consequence of Remarks \ref{rmk:Sato} and \ref{rmk:Sato_equation}.
\end{itemize}
It seems very likely that without the cleanness hypothesis the
  qualitative behavior of paths will be considerably more complicated. For
  example, if the vector field is tangent to a relatively compact subset of a
  boundary face, then it seems possible that paths will be absorbed into this
  portion of the boundary, only to re-emerge later from another portion of the
  boundary to which the vector field is transverse.

Before giving the statement of Theorem \ref{thm:Tran_prob}, we need to introduce additional notation.
For all $(t,p)\in(0,\infty)\times P$ and for all Borel measurable sets $B\subseteq P$, we denote the transition probabilities of the canonical process on $C([0,\infty);P)$ under the probability measure $\QQ^p$ in Definition \ref{defn:Martingale_problem} by
\begin{equation}
\label{eq:tran_prob}
\Gamma^P(t,p,B) = \QQ^p(\omega(t)\in B).
\end{equation}
For all $1\leq k\leq N$, where we recall that $\hbox{dim } P = N$, we let $d_k$ denote the number of boundary components of dimension $k$ to which the generalized Kimura operator is tangent. For all $1\leq i\leq d_k$, we denote by $P^k_i$ the boundary components of dimension $k$ to which the generalized Kimura operator is tangent.  We let $d\mu_{P^k_i}$ denote the measure constructed in Remark \ref{rmk:Sato}, for all $1\leq i\leq d_k$ and for all $1\leq k\leq N-1$. We can now state the main result of our article that describes the structure of the transition probabilities of the Kimura diffusion on the compact manifold $P$ with corners. 

\begin{thm}[Transition probabilities]
\label{thm:Tran_prob}
Let $P$ be a connected compact manifold with corners. Suppose that the
generalized Kimura operator satisfies the standard assumptions.  Let
$p\in P\backslash \partial^T P$. Then there are non-negative,
measurable functions,
\begin{align*}
&k^P(\cdot,p,\cdot):(0,\infty)\times P\rightarrow [0,\infty),\\
&k^{P;P^k_i}(\cdot,p,\cdot):(0,\infty)\times P^k_i\rightarrow [0,\infty),\quad\forall\, 1\leq i\leq d_k,\quad\forall\,0\leq k\leq N-1,
\end{align*}
such that the transition probabilities of the canonical process on $C([0,\infty);P)$ under the probability measure $\QQ^p$ in Definition \ref{defn:Martingale_problem} has the following expression:
\begin{equation}
\label{eq:Tran_prob}
\begin{aligned}
\Gamma^P(t,p,dw) &= k^{P}(t,p,w)\, d\mu_P(w) \,\delta_{\Int(P)}(w)\\
&\quad + \sum_{k=1}^{N-1}\sum_{i=1}^{d_k} k^{P;P^k_i}(t,p,w)\, d\mu_{P^k_i}(w) \,\delta_{\Int(P^k_i)}(w)\\
&\quad + \sum_{i=1}^{d_0} k^{P;P^0_i}(t,p,w)\, \delta_{P^0_i}(w).
\end{aligned}
\end{equation}
\end{thm}

\begin{rmk}[Structure of the transition probabilities]
\label{rmk:tran_prob}
The function $k^P(\cdot,p,\cdot)$ appearing in the structure of the transition
probability in \eqref{eq:Tran_prob} is the Dirichlet heat kernel constructed in
Theorem \ref{thm:Distribution_absorbed}.  When all the weights of the generalized Kimura operator are positive, the transition probability $\Gamma^P(t,p,dw)$ coincides with the measure defined by the Dirichlet heat kernel $k^{P}(t,p,w)\, d\mu_P(w)$, which satisfies suitable Gaussian estimates as proved in \cite[Theorems 1.2 and 5.2]{Epstein_Mazzeo_2016}. However, when the operator $L$ has weights equal to zero, the transition probability measure contains additional terms described in formula \eqref{eq:Tran_prob}, and in this case the Dirichlet heat kernel $k^{P}(t,p,w)$ satisfies suitable Gaussian estimates only away from the tangent boundary components. Along tangent boundary components we describe the estimates satisfied by the Dirichlet heat kernel in Theorem \ref{thm:Sup_est_tran_prob}.

The remaining terms appearing in
equality \eqref{eq:Tran_prob}, $k^{P;P^k_i}(\cdot,p,\cdot)$, are convolutions
between Dirichlet heat kernels, as constructed in Theorem
\ref{thm:Distribution_absorbed}, and hitting distributions, as constructed in
Theorem \ref{thm:Hitting_distribution_hypers}, relative to suitable boundary
components of the compact manifold $P$ with
corners. 
\end{rmk}

As a consequence of Theorem \ref{thm:Tran_prob}, we have the following result in which we prove that the generalized Kimura process spends zero time on the boundary components to which the operator is transverse:

\begin{cor}
\label{cor:Zero_times_spent_on_trans_boundary}
Let $P$ be a connected compact manifold with corners. Suppose that the generalized Kimura operator satisfies the standard assumptions, then, for all $p\in P$, we have that
\begin{equation}
\label{eq:Zero_times_spent_on_trans_boundary}
\int_0^{\infty}\mathbf{1}_{\{\omega(t) \in\partial^{\pitchfork} P\}}\, dt = 0,\quad\QQ^p\hbox{-a.s.},
\end{equation}
where $\QQ^p$ is the unique solution to the martingale problem in Definition \ref{defn:Martingale_problem}.
\end{cor}

\begin{rmk}[The cleanness condition (Assumption \ref{assump:Cleanness})]
The cleanness condition plays a central role in our paper. The construction of the Dirichlet heat kernel in Theorem \ref{thm:Distribution_absorbed} is based on the fact that the Kimura operator $L$ in \eqref{eq:Operator} defines a continuous bilinear form that satisfies the G\r{a}rding inequality. This result is not true in general when the weights of the operator $L$ do not satisfy property \eqref{eq:Cleanness}. Moreover, our result in Theorem \ref{thm:Hitting_corner} concerning the fact that the Kimura diffusion hits with zero probability a boundary component of co-dimension two, when it first hits the boundary of the manifold, no longer holds in the absence of the cleanness condition. For example, the two-dimensional operator $Lu = x_1u_{x_1x_1}+x_2u_{x_2x_2}+x_2u_{x_1}+x_1u_{x_2}$ does not satisfy the cleanness condition, and the underlying diffusion does not satisfy the conclusion of Theorem \ref{thm:Hitting_corner}. More precisely, the underlying diffusion is given by:
\begin{align*}
dX_1(t) &= X_2(t)\, dt+ \sqrt{2X_1(t)}\, dW_1(t),\\
dX_2(t) &= X_1(t)\, dt+ \sqrt{2X_2(t)}\, dW_2(t),
\end{align*}
where $\{(W_1(t), W_2(t))\}_{t\geq 0}$ is a two-dimensional Brownian
motion. The results of our paper imply that the coordinate process
$\{X_i(t)\}_{t\geq 0}$ is absorbed in the interior of the boundary $\{x_i=0\}$
with zero probability, for $i=1,2$. However, the one-dimensional sum
process $S(t):=X_1(t)+X_2(t)$ satisfies the stochastic differential
equation $dS(t) = S(t)\, dt+ \sqrt{2S(t)}\, dW(t)$, 
where the Brownian motion $\{W(t)\}_{t\geq 0}$ is defined by
$$
dW(t) = \sqrt{\frac{X_1(t)}{X_1(t) + X_2(t)}} \,dW_1(t) + \sqrt{\frac{X_2(t)}{X_1(t) + X_2(t)}} \,dW_2(t),\quad\forall\,t\geq 0,
$$
Such a process hits the origin, a boundary component of co-dimension two, with probability $1$ and so Theorem \ref{thm:Hitting_corner} does not hold.
\end{rmk}

\subsection{Comparison with previous research}
\label{sec:Previous_research}

Shimakura obtained a closed-form expression for the transition probabilities of the multi-dimensional Wright-Fisher process, 
\cite{Fisher_1922, Fisher_1930, Wright_1931}, \cite[Equation (3.3)]{Shimakura_1981}, which is described by its infinitesimal generator, 
\begin{equation}
\label{eq:WF_operator}
L_{\hbox{\tiny{WF}}} u 
= \frac{1}{2}\sum_{i,j=1}^n\left(\delta_{ij}x_i-x_ix_j\right)u_{x_ix_j} 
+ \sum_{i=1}^n\left(b_i-x_i\sum_{j=1}^{n+1}b_j\right)u_{x_i},
\end{equation}
where $b_1$, $b_2$, \ldots, $b_{n+1}$ are non-negative constants and $\delta_{ij}$ denotes the Kronecker delta symbol. The operator
$L_{\hbox{\tiny{WF}}}$ acts on functions defined on the $n$-dimensional simplex, 
\begin{equation}
\label{eq:Simplex}
\Sigma_n := \left\{x=(x_1,\ldots,x_n) \in\RR^n: 0\leq x_i,\quad\forall\,1\leq
  i\leq n,\quad x_1+\ldots+x_n\leq 1\right\}.
\end{equation}
In this case all weights are constant, so the cleanness condition
  is immediate. Shimakura's formula again expresses the heat kernel as a sum of
  terms indexed by the strata of the $\pa\Sigma_n$ to which
  $L_{\hbox{\tiny{WF}}}$ is tangent.

Shimakura used completely different ideas than those employed in our article to
find the structure of the transition probabilities for the Wright-Fisher
process in \cite[Formula (5.11)]{Shimakura_1981}. Shimakura computes, in closed
form, the eigenvalues and eigenfunctions (eigenpolynomials) of the
Wright-Fisher operator. This is possible because $L_{\hbox{\tiny{WF}}}$ maps
polynomials of degree $k$ into polynomials of degree $k,$ for all $k.$ This
observation is no longer true for more general versions of the Wright-Fisher
operator with arbitrary variable coefficients, such as generalized Kimura
operators. Our analysis is based on understanding the regularity properties of
the Dirichlet heat kernel, which follow from the fact that it is a weak
solution to the forward Kolmogorov equation. In addition, we study the
properties of the hitting distributions (caloric measure) and we prove that
they satisfy the doubling property. This is the extension to the degenerate
framework of generalized Kimura operators of the doubling property satisfied by
the caloric measure of parabolic equations defined by strictly elliptic
operators, \cite[Theorem 1.1]{Safonov_Yuan_1999}, \cite[Theorem
2.4]{Fabes_Garofalo_Salsa_1986}.

In~\cite{Chen_Stroock_2010}  Chen and Stroock study the Dirichlet heat kernel, its small time asymptotics, and its boundary behavior for the one-dimensional Wright-Fisher operator,
$$
L u = x(1-x) u_{xx},\quad\forall\, x\in (0,1).
$$
Their method of the proof is based on series expansions of the fundamental solutions of related model operators.

Hofrichter, Tran, and Jost employ a hierarchical scheme in \cite{Hofrichter_Tran_Jost_2014a, Hofrichter_Tran_Jost_2014c} to study the existence and uniqueness of solutions to the backward parabolic equation defined by the multi-dimensional Wright-Fisher operator,
\begin{equation}
\label{eq:WF_operator_zw}
L_{\hbox{\tiny{WF}}} u 
= \frac{1}{2}\sum_{i,j=1}^n\left(\delta_{ij}x_i-x_ix_j\right)u_{x_ix_j},\quad\forall\, x\in\Sigma_n. 
\end{equation}
This method is extended in \cite{Hofrichter_Tran_Jost_2014b} to the forward
Wright-Fisher equation, but where the adjoint of the Wright-Fisher operator in
\eqref{eq:WF_operator_zw} is computed with respect to the standard
$L^2(\Sigma_n; dx)$ space, as opposed to the weighted $L^2$-space that we use
in our article, $L^2(\Sigma_n;d\mu)$, where the measure $d\mu$ is defined in
\eqref{eq:Weight} where $B_i=0$, for all $i$.

It is also very interesting to compare the structure of the heat kernel given
in equation~\eqref{eq:Tran_prob} with that given
in~\cite{Epstein_Mazzeo_2016}, where it is assumed that the weights are
strictly positive. In this case the kernel consists of a single term of the
form $k(t,p,q)d\mu_P(q).$ The kernel function, $k$ belongs to
$C^{\infty}((0,\infty)\times P\times (P\setminus \pa P)$) and, for positive
times, is uniformly bounded on $P\times P.$ It satisfies estimates that are
quite similar to standard Gaussian estimates, provided that the distance is
measured with respect to the metric defined by the principal symbol of $L$ 
and the Lebesgue measure is replaced by the weighted measure \eqref{eq:Weight}. 
In this case the associated stochastic process has a unique stationary measure,
and the paths of the process wander forever, almost surely, in the interior of
$P.$ As is shown in~\cite{Epstein_Mazzeo_annmathstudies}, a Kimura diffusion
that is tangent to part of $\pa P,$ and meets the boundary cleanly, may have a
multiplicity of stationary measures: there is a stationary measure for each
terminal component of the boundary of $P.$ A component, $\Sigma,$ of $\pa P$ is
terminal if $L$ is tangent to $\Sigma,$ and the restricted operator
$L_{\Sigma}$ has strictly positive weights. As the analysis in this paper
shows, a path of the associated process will almost surely reach a terminal boundary
component in finite time, where it then remains forever.

\subsection{Outline of the article}
\label{sec:Outline}
We divide \S \ref{sec:Martingale_problem} into two parts. In \S \ref{subsec:Martingale_problem}, we begin by proving the well-posedness of the martingale problem associated to the generalized Kimura operator in Definition \ref{defn:Martingale_problem}. We then prove in \S\ref{subsec:Tangent_boundary} that the Kimura diffusion defined by the martingale problem is absorbed on tangent boundary components of the compact manifold with corners. 

In \S \ref{sec:Dirichlet_heat_kernel}, we introduce the framework that allows
us to prove the existence of the Dirichlet heat kernel for the parabolic
problem with homogeneous Dirichlet boundary conditions on the tangent
components of the boundary of the manifold, $\partial^T P$. We give a
variational formulation to the parabolic problem defined by the generalized
Kimura operator, in \S \ref{sec:Bilinear_form}, by defining a suitable bilinear
form. The homogeneous Dirichlet boundary condition is embedded in the choice of
the weighted Sobolev space with respect to which the bilinear form is
continuous and satisfies the G\r{a}rding inequality. The weak solutions of the
parabolic problem with homogeneous Dirichlet boundary conditions on $\partial^T
P$ define a semigroup, which has an integral representation in terms of the
Dirichlet heat kernel, a fact that we prove in \S \ref{sec:Weak_solutions}. We
continue in \S \ref{sec:Adjoint} to prove that the Dirichlet heat kernel
satisfies the forward Kolmogorov equation defined by the adjoint of the
generalized Kimura operator and to determine its regularity properties along
$\partial^TP$. We conclude with the proof of Theorem
\ref{thm:Distribution_absorbed}.

In \S \ref{sec:Parabolic_Dirichlet_problem}, we prove the existence and
uniqueness of solutions to the parabolic problem with non-homogene-ous boundary
conditions on $\partial^T P$ and establish their connections with the semigroup
constructed in \S \ref{sec:Weak_solutions} and with Kimura diffusions.  This
connection is then used in \S \ref{sec:Prob} to give the proof of Theorem
\ref{thm:Hitting_distribution_hypers}. We are able to apply Theorem
\ref{thm:Hitting_distribution_hypers} to give the proof of Theorem
\ref{thm:Hitting_corner}, when we use in addition a Landis-type growth lemma
proved by M.V. Safonov, \cite{Safonov_2016}, and a regularity result concerning
the hitting probabilities of tangent boundary components. We prove the last two
results that we mentioned in Appendix \ref{sec:Appendix}. We provide the proof
of our main result, Theorem \ref{thm:Tran_prob}, and of Corollary
\ref{cor:Zero_times_spent_on_trans_boundary} at the end of \S \ref{sec:Prob}.

\subsection{Acknowledgment}
We would like to thank M.V.~Safonov for providing the proof of Lemma
A.1 in the fundamental 2d case and for valuable
discussions. We would also like to that the referee for
  his/her very careful reading of our paper and many suggestions for
  improvement.

\subsection{Notation}
%For $r>0$ and $p\in P$, we denote by $B_r(p)$ the open ball of radius $r$ with respect to the (incomplete) Riemannian metric on the manifold $P$ defined by the principle symbol of $L.$ 
For all $a,b\in\RR$, we denote
$$
a\wedge b := \min\{a,b\}\quad\hbox{ and }\quad a\vee b := \max\{a,b\}.
$$
We use $\NN_0$ to denote $\{0\}\cup\NN.$ Many of the estimates in our paper involve positive
  constants that depend on, among other things, the operator $L:$
  $C=C(L,\ldots)$. The dependence of these constants on the $L$ is
  through the positive lower bound $\beta_0,$ appearing in
  \eqref{eq:Cleanness}, the ellipticity constant, $\Lambda,$ and  the supremum
  norms of the coefficients of $L,$ and possibly their higher order
  derivatives.  

\section{The martingale problem associated to generalized Kimura operators}
\label{sec:Martingale_problem}

In this section we give the proof of the well-posedness of the martingale problem introduced in Definition \ref{defn:Martingale_problem} and we establish qualitative properties of the underlying diffusion process at the boundary of the compact manifold with corners. In 
\S \ref{subsec:Martingale_problem}, we prove Theorem \ref{thm:Wellposed_mart_problem} and Corollary \ref{cor:Strong_Markov}, while in 
\S \ref{subsec:Tangent_boundary} we study the boundary behavior of the underlying diffusion process on boundary hypersurfaces to which the operator is tangent.

\subsection{Well-posedness of the martingale problem}
\label{subsec:Martingale_problem}

Our goal in this section is to build a (strong) Markov process on the compact manifold $P$ with corners by solving the martingale problem in Definition \ref{defn:Martingale_problem}, which is analogous to \cite[Definition 1.3.1 (ii)]{Hsu_2002} and \cite[\S 4.1.1]{Stroock_manifolds}. Generally speaking, such problems can also be addressed by solving a suitable stochastic differential equation on $P$, as in \cite[Theorem V.1.1]{Ikeda_Watanabe2}. However, this alternative approach poses several difficulties, which we now describe to motivate our choice of the method to solve the martingale problem. In \cite[Theorem V.1.1]{Ikeda_Watanabe2}, the construction of the global solution to the stochastic differential equation is done by patching local solutions. Because the coefficients of the differential operator in \cite[Theorem V.1.1]{Ikeda_Watanabe2} are Lipschitz, the solutions in local coordinate charts are \emph{strong}, and so the Brownian motion appearing in the stochastic differential equation remains unchanged when we pass from one coordinate chart to another. In our case, the coefficients of the corresponding stochastic differential equation are only $1/2$-H\"older continuous, and so we only have local \emph{weak} solutions in coordinate charts of the manifold, which means that the Brownian motion appearing in the stochastic differential equation may change between coordinate charts. To circumvent this technical problem related to the definition of the Brownian motion, we solve instead the martingale problem associated to the generalized Kimura operator as introduced in Definition \ref{defn:Martingale_problem}.

Even though Definition \ref{defn:Martingale_problem} of our martingale problem is analogous to \cite[Definition 1.3.1 (ii)]{Hsu_2002} and \cite[\S 4.1.1]{Stroock_manifolds}, our method of establishing the existence of solutions is different. In \cite[\S 1.3]{Hsu_2002} and \cite[\S 4.1.1]{Stroock_manifolds}, the martingale problem is solved by embedding the manifold into $\RR^N$, for some positive integer $N$, as a closed submanifold, via Whitney's Embedding Theorem 
\cite{Whitney_1936}, \cite[Theorem 1.2.5]{Hsu_2002}. This embedding result is not available for compact manifolds with corners, and so we take a different approach. In \cite{Pop_2013a} we prove the well-posedness of the martingale problem associated to generalized Kimura operators defined on the unbounded set $\bar S_{n,m}$, as opposed to a compact manifold $P$ with corners. This indicates that a natural way to construct solutions to the martingale problem on the manifold is to cover $P$ by coordinate charts, such that each coordinate chart is homeomorphic to a bounded neighborhood of the origin in $\bar S_{n,m}$. We can apply \cite[Propositions 2.2 and 2.8]{Pop_2013a} to construct local solutions in coordinate charts. The process of patching the local probability measures into a global solution results in issues related to the measurability and adaptedness of the local solutions with respect to a common, global filtration defined on the manifold, which we address with the aid of a result of Stroock and Varadhan, \cite[Theorem 6.1.2]{Stroock_Varadhan}.

In the proof of Theorem \ref{thm:Wellposed_mart_problem}, we write the
generalized Kimura operator in a local system of coordinates as in
\eqref{eq:Operator}, but we also need to extend the operator from a
neighborhood of the origin in $\bar S_{n,m}$ to the whole space $\bar
S_{n,m}$. When we do this extension, we  assume that the coefficients of
the operator $L$ defined in \eqref{eq:Operator} at all points $z\in\bar
S_{n,m}$ satisfy the following assumptions. For any set of indices,
$I\subseteq\{1,\ldots,n\}$, we let
\begin{align}
\label{eq:M_I}
M_I:=\left\{z=(x,y) \in S_{n,m}: x_i \in (0,1)\hbox{ for all } i \in I,\hbox{ and } x_j \in (1,\infty)\hbox{ for all } j \in I^c\right\},
\end{align}
where we denote $I^c:=\{1,\ldots,n\}\backslash I$. We make the following assumptions on the coefficients of the operator $L$ in \eqref{eq:Operator} on $\bar S_{n,m}$, as opposed to $\barB^{\infty}_1$.

\begin{assump}
\label{assump:Coeff_extended}
The coefficients of the operator $L$ defined in \eqref{eq:Operator} satisfy:
\begin{enumerate}
\item[1.] The functions $a_{ij}(z)$, $b_i(z)$, $c_{il}(z)$, $d_{lk}(z)$, and $e_l(z)$ are smooth and bounded functions on $\bar S_{n,m}$, for all $1\leq i, j\leq n$ and $1\leq l, k \leq m$.
\item[2.] The drift coefficients $b_i(z)$ satisfy condition \eqref{eq:Nonnegative} with $\barB^{\infty}_1$ replaced by $\bar S_{n,m}$.
\item[3.] The \emph{strict ellipticity} condition holds: there is a positive constant, $\Lambda$, such that for all sets of indices, $I\subseteq \{1,\ldots,n\}$, for all $z \in \bar M_I$, $\xi\in\RR^n$ and $\eta\in\RR^m$, we have 
\begin{equation}
\label{eq:Uniform_ellipticity_extended}
\begin{aligned}
&  \sum_{i\in I} \xi_i^2
  +\sum_{i\in I^c} x_i\xi_i^2
  + \sum_{i,j\in I} a_{ij}(z)\xi_i\xi_j\\
& +\sum_{i\in I} \sum_{j\in I^c} x_j(a_{ij}(z)+a_{ji}(z))\xi_i\xi_j
  + \sum_{i,j\in I^c} x_ix_j a_{ij}(z)\xi_i\xi_j\\
& +\sum_{i\in I} \sum_{l=1}^m c_{il}(z)\xi_i\eta_l + \sum_{i\in I^c} \sum_{l=1}^m x_ic_{il}(z)\xi_i\eta_l
  +\sum_{l,k=1}^m d_{lk}(z)\eta_k\eta_l\\
&\quad\geq \Lambda\left(|\xi|^2+|\eta|^2\right).
\end{aligned}
\end{equation}
\end{enumerate}
\end{assump}

We notice that Assumption \ref{assump:Coeff_extended} implies that, when the operator $L$ is restricted to $\barB^{\infty}_1,$ it also satisfies Assumption \ref{assump:Coeff}. We can now give 

\begin{proof}[Proof of Theorem \ref{thm:Wellposed_mart_problem}]
We divide the proof into three steps. In Step \ref{step:Local_mart_problem}, we construct local solutions to the martingale problem in coordinate charts. In Step \ref{step:Global_mart_problem}, we apply \cite[Theorem 6.1.2]{Stroock_Varadhan} to patch the local solutions into a global solution of the martingale problem. In Step \ref{step:Uniqueness_mart_problem}, we prove the uniqueness of solutions to the martingale problem.

\setcounter{step}{0}
\begin{step}[Construction of local solutions]
\label{step:Local_mart_problem}
Let $p\in P$ and let $U\subseteq P$ be a relatively open neighborhood of $p$. Let $V\subset\bar S_{n,m}$ be a relatively open neighborhood of $z\in\bar S_{n,m}$, such that $\psi_p:U\rightarrow V$ is a homeomorphism and $\psi_p(p)=z$. We assume without loss of generality that in this local system of coordinates, $(U,V,\psi_p)$, the generalized Kimura operator takes the form \eqref{eq:Operator}, which we denote by $L^{\psi_p}$ to indicate the dependency on the local system of coordinates. We extend the coefficients of the operator $L^{\psi_p}$ from $V$ to $\bar S_{n,m}$, so that $L^{\psi_p}$ satisfies Assumption \ref{assump:Coeff_extended}. By abuse of notation, but to keep the notation simple, we denote the extended operator by same symbol $L^{\psi_p}$.

Assumption \ref{assump:Coeff_extended} allows us to apply \cite[Propositions 2.2 and 2.4]{Pop_2013a} to obtain that there is a unique solution to the martingale problem associated to the operator $L^{\psi_p}$ on $\bar S_{n,m}$, with initial condition $z$, where we recall that $z=\psi_p(p)$. We denote by $\PP^z_{\psi_p}$ this solution. We let $\Theta_V$ denote the first exit time from the set 
$V$ of the canonical process defined in  $\bar S_{n,m}$, which we denote by $\{Z(t)\}_{t\geq 0}$. We see that we can write $Z(t)=\psi_p(\omega(t))$, for all $t \leq \Theta_{V}$ and for all 
$\omega\in C([0,\infty);P)$. Let $\tau_U$ denote the first exit time from the
set $U$ of the canonical process on $P$, $\{\omega(t)\}_{t\geq 0}$. We define a
probability measure, 
\begin{equation}
\label{eq:Local_mart_problem_0}
\begin{aligned}
\bar \QQ^p_{\psi_p}(\omega(t_i\wedge\tau_U) \in B_i,\, 1\leq i\leq k) 
:= \PP^z_{\psi_p}(Z(t_i\wedge\Theta_V) \in \psi_p(B_i),\, 1\leq i\leq k),
\\
\bar \QQ^p_{\psi_p}(\omega(t_i) \in B_i,\, i=1,2,\ldots,k) 
:= \bar \QQ^p_{\psi_p}(\omega(t_i\wedge\tau_U) \in B_i,\, 1\leq i\leq k),
\end{aligned}
\end{equation}
for all $0\leq t_1<t_2<\ldots<t_k$, $B_i\in \cB(P)$, for all $1\leq i\leq k$, and $k\in\NN$. Because $\PP^z_{\psi_p}$ is a solution to the martingale problem defined by the operator $L^{\psi_p}$ on $\bar S_{n,m}$, it follows from definition \eqref{eq:Local_mart_problem_0} that $\bar\QQ^p_{\psi_p}$ is a solution to the stopped martingale problem in Definition \ref{defn:Martingale_problem} up to time $\tau_{U}$. That is, for all test functions $f\in C^{\infty}([0,\infty)\times P)$, we have that
$$
M^{f,U}_t:=f(t,\omega(t\wedge\tau_{U})) - f(0,\omega(0))
-\int_0^{t\wedge\tau_{U}} \left(f_s(s,\omega(s))+L^{\psi_p}f(s,\omega(s))\right)\, ds,
\quad\forall\, t\geq 0,
$$
is a martingale with respect to $\bar\QQ^p_{\psi_p}$.

We next show that the construction of the probability measure $\bar\QQ^p_{\psi_p}$ in \eqref{eq:Local_mart_problem_0} is coordinate-invariant, as we describe in the sequel. Let $\psi'_p:U'\rightarrow V'$ be another system of coordinates around $p$ and denote 
$z':=\psi'_p(p)$. Then, the map 
$$
\varphi:=\psi'_p\circ\psi^{-1}_p:\psi_p(U\cap U')\rightarrow \psi'_p(U\cap U')
$$
is a diffeomorphism and the change of coordinates $u(w) = v(w')$ and $w'=\varphi(w)$ gives us that
\begin{equation}
\label{eq:Change_of_coordinates}
L^{\psi_p} u(w) = L^{\psi'_p}v(w'),\quad\forall\, w\in \psi_p(U\cap U').
\end{equation}
Applying again \cite[Propositions 2.2 and 2.4]{Pop_2013a}, we let $\PP^{z'}_{\psi'_p}$ be the unique solution to the martingale problem associated to the operator $L^{\psi'_p}$ on $\bar S_{n,m}$ with initial condition $z'$. We let $\Theta_{\psi'_p(U\cap U')}$ denote the first exit time from the set $\psi'_p(U\cap U')$ of the canonical process, $\{Z'(t)\}_{t\geq 0}$, defined on $\bar S_{n,m}$ and with initial condition $z'$. Combining \eqref{eq:Change_of_coordinates} with the property that the solutions to the martingale problem associated to Kimura operators $L^{\psi_p}$ and $L^{\psi'_p}$ are unique, by 
\cite[Proposition 2.4]{Pop_2013a}, it follows that
$$
\PP^z_{\psi_p}(Z(t_i\wedge \Theta_{\psi_p(U\cap U')}) \in B_i,\, 1\leq i\leq k) 
= \PP^{z'}_{\psi'_p}(Z'(t_i\wedge \Theta_{\psi'_p(U\cap U')}) \in \varphi(B_i),\, 1\leq i\leq k),
$$
for all $0\leq t_1<t_2<\ldots<t_k$, $B_i\in \cB(\bar S_{n,m})$, $1\leq i\leq k$, and $k\in\NN$. Defining now 
$\bar \QQ^p_{\psi'_p}$ analogously to the identities in~\eqref{eq:Local_mart_problem_0}, but replacing $\psi_p$ by $\psi'_p$, 
$U$ by $U'$, $\{Z(t)\}_{t\geq 0}$ by $\{Z'(t)\}_{t\geq 0}$, and $\tau_{U}$ by $\tau'_{U'}$, the preceding identity gives us that
\begin{equation}
\label{eq:Invariance_under_change_of_coordinates}
\bar \QQ^p_{\psi_p}(\omega(t_i\wedge\tau_{U}\wedge\tau'_{U'}) \in B_i,\, 1\leq i\leq k) 
= \bar \QQ^p_{\psi'_p}(\omega(t_i\wedge\tau_{U}\wedge\tau'_{U'}) \in B_i,\, 1\leq i\leq k), 
\end{equation}
for all $0\leq t_1<t_2<\ldots<t_k$, $B_i\in \cB(P)$, $1\leq i\leq k$, and $k\in\NN$. Thus, indeed definition \eqref{eq:Local_mart_problem_0} of the local probability measure $\bar \QQ^p_{\psi_p}$ is coordinate-invariant.

For the remaining part of the proof, we fix a finite atlas on the compact manifold $P$, which we denote by 
$\cA:=\{(U_{\alpha}, V_{\alpha}, \psi_{\alpha})\}_{\alpha}$. We use the atlas $\cA$ to construct local solutions to the martingale problem that cover the manifold. For a point $p\in P$, we let $U_{\alpha_1},\ldots,U_{\alpha_l}$ be the relatively open sets in the atlas $\cA$ containing the point $p$. We let
\begin{equation}
\label{eq:tau_p}
\tau_p:=\min\{\tau_{U_{\alpha_i}}:\, 1\leq i\leq l\}
\end{equation}
be the first time that the canonical process on $P$, $\{\omega(t)\}_{t\geq 0}$, exits the neighborhood 
$\bigcap\{U_{\alpha_i}: 1\leq i\leq l\}$. Using the coordinate-invariance property \eqref{eq:Invariance_under_change_of_coordinates}, we can define a probability measure, $\bar\QQ^p$, on $P$ by letting
\begin{equation}
\label{eq:Local_mart_problem}
\begin{aligned}
\bar \QQ^p(\omega(t_i\wedge\tau_p) \in B_i,\, 1\leq i\leq k) := \bar\QQ^p_{\psi_{\alpha_j}}(\omega(t_i\wedge\tau_p) \in B_i,\, 1\leq i\leq k),
\\
\bar \QQ^p(\omega(t_i) \in B_i,\, 1\leq i\leq k) := \bar \QQ^p(\omega(t_i\wedge\tau_p) \in B_i,\, 1\leq i\leq k),
\end{aligned}
\end{equation}
for all $0\leq t_1<t_2<\ldots<t_k$, $B_i\in \cB(P)$, $1\leq i\leq k$, and $k\in\NN$, where $1\leq j\leq l$ is arbitrarily chosen.
The measure $\bar\QQ^p$ is a local solution to the stopped martingale problem for $L$ up to time $\tau_p$, i.e. for all test functions, $f\in C^{\infty}([0,\infty)\times P)$, we have that
$$
M^{f,\cA}_t:=f(t,\omega(t\wedge\tau_p)) - f(0,\omega(0))
-\int_0^{t\wedge\tau_p} \left(f_s(s,\omega(s))+L^{\psi_{\alpha_j}}f(s,\omega(s))\right)\, ds,\quad\forall\, t\geq 0,
$$
is a martingale with respect to $\bar\QQ^p$, for all $1\leq j\leq l$. This completes the construction of the local solutions to the martingale problem.
\end{step}

\begin{step}[Construction of global solutions]
\label{step:Global_mart_problem}
Let $p_0\in P$. We now construct a global solution to the martingale problem starting from $p_0$ using an inductive argument, which consists in building a sequence of probability measures, $\QQ^n$, that are solutions to the martingale problem up to a stopping time 
$\zeta_n$, such that $\zeta_n$ tends to $\infty$, as $n\rightarrow\infty$.

\setcounter{case}{0}
\begin{case}[Base case]
\label{case:Base_case_mart_problem}
Let $\zeta_1:=\tau_{p_0}\wedge 1$, where the stopping time $\tau_{p_0}$ is defined as in \eqref{eq:tau_p} with $p$ replaced by $p_0$.  We denote $\QQ^1:=\bar\QQ^{p_0}$, where $\bar\QQ^{p_0}$ is the probability measure built in Step \ref{step:Local_mart_problem} with $p$ replaced by $p_0$. To construct $\QQ^2$, for all $\omega\in C([0,\infty);P)$, we set $\QQ^1_{\omega}:=\bar\QQ^{\omega(\zeta_1(\omega))}$ be the probability measure constructed in Step \ref{step:Local_mart_problem}, where we choose $p:=\omega(\zeta_1(\omega))$. We denote $\zeta_2:=(\zeta_1+\tau_{\omega(\zeta_1(\omega))})\wedge 2$. Our goal is to apply \cite[Theorem 6.1.2]{Stroock_Varadhan} to the probability measures $\QQ^1$ and $\QQ^1_{\omega}$ to construct a new probability measure, 
$$
\QQ^2:=\QQ^1 \otimes_{\zeta_1} \QQ^1_{\omega},
$$
and establish that $\QQ^2$ is a solution to the martingale problem up to time $\zeta_2$. We begin by checking that the hypotheses of \cite[Theorem 6.1.2]{Stroock_Varadhan} hold in our setting. Because \cite[Hypothesis (ii) of Theorem 6.1.2]{Stroock_Varadhan} is clearly satisfied by the probability measure $\QQ^1_{\omega}$, it remains to verify that \cite[Hypothesis (i) of Theorem 6.1.2]{Stroock_Varadhan} holds. That is, we need to establish that, for all measurable sets $A\in \cF$, we have that the map $C([0,\infty);P) \ni\omega\mapsto\QQ^1_{\omega}(A)$ is $\cF_{\zeta_1}$-measurable. This is equivalent to the property that, for all Borel measurable sets, $B\in \cB([0,1])$, we have that 
\begin{equation}
\label{eq:SV_meas}
\{\omega\in C([0,\infty);P):\, \QQ^1_{\omega}(A) \in B\}\in \cF_{\zeta_1}.
\end{equation}
Notice that we have the identity:
\begin{equation}
\label{eq:Set_decomposed_in_coordinate_charts_0}
\begin{aligned}
&\left\{\omega\in C([0,\infty);P):\, \QQ^1_{\omega}(A) \in B\right\} \\
&\quad\quad = \bigcup_{\alpha} 
\left\{\omega\in C([0,\infty);P):\,
\omega(\zeta_1(\omega))\in U_{\alpha}\,\hbox{  and  }\, \bar \QQ^{\omega(\zeta_1(\omega))}(A) \in B\right\}.
\end{aligned}
\end{equation}
Without loss of generality, we can assume that we can represent
$$
A=\bigcap\limits_{i=1}^k\{\omega(t_i) \in B_i\},
$$
for all $0\leq t_1<t_2<\ldots<t_k$, $B_i\in \cB(P)$, $1\leq i\leq k$, and
$k\in\NN$. For all systems of coordinates $(U_{\alpha},
V_{\alpha},\psi_{\alpha})$, let $L^{\psi_{\alpha}}$ be an extension of the
operator from $V_{\alpha}$ to $\bar S_{n_{\alpha},m_{\alpha}}$, as constructed
in Step \ref{step:Local_mart_problem}. For all $z\in \bar
S_{n_{\alpha},m_{\alpha}}$, we let $\PP^z_{\psi_{\alpha}}$ be the unique
solution to the martingale problem associated to $L^{\psi_{\alpha}}$ with
initial condition $z$, established in \cite[Propositions 2.2 and
2.4]{Pop_2013a}. We denote $\{Z_{\alpha}(t)\}_{t\geq 0}$ the canonical process
on $\bar S_{n_{\alpha},m_{\alpha}}$. We introduce the sets
\begin{align*}
M_{\alpha} &:= \bigcap\limits_{i=1}^k\{\xi\in C([0,\infty);\bar S_{n_{\alpha},m_{\alpha}}):\, 
Z_{\alpha}(t_i\wedge \nu^{\alpha}_{\xi(0)}(\xi),\xi) \in \psi_{\alpha}(B_i)\},\\
S_{\alpha} &:= \left\{\omega\in C([0,\infty);P):\,
\omega(\zeta_1(\omega))\in U_{\alpha},\, 
\PP^{\psi_{\alpha}(\omega(\zeta_1(\omega)))}_{\psi_{\alpha}}
(M_{\alpha}) \in B\right\},
\end{align*}
where $\nu^{\alpha}_z$ is the first exit time of the process $\{Z_{\alpha}(t)\}_{t\geq 0}$ from the set $\psi_{\alpha}(\bigcap\{U_{\beta}: \psi^{-1}_{\alpha}(z) \in U_{\beta}\})$. Definition \eqref{eq:Local_mart_problem} of the probability measure $\bar\QQ^p$ and identity \eqref{eq:Set_decomposed_in_coordinate_charts_0} give us that
\begin{equation}
\label{eq:Set_decomposed_in_coordinate_charts}
\left\{\omega\in C([0,\infty);P):\, \QQ^1_{\omega}(A) \in B\right\} = \bigcup_{\alpha} S_{\alpha}.
\end{equation}
From \cite[Propositions 2.2 and 2.4]{Pop_2013a} it follows that the martingale problem for $L^{\psi_{\alpha}}$ on $\bar S_{n_{\alpha},m_{\alpha}}$ is well-posed, and so by \cite[Exercise 6.7.4 and Theorem 6.2.2]{Stroock_Varadhan}, it follows that $z\mapsto\PP^z_{\psi_{\alpha}}(M_{\alpha})$ is a Borel measurable function, for all indices $\alpha$, which implies that the set
$$
T_{\alpha}:=\left\{z\in\bar S_{n_{\alpha},m_{\alpha}}:\, \PP^z_{\psi_{\alpha}}(M_{\alpha})\in B\right\}
$$ 
is Borel measurable. Using the fact that we can write $S_{\alpha}$ as
$$
S_{\alpha}=\left\{\omega\in C([0,\infty);P):\,\omega(\zeta_1(\omega)) \in U_{\alpha}\cap \psi_{\alpha}^{-1}(T_{\alpha})\right\},
$$
we conclude that $S_{\alpha}$ is $\cF_{\zeta_1}$-measurable. Combining the preceding fact with identity \eqref{eq:Set_decomposed_in_coordinate_charts}, we obtain that property \eqref{eq:SV_meas} holds. Thus, \cite[Hypothesis (i) of Theorem 6.1.2]{Stroock_Varadhan} is satisfied, and we can now conclude that the measure $\QQ^2:=\QQ^1\otimes_{\zeta_1}\QQ^1_{\omega}$ is indeed a solution to the stopped martingale problem for the generalized Kimura operator up to the stopping time $\zeta_2$.
\end{case}

\begin{case}[Inductive step]
\label{case:Inductive_step_mart_problem}
Inductively, we define the sequence of stopping times,
\begin{equation}
\label{eq:zeta_n}
\zeta_n:=\left(\zeta_{n-1}+\tau_{\omega(\zeta_{n-1}(\omega))}\right)\wedge n,\quad\forall\, n\geq 2. 
\end{equation}
We construct a sequence of probability measures, 
$$
\QQ^n:=\QQ^{n-1} \otimes_{\zeta_{n-1}} \QQ^{n-1}_{\omega},
$$
where analogously to Step \ref{step:Local_mart_problem}, we let $\QQ^{n-1}_{\omega}:=\bar\QQ^{\omega(\zeta_{n-1}(\omega))}$, for all 
$\omega\in C([0,\infty);P)$. Employing the same argument as in the proof of Case \ref{case:Base_case_mart_problem}, we obtain that the probability measure $\QQ^n$ is a solution to the stopped martingale problem for the generalized Kimura operator up to time $\zeta_n$. 
\end{case}

Letting 
\begin{equation}
\label{eq:limit_zeta_n}
\zeta:=\lim_{n\rightarrow\infty}\zeta_n,
\end{equation}
we define the probability measure $\QQ^{p_0}$ by setting
\begin{align*}
\QQ^{p_0} (\omega(t_i\wedge\zeta)\in B_i,\, 1\leq i \leq k) := 
\lim_{n\rightarrow\infty}\QQ^n\left(\omega(t_i\wedge\zeta_n)\in B_i,\, 1\leq i \leq k\right),
\end{align*}
for all $0\leq t_1<t_2<\ldots<t_k$, $B_i\in \cB(P)$, $1\leq i \leq k$, and $k\in\NN$. The construction of the measure $\QQ^{p_0}$ is concluded once we show that:

\begin{claim}
\label{claim:limt_zeta_n}
We have that $\{\zeta = \infty\} = C([0,\infty);P)$.
\end{claim}

\begin{proof}
We assume by contradiction that there is a path $\omega\in C([0,\infty);P)$ with the property that $\zeta(\omega) < \infty$. From identities \eqref{eq:limit_zeta_n} and \eqref{eq:zeta_n}, it follows that
\begin{equation}
\label{eq:limit_tau_omega_zeta}
\tau_{\omega(\zeta_n(\omega))} \rightarrow 0,\quad\hbox{ as } n\rightarrow\infty.
\end{equation}
Let $U =U(\omega) = \bigcap\{U_{\alpha} : \omega(\zeta(\omega)) \in U_{\alpha}\}$. Then $U$ is a relatively open set in $P$ and because $\zeta_n(\omega)$ converges to $\zeta(\omega)$, as $n\rightarrow\infty$, we have that there is a rank $N=N(\omega)$ such that $\omega(\zeta_n(\omega)) \in U$, for all $n > N$. Using the definition of the stopping time $\tau_{\omega(\zeta_n(\omega))}$ in \eqref{eq:tau_p}, we know that $\tau_{\omega(\zeta_n(\omega))}$ is the first time the canonical process exits the set $U$, when started from the point $\omega(\zeta_n(\omega))$. Because the sequence of stopping times $\{\zeta_n\}_{n\in\NN}$ is non-decreasing, the canonical process started from $\omega(\zeta_n(\omega))$ will pass through all the points $\omega(\zeta_m(\omega))$, with $m > n$, and also through the point $\omega(\zeta(\omega))$ before exiting $U$. Thus, we have that
$$
\tau_{\omega(\zeta_n(\omega))} \geq \tau_{\omega(\zeta_{n+1}(\omega))} \geq \tau_{\omega(\zeta_{n+2}(\omega))}
\geq \ldots \geq \tau_{\omega(\zeta(\omega))}, \quad\forall\, n > N.
$$
The preceding property together with \eqref{eq:limit_tau_omega_zeta} give us that
\begin{equation}
\label{eq:limit_tau_omega_zeta_1}
\tau_{\omega(\zeta(\omega))} = 0.
\end{equation}
We also know that $\omega(\zeta(\omega))$ is a point of a relatively open set, and so using the continuity of the paths of the canonical process, it follows that the first time the canonical process started at $\omega(\zeta(\omega))$ exits the set $U$ is positive, i.e.
\begin{equation}
\label{eq:limit_tau_omega_zeta_2}
\tau_{\omega(\zeta(\omega))} > 0.
\end{equation}
Relations \eqref{eq:limit_tau_omega_zeta_1} and \eqref{eq:limit_tau_omega_zeta_2} are in contradiction, which implies that our assumption that $\zeta(\omega) < \infty$ is false. Because $\omega\in C([0,\infty);P)$ was arbitrarily chosen, we see that the conclusion of Claim \ref{claim:limt_zeta_n} is true.
\end{proof}

Applying Claim \ref{claim:limt_zeta_n} we conclude that $\QQ^{p_0}$ is a solution to the martingale problem associated to the generalized Kimura operator with initial condition $p_0$, because each $\QQ^n$ is a solution to the martingale problem with initial condition $p_0$ up to the stopping time $\zeta_n$ and $\zeta_n$ converges to $\infty$, as $n$ tends to $\infty$.
\end{step}

\begin{step}[Uniqueness]
\label{step:Uniqueness_mart_problem}
Assume by contradiction that the martingale problem associated to the operator $L$ on $P$ has two solutions, $\QQ^{p_0,1}$ and $\QQ^{p_0,2}$, both with initial condition $p_0$. For all $T>0$ and $f\in C^{\infty}(P)$, let $u\in C^{\infty}([0,T]\times P)$ be the unique solution to the parabolic equation $u_t-Lu=0$ on $(0,\infty)\times P$ and $u(0)=f$, given by \cite[Theorem 10.0.2]{Epstein_Mazzeo_annmathstudies}. Plugging in $\varphi(s,p):=u(T-s,p)$ in identity \eqref{eq:Mart_problem_martingales} and taking expectations under the measures $\QQ^{p_0,1}$ and $\QQ^{p_0,2}$, we obtain that
$$
\EE_{\QQ^{p_0,1}}\left[f(\omega(T))\right] = \EE_{\QQ^{p_0,2}}\left[f(\omega(T))\right].
$$
Thus, the probability measures $\QQ^{p_0,1}$ and $\QQ^{p_0,2}$ have the same marginal distributions. Combining the matching of the marginal distributions with the analogue of \cite[Proposition 5.4.27]{KaratzasShreve1998} adapted to our framework, we obtain that the probability measures $\QQ^{p_0,1}$ and $\QQ^{p_0,2}$ coincide.
\end{step}

This completes the proof.
\end{proof}

We conclude this section with

\begin{proof}[Proof of Corollary \ref{cor:Strong_Markov}]
The proof of Corollary \ref{cor:Strong_Markov} can be done using the same argument as the one employed in the proof of 
\cite[Theorem 5.4.20]{KaratzasShreve1998} and we apply the well-posedness of the martingale problem in Theorem \ref{thm:Wellposed_mart_problem}. We omit the detailed proof for brevity.
\end{proof}

\subsection{Diffusion behavior along boundary components to which the operator is tangent}
\label{subsec:Tangent_boundary}

In this section we describe the behavior of the generalized Kimura process when
it hits a tangent boundary component.  We begin by stating a property of the
restriction operator that we often use:

\begin{rmk}[Equations for the restriction operator]
\label{rmk:Sato_equation}
Let $\Sigma$ be a tangent boundary component of the manifold $P$. Let $u\in C^{\infty}([0,\infty)\times P)$ be a solution to the parabolic problem,
\begin{align*}
\left\{\begin{array}{rl}
u_t-Lu= g & \hbox{ on } (0,\infty)\times P,\\ 
u(0) = f & \hbox{ on } P. 
\end{array} \right.
\end{align*}
Then the solution $u$ restricted to the boundary component $\Sigma$, $u\restrictedto_{\Sigma}$, belongs to 
$C^{\infty}([0,\infty)\times \Sigma)$ and is a solution to the parabolic problem defined by the operator $L_{\Sigma}$:
\begin{align*}
\left\{\begin{array}{rl}
(\partial_t-L_{\Sigma}) u \restrictedto_{\Sigma} = g \restrictedto_{\Sigma} & \hbox{ on } (0,\infty)\times \Sigma,\\ 
u\restrictedto_{\Sigma}(0) = f\restrictedto_{\Sigma} & \hbox{ on } \Sigma. 
\end{array} \right.
\end{align*}
\end{rmk}

We can now state:

\begin{lem}[Hitting of boundary components to which the operator is tangent]
\label{lem:Hitting_tangent_boundary}
Suppose that the generalized Kimura operator satisfies the standard assumptions. Let $\Sigma$ be a tangent boundary component of the manifold with corners $P$. Let
\begin{equation}
\label{eq:Hitting_time_Sigma}
\tau_{\Sigma} := \inf\{t \geq 0: \omega(t)\in\Sigma\}
\end{equation}
be the first hitting time of the boundary component $\Sigma$, and let
\begin{equation}
\label{eq:Probability_hit_Sigma}
p_{\Sigma}(t,p) := \QQ^p(\tau_{\Sigma} \leq t),\quad\forall\, (t,p) \in [0,\infty)\times P. 
\end{equation}
The hitting probability $p_{\Sigma}(t,p)$ has the following properties:
\begin{enumerate}
\item[(i)] $p_{\Sigma} \in C^{\infty}((0,\infty)\times P) \cap C^{\infty}([0,\infty)\times (P\backslash\Sigma))$.
\item[(ii)] $p_{\Sigma} (t,p)>0$, for all $(t,p) \in (0,\infty)\times P\backslash \partial^T P$.
\item[(iii)] $p_{\Sigma} (t,p)=0$, for all $(t,p) \in (0,\infty)\times H_i$ such that $i\in I^T$ and $H_i\cap\Sigma = \emptyset$.
\end{enumerate}
\end{lem}

In the next lemma, we prove that the process is absorbed when it hits a tangent boundary component.

\begin{lem}[Absorption on boundary components to which the operator is tangent]
\label{lem:Absorption_tangent_boundary}
Suppose that the generalized Kimura operator satisfies the standard assumptions. Let $\Sigma$ be a tangent boundary component of the manifold $P,$ then, we have that
\begin{equation}
\label{eq:Absorption_tangent_boundary}
\QQ^p(\omega(t) \in \Sigma,\,\forall\, t\geq\tau_{\Sigma}) = 1,\quad\forall\,p\in P,
\end{equation}
where the stopping time $\tau_{\Sigma}$ is defined in \eqref{eq:Hitting_time_Sigma}.
\end{lem}

We begin with 

\begin{proof}[Proof of Lemma \ref{lem:Hitting_tangent_boundary}]
We split the proof of the properties of the hitting probabilities $p_{\Sigma}$ into several steps. 

\setcounter{step}{0}
\begin{step}[Proof of property (i)]
\label{step:Hitting_prob_regularity}
For all $l\in\NN$, let $\varphi_l\in C^{\infty}([0,\infty))$ be such that 
$$
\varphi_l(s)=1,\quad\hbox{ for } 0\leq s\leq 1/l,
\quad\hbox{ and }\quad
\varphi_l(s)=0,\quad\hbox{ for } s\geq 2/l,
$$
and we set $\Phi_l(p):=\prod_{j=1}^k\varphi_l(\rho_{i_j}(p))\chi_{\Sigma}(p)$,
where we recall the definition of $H_{i_j}$ in \eqref{eq:Sigma}. The function
$\rho_i(p)$ denotes the Riemannian distance on the manifold $P$ from the point
$p$ to the boundary hypersurface $H_i,$ and $\chi_{\Sigma}$ is a non-negative
smooth function on $P$ that is 1 on $\Sigma$ and zero on the other components of the
intersection in \eqref{eq:Sigma}. Let $u_l\in C^{\infty}([0,\infty)\times P)$
be the unique solution, given by \cite[Theorem 1.5]{Pop_2013b}, to the
initial-value problem:
\begin{align*}
\left\{\begin{array}{rl}
u_t-Lu=0 & \hbox{ on } (0,\infty)\times P,\\ 
u(0) = \Phi_l& \hbox{ on } P. 
\end{array} \right.
\end{align*}
Letting $\varphi(s,p):=u_l(t-s,p)$, for all $(s,p)\in [0,t]\times P$, in \eqref{eq:Mart_problem_martingales} it follows from \eqref{eq:Mart_problem_martingales}, by taking expectation under the measure $\QQ^p$, that
$$
u_l(t,p) = \EE_{\QQ^p} \left[\Phi_l(\omega(t))\right],
$$
which, using definition \eqref{eq:Hitting_time_Sigma}, can be written as
\begin{equation}
\label{eq:Hitting_probability_approximation}
u_l(t,p) = \EE_{\QQ^p}\left[\Phi_l(\omega(t))\mathbf{1}_{\{\tau_{\Sigma}\leq t\}}\right]
+ \EE_{\QQ^p}\left[\Phi_l(\omega(t))\mathbf{1}_{\{\tau_{\Sigma} > t\}}\right].
\end{equation}
The maximum principle \cite[Corollary 5.2]{Pop_2013b} implies that
$$
\sup_{(t,p)\in [0,\infty)\times P} |u_l(t,p)| \leq 1,\quad\forall\, l\in\NN,
$$
and the supremum estimates \cite[Theorem 1.2]{Pop_2013b} show that, for all $0<T_0<T$ and for all $\kappa\in\NN$, there is a positive constant, $C=C(\kappa,T_0,T)$, such that
$$
\|u_l\|_{C^{\kappa}([T_0,T]\times P)} \leq C,\quad\forall\, l, \kappa\in \NN.
$$
Applying the Arzel\`a-Ascoli Theorem, the sequence of functions
$\{u_l\}_{l\in\NN}$ contains a subsequence that converges in
$C^{\kappa}([T_0,T]\times P)$, for all $\kappa\in\NN$, to a function $u\in
C^{\infty}((0,\infty)\times P)$. Notice that for all points $p\notin \Sigma$,
there is a positive constant $r$ and an integer $l_0\in\NN$ such that $\Phi_l
=0$ on $B_r(p)$, for all $l\geq l_0$. Thus, the local a priori Schauder
estimates \cite[Theorem 1.3]{Pop_2013b} imply that, for all $T>0$ and for all
$\kappa\in\NN$, there is a positive constant, $C=C(\kappa,r,T_0,T)$, such that
$$
\|u_l\|_{C^{\kappa}([0,T]\times B_r(p))} \leq C,\quad\forall\, l,\kappa\in \NN.
$$
We conclude from here that $u$ belongs to $C^{\infty}([0,\infty)\times(P\backslash \Sigma))$. 

We next show that $u=p_{\Sigma}$. Letting $l$ tend to $\infty$ in equation \eqref{eq:Hitting_probability_approximation} and using the fact that $\{\Phi_l\}_{l\geq 0}$ is a sequence of uniformly bounded functions that tend to $0$ on $P\backslash\Sigma$, it follows that
$$
\EE_{\QQ^p}\left[\Phi_l(\omega(t))\mathbf{1}_{\{\tau_{\Sigma} > t\}}\right] \rightarrow 0,\quad\hbox{ as } l\rightarrow\infty,
$$
and so, we obtain that
\begin{align*}
u(t,p) &= \lim_{l\rightarrow\infty} \EE_{\QQ^p}\left[\Phi_l(\omega(t))\mathbf{1}_{\{\tau_{\Sigma} \leq t\}}\right]\\
&= \lim_{l\rightarrow\infty} 
\EE_{\QQ^p}\left[\mathbf{1}_{\{\tau_{\Sigma} \leq t\}}\EE_{\QQ^p}\left[\Phi_l(\omega(t))\big{|} \cF_{\tau_{\Sigma}}\right]\right]\\
&= \lim_{l\rightarrow\infty} \EE_{\QQ^p}\left[\mathbf{1}_{\{\tau_{\Sigma} \leq t\}}
\EE_{\QQ^{\omega(\tau_{\Sigma}(\omega))}}\left[\Phi_l(\omega(t-\tau_{\Sigma}(\omega)))\right]\right],
\end{align*}
where in the second line we use the law of iterated expectations, and in the
last line of the preceding sequence of identities we use the strong Markov
property of the family of probability measure, $\{\QQ^p: p\in P\}$, established
in Corollary \ref{cor:Strong_Markov}. Using again equation
\eqref{eq:Hitting_probability_approximation}, we have that
\begin{equation}
\label{eq:Hitting_prob_approx}
u(t,p) = \lim_{l\rightarrow\infty} \EE_{\QQ^p}\left[\mathbf{1}_{\{\tau_{\Sigma} \leq t\}}
u_l(t-\tau_{\Sigma}(\omega)),\omega(\tau_{\Sigma}(\omega)))\right].
\end{equation}
From Remark \ref{rmk:Sato_equation}, we have that the functions 
$u_l\restriction_{[0,\infty)\times\Sigma}$ are smooth solutions to the initial-value problems,
\begin{align*}
\left\{\begin{array}{rl}
u_t-L_{\Sigma}u=0 & \hbox{ on } (0,\infty)\times \Sigma,\\ 
u(0) = 1& \hbox{ on } \Sigma. 
\end{array} \right.
\end{align*}
By the uniqueness of smooth solutions, established in \cite[Theorem 10.0.2]{Epstein_Mazzeo_annmathstudies}, it follows that 
$$
u_l\restriction_{[0,\infty)\times\Sigma}\equiv 1,\quad\forall\, l\in\NN,
$$
which implies that
$$
u_l(t,p) = \EE_{\QQ^p}\left[\Phi_l(\omega(t))\right] = 1,\quad\forall\, (t,p)\in [0,\infty)\times\Sigma,\quad\forall\, l\in \NN.
$$
Using this property in equation \eqref{eq:Hitting_prob_approx} and the definition of $\Phi_l$, it follows that
$$
u(t,p) = \EE_{\QQ^p}\left[\mathbf{1}_{\{\tau_{\Sigma} \leq t\}}\right] = \QQ^p(\tau_{\Sigma} \leq t).
$$
Since we established that $u$ belongs to $C^{\infty}((0,\infty)\times P))\cap C^{\infty}([0,\infty)\times (P\backslash\Sigma))$, the preceding equality  and \eqref{eq:Probability_hit_Sigma} imply that the hitting probability $p_{\Sigma}$ belongs to the same space of functions.
\end{step}

\begin{step}[Proof of property (ii)]
\label{step:Hitting_prob_positivity}
From the proof of Step \ref{step:Hitting_prob_regularity}, it follows that $p_{\Sigma}$ is a non-negative solution to equation 
$$
(\partial_t-L)p_{\Sigma} = 0\quad\hbox{ on } (0,\infty)\times P.
$$
Notice from Step \ref{step:Hitting_prob_regularity} that $p_{\Sigma}(t,p)=1$,
for all $(t,p)\in [0,\infty)\times\Sigma$. It remains to prove that
$p_{\Sigma}(t,p)$ is positive, for all $(t,p)\in
(0,\infty)\times(P\backslash \partial^T P)$. Assume by contradiction there is a
point $(t_0, p_0)\in (0,\infty)\times (P\backslash \partial^T P)$ such that
$p_{\Sigma}(t_0,p_0)=0$. Because $p_0\notin\pa^T P$ we can apply the Harnack
inequality \cite[Theorem 4.1]{Epstein_Mazzeo_2016}, \cite[Theorem
1.2]{Epstein_Pop_2013b} to conclude that there is a neighborhood of the point
$(t_0,p_0)$ where the function $p_{\Sigma}$ is identically $0$. Applying the
Harnack inequality to neighboring points, we conclude that $p_{\Sigma}$ is
identically $0$ on $(0,\infty)\times (P\backslash \partial^T P)$. Using the
continuity of the function $p_{\Sigma}$ on $(0,\infty)\times P$, we obtain in
particular that $p_{\Sigma} \equiv 0$ on $(0,\infty)\times\Sigma$, which
contradicts the proof in Step \ref{step:Hitting_prob_regularity}, where we
found that $p_{\Sigma} \equiv 1$, on $(0,\infty)\times\Sigma$. This concludes
the proof of property (ii).
\end{step}

\begin{step}[Proof of property (iii)]
Let $i\in I^T$ be such that $H_i\cap\Sigma=\emptyset$, where we recall the definition of $I^T$ in \eqref{eq:I_tangent}. From the definition of the function $\Phi_l$ in Step \ref{step:Hitting_prob_regularity}, it follows that there is $l_0\in\NN$ such that 
$\Phi_l \equiv 0$ on $H_i$, for all $l\geq l_0$. Because the operator $L$ is tangent to $H_i$, we obtain from Remark 
\ref{rmk:Sato_equation} that $u_l$ is a smooth solution to the parabolic equation $(\partial_t-L_{H_i}) u_l = 0$ on 
$(0,\infty)\times H_i$ with zero initial data, for all $l\geq l_0$. It follows that the functions $u_l\restriction_{[0,\infty)\times H_i}\equiv 0$ and letting $l$ tend to $\infty$, we obtain that the hitting distribution $p_{\Sigma}\restriction_{[0,\infty)\times H_i}\equiv 0$.
\end{step}
The proof of now complete.
\end{proof}

\begin{proof}[Proof of Lemma \ref{lem:Absorption_tangent_boundary}]
Because the paths of the canonical process are continuous and the boundary component $\Sigma$ is a closed set, we can write
$$
\{\omega(t)\in\Sigma,\,\forall\,t\geq \tau_{\Sigma}\} = \bigcap_{t \in \QQ_+}\{\omega(t)\in\Sigma,\,t\geq \tau_{\Sigma}\},
$$
and so, property \eqref{eq:Absorption_tangent_boundary} follows as soon as we prove that
\begin{equation}
\label{eq:Absorption_tangent_boundary_at_t}
\QQ^p(\omega(t)\in\Sigma,\ t\geq \tau_{\Sigma}) = 1,\quad\forall\, t\in\QQ_+,\quad\forall\, p\in P.
\end{equation}
In order to establish identity \eqref{eq:Absorption_tangent_boundary_at_t}, it is sufficient to show that, for all functions $\varphi\in C^{\infty}_c(P\backslash\Sigma)$, we have that
\begin{equation}
\label{eq:Absorption_tangent_boundary_at_t_test_function}
\EE_{\QQ^p}\left[\varphi(\omega(t)) \mathbf{1}_{\{t\geq \tau_{\Sigma}\}}\right] = 0.
\end{equation}
Let $u\in C^{\infty}([0,\infty)\times P)$ be the unique smooth solution to the equation $u_t-Lu=0$ on $(0,\infty)\times P$, with initial condition $u(0)=\varphi$, given by \cite[Theorem 10.0.2]{Epstein_Mazzeo_annmathstudies}. Because the operator $L$ is tangent to $\Sigma$ and $\varphi\restriction_{\Sigma}\equiv0$, it follows from Remark \ref{rmk:Sato_equation} that 
$u\equiv 0$ on $[0,\infty)\times\Sigma$. Using the test function $(s,p)\mapsto u(t-s,p)$ in equation \eqref{eq:Mart_problem_martingales}, it follows that
\begin{equation}
\label{eq:Evaluation_on_Sigma}
u(t,p) = \EE_{\QQ^p}\left[\varphi(\omega(t))\right] =0,\quad\forall\, (t,p)\in (0,\infty)\times\Sigma.
\end{equation}
Using the strong Markov property established in Corollary \ref{cor:Strong_Markov}, we have that
\begin{align*}
\EE_{\QQ^p}\left[\varphi(\omega(t)) \mathbf{1}_{\{t\geq \tau_{\Sigma}\}}\right]
&= \EE_{\QQ^p}\left[\mathbf{1}_{\{t\geq \tau_{\Sigma}\}}\EE_{\QQ^p}\left[\varphi(\omega(t))\big{|}\cF_{\tau_{\Sigma}}\right]\right]\\
&= \EE_{\QQ^p}\left[\mathbf{1}_{\{t\geq \tau_{\Sigma}\}}\EE_{\QQ^{\omega(\tau_{\Sigma}(\omega))}}\left[\varphi(\omega(t-\tau_{\Sigma}(\omega)))\right]\right]\\
&= \EE_{\QQ^p}\left[\mathbf{1}_{\{t\geq \tau_{\Sigma}\}} u(t-\tau_{\Sigma},\omega(\tau_{\Sigma}))\right]
\quad\hbox{(by the first equality in \eqref{eq:Evaluation_on_Sigma})}\\
&=0\quad\hbox{(by the second equality in \eqref{eq:Evaluation_on_Sigma})}.
\end{align*}
Thus, we conclude that identity \eqref{eq:Absorption_tangent_boundary_at_t_test_function} indeed holds, which implies \eqref{eq:Absorption_tangent_boundary_at_t}. This completes the proof.
\end{proof}

\section{The Dirichlet heat kernel}
\label{sec:Dirichlet_heat_kernel}

In \S \ref{sec:Bilinear_form} we introduce a bilinear form associated to the generalized Kimura operator. This allows us to give a variational formulation to the parabolic equation defined by the generalized Kimura operator and define the notion of weak solutions in \S \ref{sec:Weak_solutions}. The weak solutions to the initial value problem associated to the generalized Kimura operator form a semigroup that admits an integral representation in terms of the \emph{Dirichlet heat kernel}. We then introduce in \S \ref{sec:Adjoint} the adjoint operator of the generalized Kimura operator (with respect to a suitable weighted Sobolev space) and prove that the Dirichlet heat kernel is a weak solution to the adjoint parabolic equation. This enables us to study the boundary regularity properties of the Dirichlet heat kernel, which in turn determines the \emph{hitting distributions} of the Kimura diffusion on tangent boundaries of the compact manifold $P$.

\subsection{A bilinear form associated to the generalized Kimura operator}
\label{sec:Bilinear_form}
We define a bilinear form $Q(u,v)$ associated to the generalized Kimura operator and we prove that $Q(u,v)$ is continuous and satisfies the G\r{a}rding inequality on suitable weighted Sobolev spaces. Following \cite[\S 3]{Epstein_Mazzeo_2016}, we can associate a bilinear form $Q(u,v)$ to the generalized Kimura operator by letting
\begin{equation}
\label{eq:Q_L}
Q(u,v) = -(Lu,v)_{L^2(P;d\mu)},
\end{equation}
which can be written as
\begin{equation}
\label{eq:Bilinear_form}
Q(u,v) := Q_{\hbox{\tiny{sym}}}(u,v) + (Vu, v)_{L^2(P; d\mu)},\quad\forall\, u,v \in C^{\infty}_c(\Int(P)),
\end{equation}
where $Q_{\hbox{\tiny{sym}}}(u,v)$ is a symmetric bilinear form and $V$ is a first order vector field on the manifold $P$. When written in a local system of coordinates on a neighborhood $B^{\infty}_R$ of the origin in $\bar S_{n,m}$, the symmetric bilinear form $Q_{\hbox{\tiny{sym}}}(u,v)$ takes the form
\begin{equation}
\label{eq:Q_sym}
\begin{aligned}
Q_{\hbox{\tiny{sym}}} (u,v) 
&= \int_{B^{\infty}_R}\left(\sum_{i=1}^nx_iu_{x_i}v_{x_i} 
+ \sum_{i,j=1}^n \frac{1}{2}x_ix_ja_{ij}(u_{x_i}v_{x_j} + u_{x_j}v_{x_i})\right)\, d\mu\\
&\quad + \int_{B^{\infty}_R}\left(\sum_{i=1}^n\sum_{l=1}^m \frac{1}{2}x_ic_{il}(u_{x_i}v_{y_l} + u_{y_l}v_{x_i}) 
+ \sum_{l,k=1}^m \frac{1}{2}d_{lk}(u_{y_l}v_{y_k} + u_{y_k}v_{y_l})\right)\, d\mu,
\end{aligned}
\end{equation}
for all $u,v\in C^1(B^{\infty}_R)$, and $d\mu,$ the weighted measure,
is defined by analogy with \eqref{eq:Weight} by
\begin{equation}
\label{eq:Weight_local}
d\mu(z) =\prod_{i=1}^n x_i^{b_i(z)-1}\, dx_i\prod_{l=1}^m dy_l,\quad\forall\, z=(x,y)\in B^{\infty}_R.
\end{equation}
The first order vector field $V$ satisfies
\begin{equation}
\label{eq:V}
\begin{aligned}
V &= 
\sum_{i=1}^n x_i 
\left(\alpha_i(z) +\sum_{k=1}^n\ln x_k\left(\sum_{j=1}^n\gamma_{ikj}(z) \cdot\partial_{x_j}b_k\left(z\right)
 + \sum_{l=1}^m\nu_{ikl}(z) \cdot\partial_{y_l}b_k\left(z\right)\right)\right)\partial_{x_i}\\
&\quad+\sum_{l,k=1}^m\sum_{j=1}^n \beta_{lkj}(z)(\ln x_j\cdot\partial_{y_k} b_j\left(z\right)) \partial_{y_l},
\end{aligned}
\end{equation}
where the functions $\alpha_i,\beta_{lkj},\gamma_{ikj}, \nu_{ikl}:\bar \cB_R\rightarrow\RR$ are smooth. But for the log-singularities, this is a tangent vector field,
which is why it is, in essence, a lower order perturbation. We define the space of
functions $H^1(P;d\mu)$ to be the closure of $C^1_c(P\backslash \partial^T P)$
with respect to the norm:
\begin{equation}
\label{eq:H_1_norm}
\|u\|^2_{H^1(P;d\mu)} := Q_{\hbox{\tiny{sym}}} (u, u) + \|u\|^2_{L^2(P;d\mu)},
\end{equation}
where we recall that the portion of the boundary $\partial^T P$ is
defined in \eqref{eq:Tangent_boundary}.

In the next result we prove that the bilinear form $Q(u,v)$ is continuous and satisfies the G\r{a}rding inequality. These properties where also established in \cite[\S 3]{Epstein_Mazzeo_2016} in the case when the weights $\{B_i:1\leq i\leq n\}$ of the generalized Kimura operator are all positive. In Lemma \ref{lem:Properties_bilinear_form} we show how to extend these results to the case when the weights $\{B_i:1\leq i\leq n\}$ satisfy the cleanness condition, \eqref{eq:Cleanness}, and so we can allow for zero weights. The reason why this situation needs consideration is that the first order vector field $V$ in \eqref{eq:V} can have logarithmically divergent terms.

\begin{lem}[Properties of the bilinear form $Q(u,v)$]
\label{lem:Properties_bilinear_form}
Suppose that the generalized Kimura operator satisfies the standard
assumptions, then for all $\eps>0$, there is a positive constant,
$C=C(L,\eps,m,n)$, such that for all $u,v\in H^1(P;d\mu)$, we have
that
\begin{equation}
\label{eq:Control_perturbation}
\left|(Vu,v)_{L^2(P;d\mu)}\right| \leq \|u\|_{H^1(P;d\mu)}\left(\eps \|v\|_{H^1(P;d\mu)} + C\|v\|_{L^2(P;d\mu)}\right).
\end{equation}
Moreover, there are positive constants, $c_1, c_2$ and $c_3$, depending on the coefficients of the generalized Kimura operator, such that for all $u,v\in H^1(P;d\mu)$, we have that
\begin{align}
\label{eq:Continuity_Dirichlet_form}
|Q(u,v)| &\leq c_1 \|u\|_{H^1(P;d\mu)} \|v\|_{H^1(P;d\mu)},\\
\label{eq:Coercivity_Dirichlet_form}
Q(u,u) & \geq c_2 \|u\|_{H^1(P;d\mu)}^2 - c_3 \|u\|_{L^2(P;d\mu)}^2.
\end{align}
\end{lem}

\begin{proof}
It is sufficient to prove estimate \eqref{eq:Control_perturbation} because
inequalities \eqref{eq:Continuity_Dirichlet_form} and
\eqref{eq:Coercivity_Dirichlet_form} are  straightforward consequences of
\eqref{eq:Control_perturbation}, of definition~\eqref{eq:H_1_norm} of the norm $\|\cdot\|_{H^1(P;d\mu)}$, and of the expressions \eqref{eq:Q_sym} and \eqref{eq:V}. Let 
$\cA:=\{(\psi_{\alpha}, U_{\alpha}, V_{\alpha})\}_{\alpha}$ be a finite atlas of adapted local coordinates on the compact manifold $P$ with corners. Let $\{\varphi_{\alpha}\}_{\alpha}$ be a partition of the unity subordinate to the open cover $\{U_{\alpha}\}_{\alpha}$. Then we have that
\begin{equation}
\label{eq:Control_perturbation_1}
\left|(Vu,v)_{L^2(P;d\mu)}\right| \leq \sum_{\alpha} \left|(Vu,\varphi_{\alpha}v)_{L^2(P;d\mu)}\right|.
\end{equation}
For $\alpha$ fixed, we assume that the relatively open set $U_{\alpha}\subset P$ 
is homeomorphic to a relatively open neighborhood of the origin in 
$\bar S_{n_{\alpha}, m_{\alpha}}$. We denote by 
$J_{\alpha}\subseteq\{1,\ldots,n_{\alpha}\}$ the set of indices
corresponding to the constant weights of the generalized Kimura operator when written in the
coordinate chart $(\psi_{\alpha}, U_{\alpha}, V_{\alpha})$. Identity \eqref{eq:V} implies that there is a positive constant, $C=C(L, m_{\alpha}, n_{\alpha})$, such that
\begin{align*}
\left|(Vu,\varphi_{\alpha}v)_{L^2(P;d\mu)}\right| 
&\leq C\int_{S_{n_{\alpha},m_{\alpha}}} \left(1+\sum_{j\notin J_{\alpha}}|\ln x_j|\right) \left(\sum_{i=1}^{n_{\alpha}} x_i|u_{x_i}|+\sum_{l=1}^{m_{\alpha}}|u_{y_l}|\right) |\varphi_{\alpha}v|\, d\mu,
\end{align*}
where we used the fact that the terms $\ln x_i$ in \eqref{eq:V} are multiplied by derivatives of $b_i$, which are zero when the weight $b_i$ is constant. Denoting by $q_{\alpha}(z)$ the first parenthesis in the preceding inequality and applying H\"older's inequality, we obtain
\begin{equation}
\label{eq:Control_perturbation_2}
\left|(Vu,\varphi_{\alpha}v)_{L^2(P;d\mu)}\right| 
\leq \|u\|_{H^1(P;d\mu)} \left[\int_{S_{n_{\alpha},m_{\alpha}}}
  |q_{\alpha}| |\varphi_{\alpha}v|^2\, d\mu\right]^{\frac 12},
\end{equation}
where there is a positive constant, $C_{\alpha}$, such that
$$
|q_{\alpha}(z)| \leq C_{\alpha} \left(1+\sum_{j\notin J_{\alpha}} |\ln x_j|^2\right),\quad\forall\, z\in S_{n_{\alpha},m_{\alpha}}.
$$
The definition of the set $J_{\alpha}$ implies that the function $q_{\alpha}(z)$ contains logarithmically divergent terms $\ln x_j$ only along coordinates $j\notin J_{\alpha}$, for which the corresponding weight, $b_j$, is positive on $\{x_j=0\}\cap\partial B^{\infty}_1$. This observation allows us to easily adapt the proof of \cite[Lemma 1.1]{Epstein_Mazzeo_2016}, from the case of positive weights  on all boundary components to the case of weights that satisfy the cleanness condition \eqref{eq:Cleanness}, and obtain that for all $\eps>0$, there is a positive constant, $C=C(\eps)$, such that
\begin{equation}
\label{eq:Control_perturbation_3}
\left[\int_P |q_{\alpha}| |\varphi_{\alpha}v|^2\, d\mu\right]^{\frac 12}\leq \eps\|v\|_{H^1(P;d\mu)} + C\|v\|_{L^2(P;d\mu)}.
\end{equation}
Estimates \eqref{eq:Control_perturbation_1}, \eqref{eq:Control_perturbation_2}, and \eqref{eq:Control_perturbation_3} imply inequality \eqref{eq:Control_perturbation}. This concludes the proof.
\end{proof}

\subsection{The semigroup associated to the generalized Kimura operator and the Dirichlet heat kernel}
\label{sec:Weak_solutions}
We prove that weak solutions to the initial value problem associated to the generalized Kimura operator,
\begin{equation}
\label{eq:Initial_value_problem}
\begin{aligned}
\left\{\begin{array}{rl}
u_t-Lu=0 & \hbox{ on } (0,T)\times P,\\ 
u(0) = f& \hbox{ on } P, 
\end{array} \right.
\end{aligned}
\end{equation}
where $T>0$, generate a strongly continuous contraction semigroup on $L^2(P;d\mu)$. We then establish that the semigroup admits an integral representation in terms of the so-called Dirichlet heat kernel. The Dirichlet heat kernel gives us the distributions of the paths of the Kimura processes that are not absorbed on the tangent boundary components of the compact manifold with corners. This is the main result of this section, which we prove in Proposition \ref{prop:Dirichlet_heat_kernel_existence}.

To give the definition of weak solutions to the parabolic equation \eqref{eq:Initial_value_problem}, we first need to introduce additional function spaces. We let $C([0,T], L^2(P;d\mu))$ be the space of continuous functions, 
$u:[0,T]\rightarrow L^2(P;d\mu)$, endowed with the norm
$$
\|u\|_{C([0,T], L^2(P;d\mu))}:=\sup_{t\in [0,T]} \|u(t)\|_{L^2(P;d\mu)} <\infty.
$$
We let $L^2((0,T),L^2(P;d\mu))$ denote the space of measurable functions, $u:(0,T)\times P\rightarrow\RR$, endowed with the norm,
$$
\|u\|^2_{L^2((0,T),L^2(P;d\mu))}:=\int_{(0,T)} \|u(t)\|^2_{L^2(P;d\mu)} \, dt<\infty.
$$
Let $H^{-1}(P;d\mu)$ denote the dual space of $H^1(P;d\mu)$. We define the spaces $L^2((0,T),H^1(P;d\mu))$ and 
$L^2((0,T),H^{-1}(P;d\mu))$ analogously to $L^2((0,T),L^2(P;d\mu))$ by replacing the space $L^2(P;d\mu)$ by $H^1(P;d\mu)$ and $H^{-1}(P;d\mu)$, respectively. Finally, the space of weak solutions to the parabolic problem \eqref{eq:Initial_value_problem} is denoted by 
$\cF((0,T)\times P)$ and it consists of measurable functions, $u:(0,T)\times P\rightarrow\RR$, with the property that
\begin{equation}
\label{eq:Space_weak_sol}
u \in L^2((0,T); H^1(P;d\mu))
\quad\hbox{ and }\quad
\frac{du}{dt} \in L^2((0,T); H^{-1}(P;d\mu)),
\end{equation}
where the time derivative, $\frac{du}{dt}$, is understood as a distribution. The remark following \cite[Chapter 3, Section 4, Theorem 4.1]{Lions_Magenes1} shows that
\begin{equation}
\label{eq:inclusion_cF_C}
\cF((0,T)\times P) \subset C([0,T], L^2(P;d\mu)).
\end{equation}
Following \cite[Chapter 3, Section 4]{Lions_Magenes1}, we define the notion of a weak solution:

\begin{defn}[Weak solution]
\label{defn:Weak_solution}
Let $f\in L^2(P;d\mu)$. A function $u\in \cF((0,T)\times P)$ is a weak solution to the initial-value problem \eqref{eq:Initial_value_problem} if the following hold:
\begin{enumerate}
\item[1.] For all test functions $v\in \cF((0,T)\times P)$, we have that
\begin{equation}
\label{eq:Weak_sol_var_eq}
\int_0^T\left\langle \frac{du(t)}{dt}, v(t)\right\rangle\, dt+\int_0^T Q(u(t),v(t))\, dt= 0,
\end{equation}
where $\left\langle\cdot, \cdot\right\rangle$ denotes the dual pairing of $H^{-1}(P;d\mu)$ and $H^1(P;d\mu)$.
\item[2.] The initial condition is satisfied in the $L^2(P;d\mu)$-sense, that is
\begin{equation}
\label{eq:Weak_sol_initial_cond}
\|u(t)-f\|_{L^2(P;d\mu)}\rightarrow 0,\quad\hbox{as } t\downarrow 0.
\end{equation}
\end{enumerate}
\end{defn}

An application of \cite[Chapter 3, Section 4, Theorem 4.4.1 and Remark 4.4.3]{Lions_Magenes1} combined with the property that the bilinear form $Q(u,v)$ is continuous and satisfies the G\r{a}rding inequality, by Lemma \ref{lem:Properties_bilinear_form}, yields:

\begin{lem}[Existence and uniqueness of solutions]
\label{lem:Existence_uniqueness_homogeneous}
\cite[Chapter 3, Section 4, Theorem 4.4.1 and Remark
  4.4.3]{Lions_Magenes1} Suppose that the generalized Kimura operator
satisfies the standard assumptions, then, for all $f\in L^2(P;d\mu)$,
there is a unique weak solution, $u\in\cF((0,T)\times P)$, to the
initial-value problem \eqref{eq:Initial_value_problem}
%Added energy estimate
and satisfies the energy estimate:
\begin{equation}
\label{eq:Energy_estimate}
\begin{aligned}
&\sup_{t\in [0,T]}\|u(t)\|_{L^2(P;d\mu)} + \|u\|_{L^2((0,T);H^1(P;d\mu))} +\left\|\frac{du}{dt}\right\|_{L^2((0,T);H^{-1}(P;d\mu))}\\
&\quad\leq C\left(\|f\|_{L^2(P;d\mu)} + \|g\|_{L^2((0,T);L^2(P;d\mu))}\right),
\end{aligned}
\end{equation}
where $C=C(T, c_1, c_2, c_3)$ is a positive constant and $c_1, c_2, c_3$ are the constants appearing in \eqref{eq:Continuity_Dirichlet_form} and \eqref{eq:Coercivity_Dirichlet_form}.
\end{lem}

\begin{rmk}[The boundary condition satisfied by weak solutions]
\label{rmk:Boundary_cond}
From~\eqref{eq:Weight_local} it is apparent that the measure $d\mu$ is
not locally integrable near boundary points where a weight $b_i$
vanishes.  In the companion article, \cite{Epstein_Pop_2016}, we prove
that weak solutions to the initial-value problem
\eqref{eq:Initial_value_problem} satisfy a homogeneous Dirichlet
boundary condition along these portions of the boundary, $\partial^T P,$
to which $L$ is tangent. This property is explained in more detail in
\cite[Remark 2.2]{Epstein_Pop_2016}. In \S
\ref{sec:Parabolic_Dirichlet_problem}, we consider the Dirichlet
problem with non-homogeneous Dirichlet boundary conditions, which
plays a central role in our study of the hitting distributions on
tangent boundary components of the Kimura diffusions.
\end{rmk}

It is shown in \cite[Property (1.19)]{Sturm_1995} that the unique weak solution, $u$, to problem \eqref{eq:Initial_value_problem} with initial condition 
$f\in L^2(P;d\mu)$ has the property that
\begin{align}
\label{eq:Monotonicity_solutions}	
[0,T]\ni t\mapsto \|u(t)\|_{L^2(P;d\mu)} \quad\hbox{ is non-increasing}.
\end{align}
The uniqueness statement in Lemma \ref{lem:Existence_uniqueness_homogeneous}, the inclusion relation \eqref{eq:inclusion_cF_C}, and  property \eqref{eq:Monotonicity_solutions}	show that there is a strongly continuous, contraction semigroup on $L^2(P;d\mu)$, 
$\{T_t\}_{t\geq 0}$, such that the unique weak solution, $u$, to problem \eqref{eq:Initial_value_problem} with initial condition 
$f\in L^2(P;d\mu)$, can be represented in the form 
\begin{equation}
\label{eq:Representation_homogeneous_sol}
u(t)= T_t f,\quad\forall\, t\geq 0.
\end{equation}

To prove the existence of the Dirichlet heat kernel, we next apply a local a priori estimate from \cite{Epstein_Mazzeo_2016} to prove that the weak solutions to the initial-value problem \eqref{eq:Initial_value_problem} are H\"older continuous on $(0,T)\times (P\backslash\partial^T P)$. Let 
$d:P\times P\rightarrow [0,\infty)$ be the Riemannian distance induced on the compact manifold $P$ by the principal symbol of the generalized Kimura operator. Let $\alpha\in (0,1)$, $0<t_1<t_2$, and $U$ be a relatively open set such that $\bar U\subset P\backslash\partial^T P$. We say that a function $u:[t_1,t_2]\times\bar U\rightarrow \RR$ belongs to the H\"older space 
$C^{\alpha}_{\hbox{\tiny{WF}}}([t_1,t_2]\times\bar U)$ if the semi-norm:
$$
\|u\|_{C^{\alpha}_{\hbox{\tiny{WF}}}([t_1,t_2]\times\bar U)} 
:= \sup_{(t',p')\neq (t'',p'')} 
\frac{|u(t',p') - u(t'', p'')|}{\left(\sqrt{|t'-t''|} + d(p',p'')\right)^{\alpha}} <\infty.
$$
We can now state:

\begin{lem}[H\"older regularity on $P\backslash\partial^T P$]
\label{lem:Holder_regularity_semigroup}
Suppose that the generalized Kimura operator satisfies the standard
assumptions, then for all $t>0$ and $p\in P\backslash \partial^T P$,
there is a neighborhood of the point $(t,p)$, $(t_1,t_2)\times U$,
such that
$$
[t_1,t_2]\times\bar U\subset (0,\infty)\times P\backslash \partial^T P,
$$ 
and there are positive constants, 
$\alpha\in (0,1)$ and $C=C(L,t_1,t_2,U)$, with the property that for all $f\in L^2(P;d\mu)$, we have
\begin{equation}
\label{eq:Holder_estimate_semigroup}
\|T_tf\|_{C^{\alpha}_{\hbox{\tiny{WF}}}([t_1,t_2]\times\bar U)} \leq C \|f\|_{L^2(P;d\mu)}.
\end{equation}
\end{lem}

\begin{rmk}
Lemma \ref{lem:Holder_regularity_semigroup} shows that we can choose a continuous version on 
$(0,\infty)\times (P\backslash \partial^T P)$ of the semigroup $\{T_t\}_{t\geq 0}$ defined in \eqref{eq:Representation_homogeneous_sol}.
\end{rmk}

\begin{proof}[Proof of Lemma \ref{lem:Holder_regularity_semigroup}]
For a point $p\in P\backslash \partial^T P$, we can find a relatively open set $V$ such that $\bar U\subset V\subset P\backslash\partial^T P$ and $\hbox{dist}(U, P\backslash V) > 0$. The generalized Kimura operator has the property that is has positive weights when restricted to the set $\bar V$, which implies that we can apply the local a priori H\"older seminorm estimate 
\cite[Corollary 4.1]{Epstein_Mazzeo_2016}, which together with the immediate generalization of supremum norm estimate 
\cite[Theorem 2.1]{Sturm_1995} to the generalized Kimura operator give us estimate \eqref{eq:Holder_estimate_semigroup}. 
This completes the proof.
\end{proof}

We next prove the existence of the Dirichlet heat kernel:

\begin{prop}[Dirichlet heat kernel]
\label{prop:Dirichlet_heat_kernel_existence}
Suppose that the generalized Kimura operator satisfies the standard
assumptions. For all $t>0$ and $p\in P\backslash \partial^T P$, there
is a non-negative measurable function, $k(t,p,\cdot) \in L^2(P;d\mu)$,
such that
\begin{enumerate}
\item[(i)] 
For all $f\in L^2(P;d\mu)$, we have that
\begin{equation}
\label{eq:Semigroup_kernel}
T_tf(p) =  \int_P k(t,p,p') f(p')\, d\mu(p').
\end{equation}
\item[(ii)]
For all $f\in C^{\infty}_c(P\backslash \partial^T P)$, we have
\begin{equation}
\label{eq:Dirichlet_heat_kernel_existence}
\EE_{\QQ^p}\left[f(\omega(t))\right] = \int_P k(t,p,p') f(p')\, d\mu(p').
\end{equation}
\item[(iii)]
For all Borel measurable sets, $B\subseteq P\backslash\partial^T P$, we have that property \eqref{eq:Density_Dirichlet_heat_kernel} holds, where we let $\QQ^p$ be the unique solution to the martingale problem associated to the generalized Kimura operator constructed in Theorem \ref{thm:Wellposed_mart_problem}.
\item[(iv)]
For all $0<t_1<t_2$ and for all compact sets, $K\subset P\backslash\partial^TP$, there is a positive constant, $C=C(L,t_1,t_2,K)$, such that 
\begin{equation}
\label{eq:Weighted_L_2_estimates_Diri_kernel}
\|k(t,p,\cdot)\|_{L^2(P;d\mu)} \leq C,
\quad\forall\, t\in [t_1,t_2],\quad\forall\, p \in K.
\end{equation}
\end{enumerate}
\end{prop}

We can now give the

\begin{proof}[Proof of Theorem \ref{thm:Distribution_absorbed}]
It is immediate that property \eqref{eq:Density_Dirichlet_heat_kernel} is a direct consequence of Proposition \ref{prop:Dirichlet_heat_kernel_existence} (iii). 
\end{proof}

Before we give the proof of Proposition \ref{prop:Dirichlet_heat_kernel_existence}, we have the following:

\begin{rmk}
\label{rmk:T_t_bounded_functions}
Property \eqref{eq:Density_Dirichlet_heat_kernel} implies that we can extend the action of the semigroup $\{T_t\}_{t\geq 0}$ from
$L^2(P;d\mu)$ to the space of bounded measurable functions as follows.  Let $f:P\rightarrow\RR$ be a bounded Borel measurable function, let $t>0$, and $p\in P\backslash \partial^T P$. Then, we set:
\begin{equation}
\label{eq:T_t_bounded_functions_nonabsorbed}
\begin{aligned}
T_t f(p) &:= \EE_{\QQ^p}\left[f(\omega(t)) \mathbf{1}_{\{t <\tau_{\partial^T P}\}}\right] \\ 
&= \int_{P\backslash\partial^T P} k(t,p,p') f(p')\, d\mu(p')\\
&= \int_P k(t,p,p') f(p')\, d\mu(p').
\end{aligned}
\end{equation}
When $t>0$ and $p \in \partial^T P$, we have that $\tau_{\partial^T P} = 0$, and so it is natural to define:
\begin{equation}
\label{eq:T_t_bounded_functions_absorbed}
T_t f(p) := 0.
\end{equation}
\end{rmk}

\begin{proof}[Proof of Proposition \ref{prop:Dirichlet_heat_kernel_existence}]
Estimate \eqref{eq:Holder_estimate_semigroup} implies that for all $(t,p)\in (0,\infty)\times(P\backslash \partial^T P)$, there is a positive constant, $C$, such that
$$
|T_tf(p)| \leq C \|f\|_{L^2(P;d\mu)},\quad\forall\, f\in L^2(P;d\mu),
$$
and so, by the Riesz Representation Theorem, there is a measurable function, $k(t,p,\cdot)\in L^2(P;d\mu)$, such that identity \eqref{eq:Semigroup_kernel} holds.

For $f\in C^{\infty}_c(P\backslash \partial^T P)$, let $u\in C^{\infty}([0,\infty)\times P)$ be the unique solution to the initial value problem \eqref{eq:Initial_value_problem} given by \cite[Theorem 10.0.2]{Epstein_Mazzeo_annmathstudies}. Choosing 
$\varphi(s,p') := u(t-s,p')$ as test function in constructing the martingale in \eqref{eq:Mart_problem_martingales} and taking expectation under the probability measure $\QQ^p$, it follows that
\begin{equation}
\label{eq:Stochastic_representation1}
u(t,p) = \EE_{\QQ^p}\left[f(\omega(t))\right].
\end{equation}
Notice that on any tangent boundary hypersurface, $H\subseteq \partial^T P$, the initial data $f\equiv 0$ on $H$. Then it follows from Remark \ref{rmk:Sato} that the restriction of $u$ to $H$ solves the initial-value problem 
$(\partial_t-L_H)u = 0$ on $(0,\infty)\times H$ with initial condition $u(0)=0$ on $H$. Thus, we have that $u\equiv 0$ on $H$, from which we conclude that $u$ is also the unique weak solution to the initial-value problem 
\eqref{eq:Initial_value_problem} (with $P=H$) given by Lemma \ref{lem:Existence_uniqueness_homogeneous}. Combining identities 
\eqref{eq:Stochastic_representation1} and \eqref{eq:Semigroup_kernel}, we obtain \eqref{eq:Dirichlet_heat_kernel_existence}. 
Property \eqref{eq:Dirichlet_heat_kernel_existence} and Lemma \ref{lem:Absorption_tangent_boundary} imply 
\eqref{eq:Density_Dirichlet_heat_kernel}.

To conclude the proof, identity \eqref{eq:Semigroup_kernel} and estimate \eqref{eq:Holder_estimate_semigroup} give that there is a positive constant, 
$C=C(L,t_1,t_2,K)$, such that for all $t\in [t_1,t_2]$ and for all $p\in K$ we have that
\begin{align*}
\|k(t,p,\cdot)\|_{L^2(P;d\mu)} &=\sup\{|T_tf(p)|:\, f\in L^2(P;d\mu), \|f\|_{L^2(P;d\mu)} = 1\}\\
&\leq C,
\end{align*}
implying estimate \eqref{eq:Weighted_L_2_estimates_Diri_kernel}. This completes the proof.
\end{proof}

\subsection{Adjoint operator}
\label{sec:Adjoint}
We introduce the adjoint of the generalized Kimura operator and we prove that the Dirichlet heat kernel is a weak solution in the forward variables, $k(\cdot, p, \cdot)$, where $p\in P\backslash\partial^T P$ is fixed, to the parabolic equation defined by the adjoint operator. This property allows us to study the boundary regularity of the Dirichlet heat kernel along boundary components with constant weights. We prove these results in Theorems \ref{thm:Sup_est_tran_prob} and \ref{thm:Sup_est_tran_prob}.

We define the bilinear form $\widehat Q$ by
\begin{equation}
\label{eq:Adjoint_Dirichlet_form}
\widehat Q(u,v) := Q(v,u),\quad\forall\, u,v\in H^1(P;d\mu),
\end{equation}
and we define the adjoint operator of the generalized Kimura operator, $\widehat L$, by
$$
-\left(\widehat Lu,v\right)_{L^2(P,d\mu)} = \widehat Q(u,v),\quad\forall\,u, v\in C^1_c(P\backslash\partial^T P),
$$
where we recall the definition of the weighted measure $d\mu$ in \eqref{eq:Weight}. We also call $\{B_i:1\leq i\leq \eta\}$ the weights of the adjoint $\widehat L$ of the generalized Kimura operator of the compact manifold $P$ with corners. The adjoint operator is \emph{not} a generalized Kimura operator because, as shown in~\eqref{eq:Adjoint_operator_near_absorbant_boundary}, the coefficients of its zeroth and first order terms may have logarithmic singularities.

Analogously to Definition \ref{defn:Weak_solution} of weak solutions to the parabolic initial-value problem defined by the generalized Kimura operator, we define the notion of weak solutions to the initial-value problem for the adjoint operator $\widehat L$,
\begin{equation}
\label{eq:Initial_value_problem_adjoint}
\begin{aligned}
\left\{\begin{array}{rl}
u_t-\widehat Lu=0 & \hbox{ on } (0,T)\times P,\\ 
u(0) = f& \hbox{ on } P, 
\end{array} \right.
\end{aligned}
\end{equation}
where $T>0$ and $f\in L^2(P;d\mu)$, as follows:

\begin{defn}[Weak solution to the adjoint equation]
\label{defn:Weak_solution_adjoint}
Let $f\in L^2(P;d\mu)$. A function $u\in \cF((0,T)\times P)$ is a weak solution to the initial-value problem \eqref{eq:Initial_value_problem_adjoint} if the following hold:
\begin{enumerate}
\item[1.] For all test functions $v\in \cF((0,T)\times P)$, we have that
\begin{equation}
\label{eq:Weak_sol_var_eq_adjoint}
%Corrected with the dual pairing
\int_0^T\left\langle \frac{du(t)}{dt}, v(t)\right\rangle\, dt+\int_0^T \widehat Q(u(t),v(t))\, dt= 0,
\end{equation}
where $\left\langle\cdot, \cdot\right\rangle$ denotes the dual pairing of $H^{-1}(P;d\mu)$ and $H^1(P;d\mu)$.
\item[2.] The function $u$ satisfies property \eqref{eq:Weak_sol_initial_cond}.
\end{enumerate}
\end{defn}

Because the bilinear form $Q(u,v)$ is continuous and satisfies the G\r{a}rding inequality, it follows by 
\eqref{eq:Adjoint_Dirichlet_form} that $\widehat Q(u,v)$ also satisfies
\begin{align}
\label{eq:Continuity_Dirichlet_form_adjoint}
|\widehat Q(u,v)| &\leq c_1 \|u\|_{H^1(P;d\mu)} \|v\|_{H^1(P;d\mu)},\\
\label{eq:Coercivity_Dirichlet_form_adjoint}
\widehat Q(u,u) & \geq c_2 \|u\|_{H^1(P;d\mu)}^2 - c_3 \|u\|_{L^2(P;d\mu)}^2,
\end{align}
for all $u,v\in H^1(P;d\mu)$, where $c_1,c_2,c_3$ are positive constants depending only on the coefficients of the generalized Kimura operator. Thus, as in the case of the bilinear form $Q(u,v)$, we can apply 
\cite[Chapter 3, Section 4, Theorem 4.4.1 and Remark 4.4.3]{Lions_Magenes1} to conclude that:

\begin{lem}[Existence and uniqueness of solutions to the adjoint equation]
\label{lem:Existence_uniqueness_homogeneous_adjoint}
\cite[Chapter 3, Section 4, Theorem 4.4.1 and Remark 4.4.3]{Lions_Magenes1}
Suppose that the generalized Kimura operator satisfies the standard assumptions, then, for all 
$f\in L^2(P;d\mu)$, there is a unique weak solution, $u\in\cF((0,T)\times P)$, to the initial-value problem 
\eqref{eq:Initial_value_problem_adjoint}.
\end{lem}

We denote by $\{\widehat T_t\}_{t\geq 0}$ the semigroup generated by the adjoint operator, i.e. for all functions $f\in L^2(P;d\mu)$, we let 
$u\in \cF((0,T)\times P)$ be the unique weak solution to \eqref{eq:Initial_value_problem_adjoint} and we define
$$
\widehat T_t f := u(t),\quad\forall\, t>0,
$$
where the positive constant $T$ above is arbitrarily chosen. 

Our goal is next to prove in Lemma \ref{lem:Equation_Dirichlet_heat_kernel} that the Dirichlet heat kernel is a weak solution to the adjoint parabolic equation, which allows us to study its regularity properties on the boundary of the manifold $P$. 
%Add reference
%This will play a central role in \S \ref where we prove that the hitting probabilities of the Kimura diffusion on the tangent boundaries of the manifold have densities given by the normal derivatives of the Dirichlet heat kernel.
For this purpose, we first need to understand the form of the adjoint operator, $\widehat L$, when written in an adapted system of coordinates. Direct calculations using the form \eqref{eq:Operator} of the generalized Kimura operator in an adapted local system of coordinates, give us that the adjoint operator can be decomposed as 
\begin{equation}
\label{eq:Decomposition_adjoint}
\widehat L=L+\widehat V+\widehat c,
\end{equation}
where $\widehat V$ is the first order vector field and $\widehat c$ is the zeroth-order term, which take the form:
\begin{equation}
\label{eq:Adjoint_operator_near_absorbant_boundary}
\begin{aligned}
\widehat V &= \sum_{i=1}^n x_i \left(\alpha_i(z)
+ \sum_{k,j=1}^n \alpha_{kj}'(z)\partial_{x_k} b_j(z)\ln x_j
+ \sum_{j=1}^n \sum_{k=1}^m \alpha_{ijk}''(z)\partial_{y_k} b_j(z)\ln x_j \right)\partial_{x_i} \\
&\quad
+ \sum_{l=1}^m \left(\gamma_l(z)
+ \sum_{i,j=1}^n \gamma_{ijl}'(z)\partial_{x_i} b_j(z)\ln x_j
+ \sum_{j=1}^n \sum_{k=1}^m \gamma_{jkl}''(z)\partial_{y_k} b_j(z)\ln x_j \right)\partial_{y_l} \\
\widehat c& = \rho_l(z) 
+ \sum_{i,j=1}^n \sum_{k,p=1}^m\left(\rho_{ijl}'(z)\partial_{x_i} b_j(z) 
+ \rho_{ijl}''(z)\partial_{x_i}\partial_{y_k} b_j(z)
+ \rho_{jkl}'''(z)\partial_{y_k} b_j(z)\right)\ln x_j \\
&\quad
+ \sum_{i,j,k,l=1}^n \zeta_{ijkl}(z)\partial_{x_k} b_i(z) \ln x_i \partial_{x_l} b_j(z) \ln x_j\\
&\quad
+ \sum_{i,j=1}^n \sum_{l,k=1}^m \zeta_{ijkl}'(z)\partial_{y_k} b_i(z) \ln x_i \partial_{y_l} b_j(z) \ln x_j\\
&\quad
+ \sum_{i,j,k=1}^n \sum_{l=1}^m \zeta_{ijkl}''(z)\partial_{x_k} b_i(z) \ln x_i \partial_{y_l} b_j(z) \ln x_j,
\end{aligned}
\end{equation}
where the coefficients appearing in the expression of $\widehat V$ and $\widehat c$ are smooth functions on $\barB^{\infty}_R$, where $R$ is chosen small enough. We also call the coefficients $\{b_i:1\leq i\leq n\}$ the weights of the adjoint operator $\widehat L$ in an adapted local system of coordinates. We need

\begin{defn}
\label{defn:constant_boundary}
We denote by $\partial^c P$ be the set of points $p \in \partial P$ with the
property that there is a relatively open neighborhood $U\subset P$ of $p$ such that the weights of the generalized Kimura operator $L$ are constant on $U$. 
\end{defn}

From identity \eqref{eq:Adjoint_operator_near_absorbant_boundary}, it
follows that if $p\in\partial^c P$, then in an adapted local system of
coordinates, centered at $p,$ the adjoint operator $\widehat L$ takes
the form of a generalized Kimura operator \eqref{eq:Operator} to which
we add a smooth and bounded zeroth order term. Thus, the logarithmic
singularities in expressions
\eqref{eq:Adjoint_operator_near_absorbant_boundary} disappear.

We next prove that the Dirichlet heat kernel constructed in Proposition \ref{prop:Dirichlet_heat_kernel_existence} is a weak solution with respect to the forward variables, $k(\cdot, p, \cdot)$, to the homogeneous parabolic equation defined by the adjoint operator.

\begin{lem}[Equation satisfied by the Dirichlet heat kernel]
\label{lem:Equation_Dirichlet_heat_kernel}
Suppose that the generalized Kimura operator satisfies the standard assumptions.
For all $t>0$ and $p,q \in P\backslash \partial^T P$, the following hold:
\begin{enumerate}
\item[(i)]
The function $u:=k(t+\cdot,p,\cdot)$ is the unique weak solution to the initial-value problem \eqref{eq:Initial_value_problem_adjoint} with initial condition $f:=k(t,p,\cdot)$, and $k(\cdot, p, \cdot)$ belongs to $C^{\infty}((0,\infty)\times(\hbox{int}(P)\cup\partial^cP))$.
\item[(ii)]
The function $u:=k(t+\cdot,\cdot, q)$ is the unique weak solution to the initial-value problem \eqref{eq:Initial_value_problem} with initial condition $f:=k(t,\cdot,q)$, and $k(\cdot, \cdot, q)$ belongs to $C^{\infty}((0,\infty)\times P)$.
\end{enumerate} 
\end{lem}

We prove Lemma \ref{lem:Equation_Dirichlet_heat_kernel} with the aid of the following auxiliary result.

\begin{lem}[Relationship between $\{T_t\}_{t\geq 0}$ and $\{\widehat T_t\}_{t\geq 0}$]
\label{lem:Relationship_semigroups}
Suppose that the generalized Kimura operator satisfies the standard assumptions, then, for all 
$f, g\in L^2(P;d\mu)$ and for all $t\geq 0$, we have that
\begin{equation}
\label{eq:Relationship_semigroups}
\left(f, T_tg\right)_{L^2(P;d\mu)} = \left(\widehat T_tf, g\right)_{L^2(P;d\mu)}.
\end{equation}
\end{lem}

\begin{proof}
The proof is a straightforward adaptation of the argument used to prove \cite[Lemma 1.5]{Sturm_1995}. We omit the detailed proof as it is standard.
\end{proof}

We can now give the proof of Lemma \ref{lem:Equation_Dirichlet_heat_kernel}.

\begin{proof}[Proof of Lemma \ref{lem:Equation_Dirichlet_heat_kernel}]
The semigroup property of $\{T_t\}_{t\geq 0}$ gives us that, for all $s>0$, $f\in L^2(P;d\mu)$, and $p\in P\backslash \partial^T P$, we have that $T_{t+s}f(p) = T_tT_sf(p)$, where we also use the fact that we can choose a continuous version of the semigroup 
$\{T_t\}_{t\geq 0}$ on $(0,\infty)\times (P\backslash\partial^T P)$, by Lemma \ref{lem:Holder_regularity_semigroup}.
Combining the preceding property with identity \eqref{eq:Semigroup_kernel}, we obtain that
\begin{align*}
T_{t+s}f(p) &= \int_P k(t,p,p') T_sf(p')\, d\mu(p')\\
&= \left(k(t,p,\cdot), T_sf\right)_{L^2(P;d\mu)},
\quad\forall\, f\in L^2(P;d\mu).
\end{align*}
Using the duality property \eqref{eq:Relationship_semigroups} in the last identity from above, we obtain
$$
T_{t+s}f(p) = \left(\widehat T_s k(t,p,\cdot), f\right)_{L^2(P;d\mu)},\quad\forall\, f\in L^2(P;d\mu),\quad\forall\, s>0.
$$
Using again the representation property \eqref{eq:Semigroup_kernel}, it follows that
$$
T_{t+s}f(p) = \left(k(t+s,p,\cdot), f\right)_{L^2(P;d\mu)},\quad\forall\, f\in L^2(P;d\mu),\quad\forall\, s>0,
$$
and so we deduce that $k(t+s,p,\cdot) = \widehat T_sk(t,p,\cdot)$, for all $s>0$. 
%Added regularity
Using the fact that $k(t+\cdot,p,\cdot)$ is a weak solution to the initial-value problem \eqref{eq:Initial_value_problem_adjoint}, we can apply \cite[Theorem 1.2]{Epstein_Pop_2016} in a neighborhood of points $(t,p)\in (0,\infty)\times(\hbox{int}(P)\cup\partial^cP)$ to conclude that $k(\cdot, p,\cdot)$ belongs to $C^{\infty}((0,\infty)\times (\hbox{int}(P)\cup\partial^cP))$. This completes the proof of (i).

The argument used to prove (i) can be applied with $\{T_t\}_{t\geq 0}$ replaced by $\{\widehat T_t\}_{t\geq 0}$. We obtain that for all $(t,q)\in (0,\infty)\times (P\backslash\partial^T P)$ there is a kernel $\widehat k(t,q,\cdot) \in L^2(P;d\mu)$ such that, for all $f\in L^2(P;d\mu)$, we have that
$$
u(t,q) = \widehat T_t f(q) = \int_P f(p) \widehat k(t,q,p)\, d\mu(p)
$$
is the unique weak solution to the initial-value problem \eqref{eq:Initial_value_problem}. Moreover, the semigroup property of $\{\widehat T_t\}_{t\geq 0}$ implies that the function $\widehat k(t+\cdot, q, \cdot)$ is the unique weak solution to the initial-value problem \eqref{eq:Initial_value_problem} with initial condition $f:=\widehat k(t,q,\cdot)$, while  
%Changed reference
%\cite[Theorem 5.1]{Epstein_Mazzeo_2016} 
\cite[Theorem 1.2]{Epstein_Pop_2016} gives us that $\widehat k(\cdot, \cdot, q)$ belongs to $C^{\infty}((0,\infty)\times P)$. 
Identity \eqref{eq:Relationship_semigroups} implies that $k(t,p,q) = \widehat k(t,q,p)$, for all $t>0$ and for all $p,q\in P\backslash\partial^T P$. This completes the proof of (ii).
\end{proof}

%Rephrased
For $p\in\partial^c P$, let $\{b_i:1\leq i\leq n\}$ be the weights of the generalized Kimura operator in an adapted local system of coordinates. Without loss of generality, we can assume that there is an integer $n_0\in\NN$ such that the weights $\{b_i:1\leq i\leq n_0\}$ are $0$ and the weights $\{b_i:n_0+1\leq i\leq n\}$ are positive, i.e.
\begin{equation}
\label{eq:n_0}
b_i\restrictedto_{\{x_i=0\}\cap \partial B^{\infty}_R} =0,\quad\forall\, 1\leq i\leq n_0,
\quad\hbox{ and }\quad
b_i\restrictedto_{\{x_i=0\}\cap \partial B^{\infty}_R} \geq\beta_0/2>0,\quad\forall\, n_0+1\leq i\leq n,
\end{equation}
when $R>0$ is chosen small enough and we recall the choice of the positive constant $\beta_0$ in Assumption \ref{assump:Cleanness}. 
We denote by $\fe_i\in\NN^n$ the unit vector in $\RR^n$ with all coordinates 0, except for the $i$-th coordinate, which is equal to 1. We denote by $\ff_l\in\NN^m$ the unit vector in $\RR^m$ with all coordinates 0, except for the $l$-th coordinate, which is equal to 1. For all $\fa=(\fa_1,\ldots,\fa_n)\in\NN_0^n$, $\fb=(\fb_1,\ldots,\fb_m)\in\NN_0^m$, and $\fc\in\NN_0$ we denote
$$
D^{\fa}_xD^{\fb}_yD^{\fc}_t := 
\frac{\partial^{|\fa|}}{\partial x_1^{\fa_1}\ldots\partial x_n^{\fa_n}}
\frac{\partial^{|\fb|}}{\partial y_1^{\fb_1}\ldots\partial y_m^{\fb_m}}
\frac{\partial^{\fc}}{\partial t^{\fc}},
$$ where $|\fa|:=|\fa_1|+\ldots+|\fa_n|$ and
$\fb=|\fb_1|+\ldots+|\fb_m|$. If a coordinate of $\fa,\fb$ or $\fc$ is
zero then there is no derivative taken in the corresponding direction.

Theorem \ref{thm:Sup_est_tran_prob} states the
  estimates satisfied by the Dirichlet heat kernel in a neighborhood
  of tangent boundary components. This result is a direct application
  of the boundary estimates obtained in \cite[Theorem
    1.2]{Epstein_Pop_2016}, where the authors use methods based on
  higher order derivatives estimates in weighted Sobolev spaces and
  suitable Morrey-type embedding theorems to describe the boundary
  behavior of solution to generalized Kimura equations along tangent
  boundary components. Given the fact that the Dirichlet heat kernel
  is a solution to the forward Kolmogorov equation
  \eqref{eq:Initial_value_problem_adjoint}, it follows from Lemma
  \ref{lem:Equation_Dirichlet_heat_kernel}~(i) that we can apply
  \cite[Theorem 1.2]{Epstein_Pop_2016} and the estimate
  \eqref{eq:Weighted_L_2_estimates_Diri_kernel} to obtain: 

\begin{thm}[Boundary regularity along boundaries with constant weights]
\label{thm:Sup_est_tran_prob}
Suppose that the generalized Kimura operator satisfies the standard assumptions.
Let $q\in\partial^c P$ and let $R>0$ be such that in an adapted system of coordinates we identify $q$ with the origin and the restriction of the adjoint operator $\widehat L$ to $B^{\infty}_R$ is a generalized Kimura operator with constant weights. Then for all 
$0<r<R$, $0<t<T$, and $\fb\in\NN_0^m$, there is a positive constant, $C=C(\fb,L,r,R,t,T)$, such that for all $s\in [t,T]$, $p\in P\backslash\partial^T P$, and for all $z\in B^{\infty}_r$, we have that
\begin{align}
\label{eq:Sup_est_sol_1_tran_prob}
|D^{\fb}_y k(s,p,z)| + |D^{\fe_k}_x D^{\fb}_y k(s,p,z)| &\leq C \prod_{i=1}^{n_0} x_i,
\quad\forall\, n_0+1\leq k\leq n,\\
\label{eq:Sup_est_sol_2_tran_prob}
|D^{\fe_k}_x D^{\fb}_y k(s,p,z)| &\leq C \prod_{\stackrel{i=1}{i\neq k}}^{n_0} x_i,
\quad\forall\, 1\leq k\leq n_0,
\end{align}
where we recall the definition of the integer $n_0$ in \eqref{eq:n_0}.
\end{thm}

\section{Non-homogeneous Parabolic Dirichlet problem}
\label{sec:Parabolic_Dirichlet_problem}

We recall from Remark \ref{rmk:Boundary_cond} that in \S \ref{sec:Dirichlet_heat_kernel} we studied weak solutions to the initial-value problem \eqref{eq:Initial_value_problem} defined by the generalized Kimura operator, which satisfied homogeneous Dirichlet boundary conditions along the tangent portion of the boundary of the compact manifold $P$ with corners. This allowed us to construct the Dirichlet heat kernel. This section is preparation towards studying the hitting distributions on tangent boundary components of Kimura diffusions and, for this purpose, we study the parabolic problem for the generalized Kimura operator with \emph{non-homogeneous} Dirichlet boundary conditions along the tangent portion of the boundary,
\begin{equation}
\label{eq:Parabolic_Dirichlet}
\begin{aligned}
\left\{\begin{array}{rl}
u_t-Lu=0 & \hbox{ on } (0,\infty)\times P\backslash \partial^T P,\\
u=\zeta&\hbox{ on } (0,\infty)\times \partial^T P,\\ 
u = 0& \hbox{ on } \{0\}\times P. 
\end{array} \right.
\end{aligned}
\end{equation}
Our main result is Theorem \ref{thm:Parabolic_Dirichlet} where we prove that the unique solution to the non-homogeneous Dirichlet problem \eqref{eq:Parabolic_Dirichlet} can be represented by a Duhamel formula in terms of the semigroup $\{T_t\}_{t\geq 0}$ constructed in \S \ref{sec:Weak_solutions} and as a stochastic representation in terms of the unique solution to the martingale problem associated to the generalized Kimura operator introduced in Definition \ref{defn:Martingale_problem}.

\begin{thm}[Parabolic boundary-value problem]
\label{thm:Parabolic_Dirichlet}
Suppose that the generalized Kimura operator satisfies the standard assumptions.
Let $\zeta\in C^{\infty}([0,\infty)\times P)$ be such that 
\begin{equation}
\label{eq:eta_t_0}
\zeta(0,p)=0,\quad\forall\, p\in P.
\end{equation}
Then the parabolic boundary-value problem \eqref{eq:Parabolic_Dirichlet} has a unique solution, 
\begin{equation}
\label{eq:Nice_space}
u\in C^{\infty}([0,\infty)\times P\backslash \partial^T P) \cap C([0,\infty)\times P),
\end{equation}
and the following hold:
\begin{itemize}
\item[(i)]
The solution $u$ satisfies the semigroup representation,
\begin{equation}
\label{eq:Parabolic_Dirichlet_rep_semigroup}
u(t) = \zeta(t) - \int_0^t T_{t-s}(\partial_s-L)\zeta(s)\,ds,\quad\forall\,(t,p)\in (0,\infty)\times P.
\end{equation}
\item[(ii)]
The solution $u$ satisfies the stochastic representation,
\begin{equation}
\label{eq:Parabolic_Dirichlet_rep_stoch}
u(t,p) = \EE_{\QQ^p}\left[\zeta(t-(t\wedge\tau_{\partial^T P}),\omega((t\wedge\tau_{\partial^T P})))\right],
\quad\forall\,(t,p)\in (0,\infty)\times P,
\end{equation}
where the stopping time $\tau_{\partial^T P}$ is defined in \eqref{eq:tau_tangent} and 
$\QQ^p$ is the unique solution to the martingale problem in Definition \ref{defn:Martingale_problem}.

\end{itemize}
\end{thm}

\begin{rmk}[The semigroup representation of the solution to \eqref{eq:Parabolic_Dirichlet}]
We recall that the action of the semigroup $\{T_t\}_{t\geq 0}$ on the space of bounded Borel measurable functions is explained in Remark \ref{rmk:T_t_bounded_functions}. This is the sense in which we should understand the integrand on the right-hand side of identity \eqref{eq:Parabolic_Dirichlet_rep_semigroup}.
\end{rmk}

While the assumptions in Theorem \ref{thm:Parabolic_Dirichlet} are strong in terms of the regularity of the boundary data, it suffices for the application to the study of the hitting distributions on the tangent boundary components of the Kimura diffusion.
We prove Theorem \ref{thm:Parabolic_Dirichlet} with the aid of Lemmas \ref{lem:Initial_value_Dirichlet} and 
\ref{lem:Inhomogeneous_Dirichlet}.

\begin{lem}[Initial-value Dirichlet problem]
\label{lem:Initial_value_Dirichlet}
Suppose that the generalized Kimura operator satisfies the standard assumptions.
Let $f\in C^{\infty}(P)$ and let $u$ be defined by
\begin{align}
\label{eq:Initial_value_Dirichlet_rep_semigroup}
u(t,p) &= T_tf(p)\\
\label{eq:Initial_value_Dirichlet_rep_stoch}
& = \EE_{\QQ^p}\left[f(\omega(t)) \mathbf{1}_{\{t<\tau_{\partial^T P}\}}\right], 
\end{align}
for all $(t,p) \in (0,\infty)\times P$ and 
where we recall that the stopping time $\tau_{\partial^T P}$ is defined in \eqref{eq:tau_tangent} and $\QQ^p$ is the unique solution to the martingale problem in Definition \ref{defn:Martingale_problem}. Then we have that 
$$
u\in C^{\infty}([0,\infty)\times P\backslash \partial^T P)\cap C^{\infty}((0,\infty)\times P),
$$ 
and for all $0<t<T$, $l\in\NN$, relatively open sets $U\subset V\subset P\backslash \partial^T P$ such that $\hbox{dist}(\bar U, P\backslash\bar V)>0$, there are positive constants,
$C_1=C_1(l,L,t,T)$ and $C_2=C_2(l,L,T, U, V)$, such that
\begin{align}
\label{eq:Initial_value_Dirichlet_Holder_est}
\|u\|_{C^l([t,T]\times P)} &\leq C_1\|f\|_{C(P)},\\
\label{eq:Initial_value_Dirichlet_Holder_est_up_to_time_0}
\|u\|_{C^l([0,T]\times \bar U)} &\leq C_2\|f\|_{C^{l+3}(\bar V)}.
\end{align}
\end{lem}

\begin{proof}
We recall from \eqref{eq:T_t_bounded_functions_nonabsorbed} that a function $u$ defined by 
\eqref{eq:Initial_value_Dirichlet_rep_semigroup} also satisfies identity 
\eqref{eq:Initial_value_Dirichlet_rep_stoch}. To prove the remaining conclusions of Lemma \ref{lem:Initial_value_Dirichlet}, we use an approximation argument. Let
$\varphi_k:P\rightarrow[0,1]$ be a smooth cut-off function such that
\begin{equation}
\label{eq:Defn_cutoff_varphi}
\begin{aligned}
\varphi_k\equiv 1&\quad\hbox{ on }\quad \{p\in P: \hbox{dist}(p, \partial^T P)\geq 1/k\},\\
\varphi_k\equiv 0&\quad\hbox{ on } \quad \{p\in P:\hbox{dist}(p, \partial^T P)\leq 1/(2k)\},
\end{aligned}
\end{equation}
where the distance is taken with respect to the Riemannian metric induced on the manifold $P$ by the principal symbol of the generalized Kimura operator. Setting 
$$
f_k:=f\varphi_k,\quad\forall\, k\in\NN,
$$
we apply \cite[Theorem 10.0.2]{Epstein_Mazzeo_annmathstudies} and we let $u_k\in C^{\infty}([0,\infty)\times P)$ be
the unique smooth solution to the initial-value problem
\begin{equation}
\label{eq:Initial_value_Dirichlet}
\begin{aligned}
\left\{\begin{array}{rl}
u_t-Lu=0 & \hbox{ on } (0,\infty)\times P,\\
u=0&\hbox{ on } (0,\infty)\times \partial^T P,\\ 
u = f_k& \hbox{ on } \{0\}\times P.
\end{array} \right.
\end{aligned}
\end{equation}
Because $f_k\equiv 0$ on $\partial^T P$, it follows from Remark \ref{rmk:Sato} that $u_k\equiv 0$ on $[0,\infty)\times \partial^T P$. By Proposition \ref{prop:Dirichlet_heat_kernel_existence}, we have that
\begin{align}
\label{eq:Initial_value_Dirichlet_rep_seq_semigroup}
u_k(t,p) &= T_tf_k(p) \\
\label{eq:Initial_value_Dirichlet_rep_seq_stoch}
&= \EE_{\QQ^p}\left[f_k(\omega(t))\right],
\end{align}
for all $k\in\NN$ and for all $(t,p)\in[0,\infty)\times P$. The preceding identity and definition \eqref{eq:Defn_cutoff_varphi} of the cutoff functions $\varphi_k$ yield the uniform bound,
\begin{equation}
\label{eq:Initial_value_Dirichlet_sup_est_seq}
\|u_k\|_{C([0,\infty)\times P)} \leq \|f\|_{C(P)},\quad\forall\, k\in\NN.
\end{equation}
The supremum estimate \cite[Theorem 1.2]{Pop_2013b} implies that for all $0<t<T$ and $l\in\NN$ there is a positive constant, 
$C_1=C_1(l,L,t,T)$, such that
\begin{equation}
\label{eq:Initial_value_Dirichlet_Holder_est_approx}
\|u_k\|_{C^l([t,T]\times P)} \leq C_1\|f\|_{C(P)},\quad\forall\, k\in\NN,
\end{equation}
and so we can extract a subsequence converging uniformly in $C^l([t,T]\times P)$, for all $l\in\NN$, to a function $\bar u\in C^{\infty}((0,\infty)\times P)$. Letting $k$ tend to $\infty$ in identity \eqref{eq:Initial_value_Dirichlet_rep_seq_stoch}, using \eqref{eq:Defn_cutoff_varphi} and \eqref{eq:Initial_value_Dirichlet_rep_stoch}, we obtain that $\bar u = u$. Thus, inequality \eqref{eq:Initial_value_Dirichlet_Holder_est_approx} implies that $u$ satisfies estimate \eqref{eq:Initial_value_Dirichlet_Holder_est} and we have that $u$ belongs to $C^{\infty}((0,\infty)\times P)$. 

It remains to show that the solution $u$ belongs to $C^{\infty}([0,\infty)\times P\backslash \partial^T P)$. 
For $p\in P\backslash \partial^T P$, we choose $r>0$ such that $B_{2r}(p)\subset P\backslash \partial^T P$, where $B_r(p)$ denotes the relative open ball in the compact manifold $P$ centered at $p$ and of radius $r$ with respect to the Riemannian metric induced on the manifold $P$ by the generalized Kimura operator. Because  $f_k=f$ on $B_{2r}(p)$, for $k$ large enough, we can apply 
\cite[Theorem 1.3]{Pop_2013b} to obtain that for all $T>0$ and $l\in\NN$ there is a positive constant, 
$C_2=C_2(l,L,r,T)$, such that
\begin{equation}
\label{eq:Initial_value_Dirichlet_Holder_est_up_to_time_0_approx}
\|u_k\|_{C^l([0,T]\times \bar B_r(p))} \leq C_2\|f\|_{C^{l+3}(\bar B_{2r}(p))},\quad\forall\, k\in\NN,
\end{equation}
which implies that $u$ belongs to 
$C^{\infty}([0,\infty)\times P\backslash \partial^T P)$, by letting $k$ tend to $\infty$. Choosing relatively open sets $U$ and $V$
such that $\bar U\subset\bar V\subset P\backslash \partial^T P$ and $\hbox{dist}(\bar U, P\backslash\bar V)>0$, there is a positive constant $r$ and there are points $\{p_i:i=1,\ldots,I\}\subset P\backslash \partial^T P$, such that
$$
U\subset \bigcup_{i=1}^I B_r(p_i),\quad\hbox{ and }\quad  \bigcup_{i=1}^I B_{2r}(p_i) \subset V.
$$
Applying estimate \eqref{eq:Initial_value_Dirichlet_Holder_est_up_to_time_0_approx} on each ball $B_r(p_i)$, it follows that for all $0<T$ and $l\in\NN$, there is a positive constant, 
$C_3=C_3(l,L,T, U, V)$, such that
$$
\|u_k\|_{C^l([0,T]\times \bar U)} \leq C_3\|f\|_{C^{l+3}(\bar V)},\quad\forall\, k\in\NN,\quad\forall\,l\in\NN,
$$
and letting $k$ tend to $\infty$, we obtain that estimate \eqref{eq:Initial_value_Dirichlet_Holder_est_up_to_time_0} holds. This completes the proof.
\end{proof}
	
\begin{lem}[Non-homogeneous Dirichlet problem]
\label{lem:Inhomogeneous_Dirichlet}
Suppose that the generalized Kimura operator satisfies the standard assumptions.
Let $g\in C^{\infty}([0,\infty)\times P)$. Then the non-homogeneous Dirichlet problem,
\begin{equation}
\label{eq:Inhomogeneous_Dirichlet}
\begin{aligned}
\left\{\begin{array}{rl}
u_t-Lu=g & \hbox{ on } (0,\infty)\times P\backslash \partial^T P,\\
u=0&\hbox{ on } (0,\infty)\times \partial^T P,\\ 
u = 0& \hbox{ on } \{0\}\times P, 
\end{array} \right.
\end{aligned}
\end{equation}
has a unique solution $u$ that satisfies \eqref{eq:Nice_space}. Moreover, the solution $u$ satisfies the representations:
\begin{align}
\label{eq:Inhomogeneous_Dirichlet_rep_semigroup}
u(t,p) &= \int_0^t T_{t-s}g(s)\, ds \\
\label{eq:Inhomogeneous_Dirichlet_rep_stoch}
&= \EE_{\QQ^p}\left[\int_0^{t\wedge\tau_{\partial^T P}}g(t-s,\omega(s))\, ds\right],
\end{align}
for all $(t,p)\in (0,\infty)\times P$, 
where the stopping time $\tau_{\partial^T P}$ is defined in \eqref{eq:tau_tangent} and $\QQ^p$ is the unique solution to the martingale problem in Definition \ref{defn:Martingale_problem}. 
\end{lem}
\begin{rmk} In general the function $u(t,p)$ cannot be differentiable along
  $\pa^T P.$
\end{rmk}
\begin{proof}
We first prove the \emph{existence} of solutions to the non-homogeneous Dirichlet problem \eqref{eq:Inhomogeneous_Dirichlet} using Duhamel's principle. We define $u$ as in \eqref{eq:Inhomogeneous_Dirichlet_rep_semigroup} and we notice that identity 
\eqref{eq:Inhomogeneous_Dirichlet_rep_stoch} holds by the first equality in
\eqref{eq:T_t_bounded_functions_nonabsorbed} and property
\eqref{eq:Absorption_tangent_boundary}. Estimate
\eqref{eq:Initial_value_Dirichlet_Holder_est_up_to_time_0} and the
representation formula \eqref{eq:Initial_value_Dirichlet_rep_semigroup} give us
that for all $t>0$, $l\in\NN$, and relatively open sets $U\subset V\subset
P\backslash \pa^{T} P$ such that $\hbox{dist}(\bar U, P\backslash\bar V)>0$, there is a positive constant,
$C=C(l,L,t,U,V)$, such that
$$
\|T_{\cdot-s}g(s)\|_{C^l([s,t]\times\bar U)} \leq C\|g(s)\|_{C^{l+3}(\bar V)},
$$
from which it follows that the function $u$ defined by \eqref{eq:Inhomogeneous_Dirichlet_rep_semigroup} belongs to
$C^{\infty}([0,\infty)\times P\backslash \partial^T P)$. It remains to show that $u$ is continuous up to 
$[0,\infty)\times \partial^T P$ and it is identically $0$ on this portion of the boundary. 
For $(t,p)\in [0,\infty)\times \partial^T P$, we let $\{(t_k,p_k)\}_{k\in\NN}$ be a sequence of points in 
$[0,\infty)\times(P\backslash \partial^T P)$ that converges to $(t,p)$. If $t=0$, we notice that 
$$
|u(t_k,p_k)| \leq \|g\|_{C([0,T_0]\times P)} t_k\rightarrow 0,\quad\hbox{ as } k\rightarrow\infty,
$$
where we used the fact that
\begin{equation}
\label{eq:Inhomogeneous_Dirichlet_sup}
\|T_tg(s)\|_{C(P)} \leq \|g\|_{C([0,T_0]\times P)},\quad\forall\, s\in [0,T_0],
\end{equation}
and where we denote $T_0:=\sup_{k} t_k$. If $t>0$, we choose $\eps\in (0,t/2)$ and we decompose
\begin{equation}
\label{eq:Inhomogeneous_Dirichlet_decompose}
u(t_k,p_k) = \int_0^{t-\eps} T_{t_k-s}g(s)(p_k)\, ds + \int_{t-\eps}^{t_k} T_{t_k-s}g(s)(p_k)\, ds,\quad\forall\, k\in\NN.
\end{equation}
We let $K_1=K_1(\eps)$ be an integer such that $t_k\in (t-\eps/2,t+\eps/2)$, for all $k\geq K_1$. Then for all $s\in [0,t-\eps]$, we have that $t_k-s\geq \eps/2$, and so Lemma \ref{lem:Initial_value_Dirichlet} gives us that $T_{t_k-s}g(s)(p_k)$ converges to $0$, as $(t_k,p_k)$ converges to $(t,p)$. Using also the uniform bound \eqref{eq:Inhomogeneous_Dirichlet_sup}, the Dominated Convergence Theorem implies that there is an integer $K_2=K_2(\eps)$ such that 
$$
\left|\int_0^{t-\eps} T_{t_k-s}g(s)(p_k)\, ds\right| <\eps,\quad\forall\, k\geq K_2\vee K_1.
$$
The uniform bound \eqref{eq:Inhomogeneous_Dirichlet_sup} and the fact that $|t_k-(t-\eps)|<2\eps$ imply that
$$
\left|\int_{t-\eps}^{t_k} T_{t_k-s}g(s)(p_k)\, ds\right| <2\eps\|g\|_{C([0,T_0]\times P)},\quad\forall\, k\geq K_2\vee K_1.
$$
Combining the preceding two inequalities, we conclude that for all $\eps>0$, there is and integer $K_0=K_0(\eps):=K_2\vee K_1$, such that $|u(t_k,p_k)|<\eps$, for all $k\geq K_0$. Combining the cases when $t=0$ and $t>0$, we proved that $u$ is continuous up to 
$[0,\infty)\times \partial^TP$ and it is identically $0$ on this portion of the boundary.

To prove the \emph{uniqueness} of solutions $u$ such that \eqref{eq:Nice_space} holds, we choose $\varphi(s,p):=u(t-s,p)$ in the definition of the martingales \eqref{eq:Mart_problem_martingales}, where we let $s \in [0,t\wedge\tau_{\partial^T P})$ and $p\in P$. Taking expectation under the probability measure $\QQ^p$, it follows that $u$ satisfies the stochastic representation \eqref{eq:Inhomogeneous_Dirichlet_rep_stoch}, which implies that the solution is uniquely defined in terms of the source function $g$. This completes the proof.
\end{proof}

We conclude this section with

\begin{proof}[Proof of Theorem \ref{thm:Parabolic_Dirichlet}]
By defining the solution $u$ by
\begin{equation}
\label{eq:Sol_inhom_Dirichlet}
u:=\zeta + v,
\end{equation}
and letting $v$ be the unique solution to the non-homogeneous Dirichlet problem \eqref{eq:Inhomogeneous_Dirichlet} with source function
$g:=-(\partial_t-L)\zeta$, it follows that $u$ satisfies \eqref{eq:Nice_space} and the semigroup representation 
\eqref{eq:Parabolic_Dirichlet_rep_semigroup}. Because the function $v$ is identically $0$ on $\{0\}\times P$ 
and on $(0,\infty)\times\partial^T P$ by construction, property \eqref{eq:eta_t_0} is used to ensure that the function 
$u$ defined in \eqref{eq:Sol_inhom_Dirichlet} satisfies the initial condition $u=0$ on $\{0\}\times P$, and so indeed $u$ is a classical solution to \eqref{eq:Parabolic_Dirichlet}. As usual, the stochastic representation \eqref{eq:Parabolic_Dirichlet_rep_stoch} follows by choosing $\varphi(s,p):=u(t-s,p)$ in the definition of the martingales \eqref{eq:Mart_problem_martingales}, where we let 
$s \in [0,t\wedge \tau_{\partial^T P})$, and by taking expectation under the probability measure $\QQ^p$ and using the fact that the function $u$ belongs to the space \eqref{eq:Nice_space}. This completes the proof.
\end{proof}

\section{Hitting distributions of tangent boundary components and transition probabilities}
\label{sec:Prob}

In this section we give the proofs of our main results concerning the distribution probabilities of Kimura diffusions. We divide the section into two parts. In \S\ref{sec:Hitting_distributions} we study the properties of the hitting distributions (caloric measure) of the generalized Kimura process, and in \S\ref{sec:Proof_main_result} we establish the structure of the transition probabilities in Theorem \ref {thm:Tran_prob}.

\subsection{Properties of hitting distributions}
\label{sec:Hitting_distributions}

We begin with the proof of Theorem \ref{thm:Hitting_distribution_hypers} by first establishing in Proposition \ref{prop:Integral_representation} an integral representation of solutions to the non-homogeneous parabolic Dirichlet problem \eqref{eq:Parabolic_Dirichlet}. To prove Theorem \ref{thm:Hitting_corner} we consider separately the case when two tangent boundary components meet and when a tangent and a transverse boundary component meet. The former case is considered in Lemma \ref{lem:Hitting_corner_tangent_tangent}, which is a direct consequence of Theorem \ref{thm:Hitting_distribution_hypers}. The latter case requires more careful consideration because we do not have good estimates on the derivatives of the Dirichlet heat kernel on neighborhoods of transverse boundary components due to the presence of the logarithmic singularities in the coefficients of the adjoint operator. Nevertheless, we prove the latter case in 
Lemma \ref{lem:Hitting_corner_tangent_transverse} via a Landis-type growth estimate, Lemma \ref{lem:Growth_lemma}, and a regularity result for the hitting probabilities of lower dimensional boundary components, Lemma \ref{lem:Regularity_hitting_probability}. We conclude \S \ref{sec:Hitting_distributions} with the proof of the doubling property of the hitting distribution in Theorem \ref{thm:Doubling_property}. 

We denote 
\begin{equation}
\label{eq:Widehat_x_i}
\widehat x_i = (x_1,\ldots,x_{i-1},0,x_{i+1},\ldots,x_n),\quad\forall\, x=(x_1,\ldots,x_n)\in\bar\RR_+^n.
\end{equation}
We first prove an integral representation of solutions to the non-homogeneous Dirichlet problem \eqref{eq:Parabolic_Dirichlet}: 

\begin{prop}[Integral representation of solutions to the non-homogeneous Dirichlet problem]
\label{prop:Integral_representation}
Suppose that the generalized Kimura operator satisfies the standard assumptions.
Let $q\in\partial^c P$ and let $R>0$ be such that in an adapted system of coordinates we identify $q$ with the origin and the generalized Kimura operator has constant weights on $B^{\infty}_R$.
Let $0<r<R$ and $\zeta\in C^{\infty}_c((0,\infty)\times \barB^{\infty}_r)$ and let $u$ be the unique solution to the non-homogeneous parabolic Dirichlet problem \eqref{eq:Parabolic_Dirichlet}. Let 
\begin{equation}
\label{defn:time_t_0}
t_0:=\inf\{t\geq 0:\, \hbox{supp }\zeta \subseteq [0,t]\times \barB^{\infty}_r\}.
\end{equation}
Then, for all $t>2t_0$ and $p\in P\backslash \partial^T P$, the solution $u$ satisfies the integral representation,
\begin{equation}
\label{eq:Integral_representation}
u(t,p) =
\int_0^t\left(\sum_{i=1}^{n_0}\int_{\Int(H_i)} k_{x_i}(t-s,p,(\widehat x_i,y))
\zeta(s,(\widehat x_i,y)) d\mu_{H_i}(\widehat x_i,y)\right)\,ds,
\end{equation}
where without loss of generality we choose the integer $n_0$ such that it satisfies property \eqref{eq:n_0}, we denote 
$H_i=\{z\in\bar S_{n,m}: x_i=0\}$, and we recall the definition of the weighted measure $d\mu_{H_i}$ in Remark \ref{rmk:Sato}.
\end{prop}

\begin{proof}
We divide the proof in several steps. In Step \ref{step:RHS_welldefined} we establish that the right-hand side in identity 
\eqref{eq:Integral_representation} is well-defined, while in Steps \ref{step:Approximation} and \ref{step:Convergence} we employ an approximation procedure to prove that identity \eqref{eq:Integral_representation} holds.

\setcounter{step}{0}
\begin{step}[The right-hand side in \eqref{eq:Integral_representation} is well-defined]
\label{step:RHS_welldefined}
Remark \ref{rmk:Sato} and formula \eqref{eq:Weight_local} give us that we can write the measure $d\mu_{H_i}$ is adapted local coordinates as:
\begin{equation}
\label{eq:d_mu_H_i}
d\mu_{H_i} = \frac{x_i^{-b_i+1}d\mu(x,y)}{dx_i},\quad\forall\, (x,y)\in S_{n,m},\quad\forall\, 1\leq i\leq n.
\end{equation}
The assumption that $t>2t_0$, where the positive constant $t_0$ is defined in \eqref{defn:time_t_0}, gives us that the integrand 
$k_{x_i}(t-s,p,(\widehat x_i,y))$ on the right-hand side of \eqref{eq:Integral_representation} is evaluated only for $t-s>t_0$.
Because the operator has constant weights in the adapted local system of coordinates, we can apply
%Corrected reference to theorem
Theorem \ref{thm:Sup_est_tran_prob}
to obtain that the normal derivative $k_{x_i}(\cdot,p,\cdot)$ is a smooth function on 
%Added ^\infty to ball
$(0,\infty)\times(\Int(H_i)\cap\bar B_R^{\infty})$, for all $p\in P\backslash\partial^TP$ and for all $1\leq i\leq n_0$. Thus, we have that $k_{x_i}(\cdot,p,\cdot)$ is a continuous function up to 
%Added ^\infty to ball
$(0,\infty)\times (\Int(H_i)\cap\bar B_R^{\infty})$. Identity \eqref{eq:d_mu_H_i} and estimate 
\eqref{eq:Sup_est_sol_2_tran_prob} give us that there is a positive constant, 
$C=C(L,t_0,t,r,R)$, such that
%Updated from here on with the new boundary estimates
\begin{align*}
|k_{x_i}(s,p,(\widehat x_i,y))| \,d\mu_{H_i}(\widehat x_i,y) 
&\leq C \prod_{\stackrel{j=1}{j\neq i}}^{n_0} x_jx_j^{-1}\,dx_j
\prod_{k=n_0+1}^n x_k^{b_k-1}\,dx_k\prod_{l=1}^mdy_l\\
&\leq C \prod_{k=n_0+1}^n x_k^{b_k-1}\,dx_k 
\prod_{\stackrel{j=1}{j\neq i}}^{n_0}\,dx_j
\prod_{l=1}^mdy_l,
\end{align*}
for all $s\in [t_0,t]$, $p\in P\backslash \partial^T P$, $(\widehat x_i,y)\in\Int(H_i)\cap\bar B^{\infty}_R$, and $1\leq i\leq n_0$. It is easy to see that the right-hand side of the preceding inequality is integrable on $\Int(H_i)\cap\bar B^{\infty}_R$. This completes the proof of the fact that the right-hand side in identity \eqref{eq:Integral_representation} is well-defined.
\end{step}

\begin{step}[Approximation]
\label{step:Approximation}
In an adapted local system of coordinates, using identity
\eqref{eq:Parabolic_Dirichlet_rep_semigroup} and Remark
\ref{rmk:T_t_bounded_functions}, for $t>t_0,$ we can then write
\begin{equation}
\label{eq:u_kernel}
u(t,p) = -\int_0^t \int_{S_{n,m}} k(t-s,p,(x,y))(\partial_s-L)\zeta(s,(x,y))\,d\mu(x,y)\, ds,
\end{equation}
where we use the fact that $\zeta(t,\cdot)=0$, for all $t\geq t_0$ by \eqref{defn:time_t_0}. We recall that the measure $d\mu$ can be represented as in \eqref{eq:Weight_local}. The fact that the right-hand side in the preceding identity is well-defined can be proved using a similar argument of the one described in Step \ref{step:RHS_welldefined}, in which we replace the appeal to estimate 
\eqref{eq:Sup_est_sol_2_tran_prob} to that of estimate \eqref{eq:Sup_est_sol_1_tran_prob} (with $\fb=0$) and we replaced the measure $d\mu_{H_i}$ by $d\mu$. For all $\eps>0$ and $1\leq i\leq n$, we denote 
$$
S^{\eps}_{n,m}:=(\eps,\infty)^n\times\RR^m
\quad\hbox{ and }\quad
H^{\eps}_i:=\partial S^{\eps}_{n,m}\cap\{x_i=\eps\},
$$
and we set
$$
u^{\eps}(t,p) := -\int_0^t \int_{S^{\eps}_{n,m}} k(t-s,p,(x,y))(\partial_s-L)\zeta(s,(x,y))\,d\mu(x,y)\, ds,
$$
which converges to $u(t,p)$, as $\eps$ tends to 0, by the Dominated Convergence
Theorem since the right-hand side in \eqref{eq:u_kernel} is an integrable function. We write $u^{\eps}(t,p)=I^{\eps}_1+I^{\eps}_2$, where we define
\begin{align*}
I^{\eps}_1&:=-\int_0^t \int_{S^{\eps}_{n,m}} k(t-s,p,(x,y))\partial_s\zeta(s,(x,y))\,d\mu(x,y)\, ds,\\
I^{\eps}_2&:= \int_0^t \int_{S^{\eps}_{n,m}} k(t-s,p,(x,y))L\zeta(s,(x,y))\,d\mu(x,y)\, ds.
\end{align*}
Using Lemma \ref{lem:Equation_Dirichlet_heat_kernel} and the fact that the adjoint operator $\widehat L$ is hypo-elliptic on 
$\Int(P)$, we have that $k(\cdot,p,\cdot)$ is a smooth function on $((0,\infty)\times \bar S^{\eps}_{n,m})\cap \hbox{supp }\zeta$, and so we can integrate by parts in the preceding integrals. Using the fact that that $\zeta(0)=\zeta(t)=0$, for all $t\geq t_0$, we obtain
\begin{equation}
\label{eq:I_1}
I^{\eps}_1:=-\int_0^t \int_{S^{\eps}_{n,m}} \partial_t k(t-s,p,(x,y))\zeta(s,(x,y))\,d\mu(x,y)\, ds.
\end{equation}
Integration by parts in the term $I^{\eps}_2$ gives us
\begin{equation}
\label{eq:I_2}
I^{\eps}_2:=\sum_{i=1}^n J^{\eps}_i 
-\int_0^t \int_{S^{\eps}_{n,m}} 
\widehat L k(t-s,p,(x,y))\zeta(s,(x,y))\,d\mu(x,y)\, ds,
\end{equation}
where the boundary terms $J^{\eps}_i$, for $1\leq i\leq n$, have the following expressions
\begin{equation*}
\begin{aligned}
J^{\eps}_i
&:= 
-\eps^{b_i}\int_0^t\int_{H^{\eps}_i}(1+x_ia_{ii})k(t-s,p,(x,y))\zeta_{x_i}(s,(x,y)) 
\,d\mu_{H_i}(\widehat x_i,y)\,ds
\\
&\quad 
+ \eps^{b_i} \int_0^t\int_{H^{\eps}_i}
\left(a_{ii}+x_i\partial_{x_i}a_{ii}+\sum_{j=1}^n(b_ja_{ij}+x_j\partial_{x_j}a_{ij})\right) 
k(t-s,p,(x,y))\zeta(s,(x,y))
\,d\mu_{H_i}(\widehat x_i,y)\,ds
\\
&\quad 
+ \eps^{b_i} \sum_{j\neq i}\int_0^t\int_{H^{\eps}_i} x_ja_{ij} k(t-s,p,(x,y))\zeta_{x_j}(s,(x,y))
\,d\mu_{H_i}(\widehat x_i,y)\,ds
\\
&\quad 
+ \eps^{b_i} \sum_{l=1}^m\int_0^t\int_{H^{\eps}_i} c_{il} k(t-s,p,(x,y)))\zeta_{y_l}(s,(x,y))
\,d\mu_{H_i}(\widehat x_i,y)\,ds
\\
&\quad 
+ \eps^{b_i} \sum_{j\neq i}\int_0^t\int_{H^{\eps}_i} x_ja_{ij} k_{x_j}(t-s,p,(x,y))\zeta(s,(x,y))
\,d\mu_{H_i}(\widehat x_i,y)\,ds
\\
&\quad 
+ \eps^{b_i} \int_0^t\int_{H^{\eps}_i} (1+x_ia_{ii})k_{x_i}(t-s,p,(x,y))\zeta(s,(x,y)) 
\,d\mu_{H_i}(\widehat x_i,y)\,ds.
\end{aligned}
\end{equation*}
Combining identities \eqref{eq:I_1} and \eqref{eq:I_2} together with Lemma \ref{lem:Equation_Dirichlet_heat_kernel} it follows that 
\begin{equation}
\label{eq:u_eps_approx}
u^{\eps}(t,p) = \sum_{i=1}^n J^{\eps}_i.
\end{equation}
It remains prove that $u^{\eps}(t,p)$ converges, as $\eps\rightarrow 0$, to the expression of $u(t,p)$ in 
\eqref{eq:Integral_representation}.
\end{step}

\begin{step}[Convergence]
\label{step:Convergence}
For all $1\leq i\leq n$, we denote each line in the expression of $J^{\eps}_i$ by 
%Replaced index j by i' because j was also used with a different purpose
$J^{\eps}_{i,i'}$, for $1\leq i'\leq 6$. Using the expression of the measure $d\mu_{H_i}$ in \eqref{eq:d_mu_H_i} and applying estimate 
%Replaced reference
\eqref{eq:Sup_est_sol_1_tran_prob} (with $\fb=0$), there is a positive constant, $C=C(L,t_0,t,r,R)$, such that for all $1\leq i'\leq 4$, we have that
%Updated proof with the new boundary estimates from here on
\begin{align*}
\sum_{i'=1}^4 |J^{\eps}_{i,i'}| 
&\leq C \eps^{b_i} \int_{H^{\eps}_i\cap B^{\infty}_r} x_i^{-b_i+1} \prod_{j=1}^{n_0} x_jx_j^{-1} 
\prod_{k=n_0+1}^n  x_k^{b_k-1}
\prod_{\stackrel{l=1}{l\neq i}}^n dx_l \,dy
\\
&\leq C \eps^{b_i}\int_{H^{\eps}_i\cap B^{\infty}_r} x_i^{-b_i+1}
\prod_{k=n_0+1}^n  x_k^{b_k-1}
\prod_{\stackrel{l=1}{l\neq i}}^n dx_l \,dy.
\end{align*}
The preceding sum of four terms is bounded by $C\eps$, for all $1\leq i\leq n_0$, and by $C\eps^{b_i}$, for all $n_0+1\leq i\leq n$. Thus, we have shown that
\begin{equation}
\label{eq:First_4_terms}
\sum_{i'=1}^4 |J^{\eps}_{i,i'}|  \leq C\left(\eps + \eps^{b_i} \right),\quad\forall\, 1\leq i\leq n.
\end{equation}
To bound the term $J^{\eps}_{i,5}$, we apply estimate \eqref{eq:Sup_est_sol_2_tran_prob}, when $1\leq j\leq n_0$ and $j\neq i$, and  estimate \eqref{eq:Sup_est_sol_1_tran_prob}, when $n_0+1\leq j\leq n$ and $j\neq i$, and we obtain, by employing the same argument used to prove \eqref{eq:First_4_terms}, that
\begin{equation}
\label{eq:Fifth_term}
|J^{\eps}_{i,5}|
\leq C \left(\eps + \eps^{b_i}\right),\quad\forall\, 1\leq i\leq n.
\end{equation}
Next we consider the term $J^{\eps}_{i,6}$ and we see that estimate \eqref{eq:Sup_est_sol_1_tran_prob} gives us for all 
$n_0+1\leq i\leq n$ that
\begin{equation}
\label{eq:Sixth_term_after_n_0}
|J^{\eps}_{i,6}|
\leq C \left(\eps + \eps^{b_i}\right),\quad\forall\, n_0+1\leq i\leq n,
\end{equation}
while the argument employed in the proof of Step \ref{step:RHS_welldefined}
gives us, for $1\leq i\leq n_0,$ that
$$
J^{\eps}_{i,6}\rightarrow 
\int_0^t\int_{\Int(H_i)} k_{x_i}(t-s,p,(\widehat x_i,y)) \zeta(s,(\widehat x_i,y))
\, d\mu_{H_i}(\widehat x_i,y)\, ds,
\quad\hbox{ as }
\eps\rightarrow 0.
$$
The preceding property combined with estimates \eqref{eq:First_4_terms}, \eqref{eq:Fifth_term}, \eqref{eq:Sixth_term_after_n_0}, and identity \eqref{eq:u_eps_approx} give us that identity \eqref{eq:Integral_representation} indeed holds.
\end{step}
This completes the proof.
\end{proof}

With the aid of Proposition \ref{prop:Integral_representation}, we can now give

\begin{proof}[Proof of Theorem \ref{thm:Hitting_distribution_hypers}]
Let $i\in I^T$ and $q\in \Int(H_i)$. We choose an adapted system of coordinates such that we identify $q$ with the origin in 
$S_{1, N-1}$ and we write the generalized Kimura operator in the form
\eqref{eq:Operator} with $n=1$ and $m=N-1$ on $B^{\infty}_R$. We notice that because
$i\in I^T$, the weights of the generalized Kimura operator are constant on
$B^{\infty}_R.$ Fixing notation so that $i\leftrightarrow 1,$ we have $b_1\equiv 0,$ and so we can apply Proposition \ref{prop:Integral_representation}. Let $0<r<R$ and $\zeta$ be a smooth function with compact support in 
$(0,\infty)\times (B^{\infty}_r\cup\{x_1=0\})$. Let $u$ be the unique solution to the non-homogeneous parabolic Dirichlet problem 
\eqref{eq:Parabolic_Dirichlet} with boundary data $\zeta$. Let $t_0$ be the positive constant defined in \eqref{defn:time_t_0}. By Proposition \ref{prop:Integral_representation}, for all $t>2t_0$, we obtain the integral representation 
\eqref{eq:Integral_representation} of the solution $u$, and from Proposition \ref{thm:Parabolic_Dirichlet}, we obtain the stochastic representation \eqref{eq:Parabolic_Dirichlet_rep_stoch} of the solution $u$, which yield the identity
\begin{align*}
&\int_0^t \int_{\{x_1=0\}} \zeta(t-s,(0,y)) 
\QQ^p\left(\tau_{\partial^T P}\in(s,s+ds], \omega(\tau_{\partial^T P})\in \prod_{l=1}^{N-1}(y_l,y_l+dy_l]\right)\\
&\qquad=\int_0^t\int_{\{x_1=0\}} k_{x_1}(s,p,(0,y))\zeta(t-s,(0,y)) d\mu_{\{x_1=0\}}(0,y)\,ds,
\end{align*}
for all $(t,p)\in (2t_0,\infty)\times P\backslash \partial^T P$. Because the smooth test function $\zeta$ with compact support in 
$(0,\infty)\times (B^{\infty}_r\cup\{x_1=0\})$ is arbitrarily chosen, the preceding identity implies the existence of a density functions
$h_i$ as in \eqref{eq:Hitting_density_H_i} such that property \eqref{eq:Hitting_H_i} holds. 
The regularity property \eqref{eq:h_i_reg} is a direct consequence of Theorem \ref{thm:Sup_est_tran_prob}.
This completes the proof.
\end{proof}

Our goal is next to prove Theorem \ref{thm:Hitting_corner} with the aid of several auxiliary results. We begin with the proof of

\begin{lem}[Hitting probability of two tangent hypersurfaces]
\label{lem:Hitting_corner_tangent_tangent}
Suppose that the generalized Kimura operator satisfies the standard
assumptions. Let $i, j\in I^T$, and $q \in H_i\cap H_j$ be such that
there is a positive constant $r$ such that
\begin{equation}
\label{eq:Cond_tangent_tangent}
B_r(q) \cap \partial P \subseteq \partial^c P.
\end{equation}
Then, for all points $p\in P\backslash \partial^T P$, we have that 
\begin{equation}
\label{eq:Hitting_tangent_tangent}
\QQ^p(\omega(\tau_{\partial^T P}) \in H_i\cap H_j \cap B_{r/2}(q)) = 0.
\end{equation}
\end{lem}

\begin{proof}
Let $k\in\NN$ and let $\varphi_k\in C^{\infty}(P)$ be such that
\begin{align*}
\varphi_k = 1&\quad\hbox{ on }\quad \{w\in P: \hbox{ dist}(w, H_i\cap H_j) \leq 1/k\} \cap B_{r/2}(q),\\
\varphi_k = 0&\quad\hbox{ on }\quad \{w\in P: \hbox{ dist}(w, H_i\cap H_j) \geq 2/k\} \cup B^c_{r/2+1/k}(q),
\end{align*}
where the distance in the preceding definition is with respect to the Riemannian metric induced on the manifold $P$ by the generalized Kimura operator, and $B^c_s(q)$ denotes the complement of $B_s(q)$ in $P$. Let $T>0$ and let $\psi_k\in C^{\infty}([0,\infty))$ be such that 
$$
\psi_k(t) = 1\quad\hbox{ on } [1/k, T],
\quad\hbox{ and }\quad 
\psi_k(t) = 0\quad\hbox{ on } [0,1/(2k)]\cup[T+1/k,\infty).
$$
Let $u_k$ be the unique solution to the non-homogeneous parabolic Dirichlet problem 
\eqref{eq:Parabolic_Dirichlet} with boundary condition $\zeta_k(t,p) :=\psi_k(t)\varphi_k(p)$. Identity 
\eqref{eq:Parabolic_Dirichlet_rep_stoch} gives us that
$$
u_k(t,p) = \EE_{\QQ^p}\left[\zeta_k(t-t\wedge\tau_{\partial^T P},\omega(t\wedge\tau_{\partial^T P}))\right],
\quad\forall\,(t,p)\in(0,\infty)\times P,
$$
Condition \eqref{eq:Cond_tangent_tangent} allows us to apply Proposition \ref{prop:Integral_representation}. Using the fact that the function $\zeta_k$ converges to 0 a.e. on $\partial P$ with respect to the Lebesgue measure on $\partial P$, as $k \rightarrow\infty$, we obtain from the integral representation 
\eqref{eq:Integral_representation} that $u_k(t,p)$ tends to 0, as $k\rightarrow\infty$, for all 
$(t,p) \in (2T,\infty)\times P\backslash \partial^T P$. From the construction of the sequence of functions $\{\zeta_k\}_{k\in\NN}$, the right-hand side of the preceding identity converges to 
$$
\QQ^p(t-T\leq\tau_{\partial^T P} \leq t,\, \omega(\tau_{\partial^T P})\in H_i\cap H_j\cap B_{r/2}(q)),
$$
as $k\rightarrow\infty$, for all $(t,p) \in (2T,\infty)\times P\backslash \partial^T P$. Because $T$ can be chosen arbitrarily small and $t$ can be chosen arbitrarily large such that $t>2T$, the preceding identity implies property 
\eqref{eq:Hitting_tangent_tangent}. This completes the proof.
\end{proof}

We next prove

\begin{lem}[Hitting probability of a tangent and a transverse hypersurface]
\label{lem:Hitting_corner_tangent_transverse}
Suppose that the generalized Kimura operator satisfies the standard assumptions,
then, for all points $p\in P\backslash \partial^T P$, $i\in I^T$, and $j \in I^{\pitchfork}$, we have that property 
\eqref{eq:Hitting_corner} holds.
\end{lem}

\begin{proof}
Lemma \ref{lem:Regularity_hitting_probability} implies that the function $u$ defined by \eqref{eq:Function_hitting_corner}:
\begin{equation}
\label{eq:Function_hitting_corner-2}
u(p) := \QQ^p(\omega(\tau_{\partial^T P}) \in H_i\cap H_j),\quad\forall\, p \in P, 
\end{equation}
  locally satisfies the hypotheses of Lemma
 \ref{lem:Growth_lemma}. Thus, a standard covering argument combined
 with Lemma \ref{lem:Growth_lemma} implies that there is a constant,
 $\theta\in (0,1)$, such that
$$
\sup\{u(p): p\in P,\, \hbox{dist}(p,H_i\cap H_j) <r/2\} \leq \theta \sup\{u(p): p\in P,\, \hbox{dist}(p,H_i\cap H_j) <r\},
$$
for all $r>0$ small enough, where the distance is computed with respect to the Riemannian metric induced by the generalized Kimura operator on the manifold $P$. Because the function $u$ defines the hitting probability of the Kimura diffusion on the boundary component $H_i\cap H_j$, we clearly have that
$$
\sup\{u(p): p\in P,\, \hbox{dist}(p,H_i\cap H_j) <r/2\} \geq \sup\{u(p): p\in P,\, \hbox{dist}(p,H_i\cap H_j) <r\}.
$$
The preceding two identities and the fact that $\theta<1$ show that $u(p)=0$, for all $p\in P\backslash(H_i\cap H_j)$, which implies property \eqref{eq:Hitting_corner}. This completes the proof.
\end{proof}

We can now give

\begin{proof}[Proof of Theorem \ref{thm:Hitting_corner}]
When $j\in I^{\pitchfork}$, the conclusion follows from Lemma \ref{lem:Hitting_corner_tangent_transverse}. We next consider the case when $j \in I^T$. Because the set $H_i\cap H_j$ is compact, it is sufficient to prove that for all $q \in H_i\cap H_j$, there is a positive constant $r$ such that property \eqref{eq:Hitting_tangent_tangent} holds. If $q \in \Int(H_i\cap H_j)$, then there is $r>0$ such that the generalized Kimura operator has zero weights on $B_r(q)$, and so the hypotheses of Lemma \ref{lem:Hitting_corner_tangent_tangent} are satisfied, which implies that property \eqref{eq:Hitting_tangent_tangent} holds. When 
$q \in (H_i\cap H_j)\backslash \Int(H_i\cap H_j)$, exactly one of the following two can hold:
\begin{align}
\label{eq:all_tangent}
&\hbox{ for all } k \in \{1,\ldots,\eta\} \hbox{ such that } q \in H_k, \hbox{ then } k \in I^T,\\
\label{eq:at_least_one_transverse}
&\hbox{ there is } k \in \{1,\ldots,\eta\} \hbox{ such that } q \in H_k, \hbox{ and } k \in I^{\pitchfork}.
\end{align}
When property \eqref{eq:all_tangent} holds, we again are in the situation when there is $r>0$ such that the generalized Kimura operator has zero weights on $B_r(q)$, and so we can apply Lemma \ref{lem:Hitting_corner_tangent_tangent} to obtain that \eqref{eq:Hitting_tangent_tangent} holds. When property \eqref{eq:at_least_one_transverse} holds, we can apply Lemma \ref{lem:Hitting_corner_tangent_transverse} to $i\in I^T$ and $k\in I^{\pitchfork}$ to obtain the conclusion. This completes the proof of Theorem \ref{thm:Hitting_corner}.
\end{proof}

%Added
We conclude this section with the proof of the doubling property of the hitting distributions (caloric measure). This is the extension to the degenerate framework of generalized Kimura operators of the doubling property satisfied by the caloric measure of parabolic equations defined by strictly elliptic operators, \cite[Theorem 1.1]{Safonov_Yuan_1999}, \cite[Theorem 2.4]{Fabes_Garofalo_Salsa_1986}. Our method of the proof is based on the  boundary Harnack principle established in \cite{Epstein_Pop_2016}. For all $(t,p)\in (0,\infty)\times P$ and for all $r>0$, we denote
\begin{align}
\label{eq:Cylinder}
C_r(t,p) &:= (t-r^2, t+r^2)\times B_r(p),\\
\label{eq:Boundary_cylinder}
\Delta_r(t,p) &:= C_r(t,p) \cap \partial^T P.
\end{align}

We can now state
\begin{thm}[Doubling property of hitting distributions]
\label{thm:Doubling_property}
Let $\delta\in (0,1)$ and $T>0$. Then there is a positive constant, $D=D(\delta, T)$, such that for all $q\in \partial^TP$ such that
\begin{equation}
\label{eq:Interior_tangent_boundary}
\hbox{dist}_{\partial P}(q,\partial^{\pitchfork} P) >\delta,
\footnote{The distance is computed with respect to the intrinsic Riemannian metric}
\end{equation}
we have that, for all $0<r<\delta/8$, $16\delta^2<t<T-16r^2$, and $p\in P\backslash\partial^TP$,
\begin{equation}
\label{eq:Doubling_property}
\QQ^p(\omega(\tau_{\partial^T P}) \in \Delta_{2r}(t,q)) \leq D\QQ^p(\omega(\tau_{\partial^T P}) \in \Delta_r(t,q)).
\end{equation}
\end{thm}

\begin{proof}
Without loss of generality we can assume that the positive constant $\delta$ is small enough such that in an adapted local system of coordinates, we can identify $q\in\partial^TP$ with the origin in $\bar S_{n,m}$, and the generalized Kimura operator takes the form of the operator $L$ in \eqref{eq:Operator}. We can also assume that the weights of the operator $L$ satisfy \eqref{eq:n_0}, where $R$ is chosen large enough. Applying Proposition \ref{prop:Integral_representation} we can write
\begin{equation}
\label{eq:Integral_cylinder_boundary}
\QQ^p(\omega(\tau_{\partial^T P}) \in \Delta_{2r}(t,q))
=\sum_{i=1}^{n_0} \int_{\Delta_r(t,O)\cap H_i} k_{x_i}(s,p,z)\, d\mu_{H_i}(z)\, ds,
\end{equation}
where we denote
$$
H_i=\{x_i=0\}\cap\partial S_{n,m}
\quad\hbox{ and }\quad 
d\mu_{H_i}(z)=\prod_{\stackrel{j=1}{j\neq i}}^{n_0} x_j^{-1}\prod_{l=n_0+1}^n x_l^{b_l(z)-1},\quad\forall\, 1\leq i\leq n_0.
$$ 
Let $r_0:=\delta/6;$ define
  $\fe\in\RR^{n+m}$ as the vector having all coordinates equal to $0$ except for the first $n$ coordinates, which are equal to $1$, and 
	$\ff\in\RR^{n+m}$ as the vector having all coordinates equal to $0$ except for the last $m$ coordinates, which are equal to $1$, and 
$$
A_r(O):=:=\frac{r^2}{4n}\fe + \frac{r}{2\sqrt{m}}\ff
\quad\hbox{ and }\quad
w^T(z) := \prod_{j=1}^{n_0}x_j^{-1}.
$$
Applying \cite[Estimates (1.17) and (1.18)]{Epstein_Pop_2016} and Lemma \ref{lem:Equation_Dirichlet_heat_kernel} (i), it follows that there is a positive constant $H$ with the property that
\begin{align*}
k_{x_i}(t,p,z) &\leq H w^T(A_{r_0}(O)) k(t+4r_0^2,p, A_{r_0}(O)) \prod_{\stackrel{j=1}{j\neq i}}^{n_0} x_j,
\quad\forall\, (t,z)\in\Delta_{2r}(t,O),\\
k_{x_i}(t,p,z) &\geq H w^T(A_{r_0}(O)) k(t-2r_0^2,p, A_{r_0}(O)) \prod_{\stackrel{j=1}{j\neq i}}^{n_0} x_j,
\quad\forall\, (t,z)\in\Delta_{r}(t,O),
\end{align*}
which gives us that
\begin{align*}
&\int_{\Delta_{2r}(t,O)\cap H_i} k_{x_i}(s,p,z)\, d\mu_{H_i}(z)\, ds \\
&\qquad\qquad\leq H w^T(A_{r_0}(O)) k(t+4r_0^2,p,A_{r_0}(O)) 
\int_{\Delta_{2r}(t,O)\cap H_i} \prod_{l=n_0+1}^n x_l^{b_l(z)-1},\\
&\int_{\Delta_{r}(t,O)\cap H_i} k_{x_i}(s,p,z)\, d\mu_{H_i}(z)\, ds \\
&\qquad\qquad\geq H w^T(A_{r_0}(O)) k(t-2r_0^2,p,A_{r_0}(O)) 
\int_{\Delta_{r}(t,O)\cap H_i} \prod_{l=n_0+1}^n x_l^{b_l(z)-1}\,dz\,ds.
\end{align*}
By \cite[Proposition 3.1]{Epstein_Mazzeo_2016}, the measure $\prod_{l=n_0+1}^n x_l^{b_l(z)-1}\,dz$ satisfies the doubling property, and so there is a positive constant, $D_0$, such that 
$$
\int_{\Delta_{2r}(t,O)\cap H_i} \prod_{l=n_0+1}^n x_l^{b_l(z)-1}\,dz\,ds
\leq
D_0\int_{\Delta_{r}(t,O)\cap H_i} \prod_{l=n_0+1}^n x_l^{b_l(z)-1}\,dz\,ds.
$$
Combining the preceding three inequalities it follows that
\begin{align*}
&\int_{\Delta_{2r}(t,O)\cap H_i} k_{x_i}(s,p,z)\, d\mu_{H_i}(z)\, ds \\
&\qquad\qquad\leq D_0 \frac{ k(t+4r_0^2,p,A_{r_0}(O))}{k(t-2r_0^2,p,A_{r_0}(O))}
\int_{\Delta_{r}(t,O)\cap H_i} k_{x_i}(s,p,z)\, d\mu_{H_i}(z)\, ds. 
\end{align*}
Lemma \ref{lem:Equation_Dirichlet_heat_kernel} (ii) and the elliptic-type Harnack inequality \cite[Lemma 4.5]{Epstein_Pop_2016}, gives us that there is a positive constant, $H_0$, such that
$$
\frac{ k(t+4r_0^2,p,A_{r_0}(O))}{k(t-2r_0^2,p,A_{r_0}(O))} \leq H_0,\quad\forall\, p \in P\backslash\partial^T P,\quad\forall\, t>16\delta^2.
$$
The preceding two inequalities and identity \eqref{eq:Integral_cylinder_boundary} imply the doubling property \eqref{eq:Doubling_property}. This concludes the proof.
\end{proof}

\subsection{Structure of the transition probabilities}
\label{sec:Proof_main_result}
Finally, we give the proof of the main result in our article:

\begin{proof}[Proof of Theorem \ref{thm:Tran_prob}]
Let $f\in C^{\infty}(P)$, $p\in P\backslash\partial^T P$, and let $\QQ^p$ be the probability measure constructed in Theorem 
\ref{thm:Wellposed_mart_problem}. Using identity \eqref{eq:tran_prob}, we have that
\begin{equation}
\label{eq:Tran_prob_with_test_function}
\int_P f(w) \Gamma^P(t,p,dw) = \EE_{\QQ^p}\left[f(\omega(t))\right],\quad\forall\, t\geq 0.
\end{equation}
The right-hand side of the preceding identity can be written as the sum of two terms as follows:
\begin{align}
\EE_{\QQ^p}\left[f(\omega(t))\right] &= \EE_{\QQ^p}\left[f(\omega(t))\mathbf{1}_{\{t<\tau_{\partial^T P}\}}\right]
+ \EE_{\QQ^p}\left[f(\omega(t))\mathbf{1}_{\{t\geq\tau_{\partial^T P}\}}\right]
\notag\\
&= \EE_{\QQ^p}\left[f(\omega(t))\mathbf{1}_{\{t<\tau_{\partial^T P}\}}\right]
+ \EE_{\QQ^p}\left[\mathbf{1}_{\{t\geq\tau_{\partial^T P}\}}
\EE_{\QQ^p}\left[f(\omega(t)) \big{|}\cF_{\tau_{\partial^T P}}\right]\right]
\notag\\
\label{eq:Tran_prob_with_test_function_decomposition}
&= \EE_{\QQ^p}\left[f(\omega(t))\mathbf{1}_{\{t<\tau_{\partial^T P}\}}\right]
+ \EE_{\QQ^p}\left[\mathbf{1}_{\{t\geq\tau_{\partial^T P}\}}
\EE_{\QQ^p}\left[f(\omega(t)) \big{|}\omega(\tau_{\partial^T P})\right]\right],
\end{align}
where in the last line we used the strong Markov property of the canonical
process established in Corollary \ref{cor:Strong_Markov}. We recall the
definition of the stopping time $\tau_{\partial^T P}$ in \eqref{eq:tau_tangent}
and of the filtration $\{\cF_t\}_{t\geq 0}$ in Definition \ref{defn:Martingale_problem}. Identity \eqref{eq:T_t_bounded_functions_nonabsorbed} yields
\begin{equation}
\label{eq:First_int_tran_prob}
\EE_{\QQ^p}\left[f(\omega(t))\mathbf{1}_{\{t<\tau_{\partial^T P}\}}\right] = \int_{\Int(P)}f(w) k^P(t,p,w)\,d\mu_{P}(w).
\end{equation}
Here we denote  by $k^P(\cdot,p,\cdot)$ the Dirichlet heat kernel constructed in Theorem \ref{thm:Distribution_absorbed}. Thus, we obtain that
\begin{align*}
\EE_{\QQ^p}\left[f(\omega(t))\right] = \int_{\Int(P)}f(w) k^P(t,p,w)\,d\mu_{P}(w)
+ \EE_{\QQ^p}\left[\mathbf{1}_{\{t\geq\tau_{\partial^T P}\}}
\EE_{\QQ^p}\left[f(\omega(t)) \big{|}\omega(\tau_{\partial^T P})\right]\right].
\end{align*}
We prove Theorem \ref{thm:Tran_prob} by induction on the dimension $N$ of the compact manifold $P$ with corners. We divide the proof into two steps. In Step \ref{step:TP_base_case} we establish that equality \eqref{eq:Tran_prob} holds when $N=1$, and in Step \ref{step:TP_induction_step} we prove the induction step.

\setcounter{step}{0}
\begin{step}[Base case]
\label{step:TP_base_case}
When $N=1$, the manifold $P$ is a smooth connected curve with two endpoints
(which we view as sets), $P^0_1$ and $P^0_2,$ We start by assuming that $d_0=2$. Using the identity,
$$
\EE_{\QQ^p}\left[f(\omega(t)) \big{|}\omega(\tau_{\partial^T P})\right] 
= \sum_{i=1}^{d_0} \EE_{\QQ^p}\left[f(\omega(t)) \big{|}\omega(\tau_{\partial^T P})\right] 
\mathbf{1}_{\{\omega(\tau_{\partial^T P}) \in P^0_i\}}, 
\quad \QQ^p\hbox{-a.s. on } \{t\geq \tau_{\partial^T P}\},
$$
and applying Lemma \ref{lem:Absorption_tangent_boundary} with $\Sigma=P^0_i$, for $i=1,2$, it follows that
$$
\EE_{\QQ^p}\left[f(\omega(t)) \big{|}\omega(\tau_{\partial^T P})\right] 
= \sum_{i=1}^{d_0} f(P^0_i) \mathbf{1}_{\{\omega(\tau_{\partial^T P}) \in P^0_i\}}, 
\quad \QQ^p\hbox{-a.s. on } \{t\geq\tau_{\partial^T P}\}.
$$
The preceding equality yields
\begin{align*}
\EE_{\QQ^p}\left[\mathbf{1}_{\{t\geq\tau_{\partial^T P}\}}\EE_{\QQ^p}\left[f(\omega(t)) 
\big{|}\omega(\tau_{\partial^T P})\right]\right]
&= \sum_{i=1}^{d_0} f(P^0_i) \QQ^p(\omega(t\wedge\tau_{\partial^T P}) \in P^0_i)\\
&= \sum_{i=1}^{d_0} f(P^0_i) p_{P^0_i}(t,p),
\end{align*}
where the function $p_{P^0_i}(t,p)$ is defined by \eqref{eq:Probability_hit_Sigma}, in which we choose $\Sigma=P^0_i$, for all $1\leq i\leq d_0$. Denoting $p_{P^0_i}(t,p)$ by $k^{P;P^0_i}(t,p,p_i)$, where $p_i$ denotes the point in $P^0_i$, the preceding identity together with equalities \eqref{eq:First_int_tran_prob}, \eqref{eq:Tran_prob_with_test_function_decomposition}, and \eqref{eq:Tran_prob_with_test_function} yield
\begin{align*}
\int_P f(w) \Gamma^P(t,p,dw) = \int_{\Int(P)} f(w)k^P(t,p,w)\, d\mu_P(w) 
+ \sum_{i=1}^{d_0} f(P^0_i) k^{P;P^0_i}(t,p,w)\delta_{P^0_i}(w),
\end{align*}
where we recall that $\delta_{P^0_i}(w) = 1$, if $w=p_i$, and
$\delta_{P^0_i}(w) = 0$, otherwise. Because the test function $f$ was
arbitrarily chosen in $C^{\infty}(P)$, the preceding identity implies
\eqref{eq:Tran_prob} when $N=1$. This completes the proof of Step
\ref{step:TP_base_case}, assuming that $d_0=2;$ the argument with $d_0=1$ is
quite similar and is left to the reader.
\end{step}

\begin{step}[Induction step]
\label{step:TP_induction_step}
We now assume that identity \eqref{eq:Tran_prob} holds for all connected
compact manifolds $P$ with corners of dimension $N-1$ and generalized Kimura
diffusion operators, satisfying the hypotheses of the theorem. We now prove
that it also holds for connected compact manifolds with corners, with
generalized Kimura diffusion operators, of dimension $N$. 

Let $P$ be a compact manifold with corners of dimension $N$. To simplify
notation for the remaining part of the proof, we introduce the abbreviations:
$$
\tau:= \tau_{\partial^T P},
\quad\hbox{ and }\quad
\tau_{k,i}:= \tau_{\partial^T P^k_i},\quad\forall\, 1\leq i\leq d_k,\quad\forall\, 1\leq k\leq N-1.
$$
When $N>1$, Theorem \ref{thm:Hitting_corner} gives us that
$$
\EE_{\QQ^p}\left[f(\omega(t)) \big{|}\omega(\tau)\right] 
= \sum_{i=1}^{d_{N-1}} \EE_{\QQ^p}\left[f(\omega(t)) \big{|}\omega(\tau_{N-1, i})\right] 
\mathbf{1}_{\{\omega(\tau) \in \Int(P^{N-1}_i)\}}, \quad \QQ^p\hbox{-a.s. on } \{t\geq \tau\},
$$
while Lemma \ref{lem:Absorption_tangent_boundary} implies that
$$
\EE_{\QQ^p}\left[f(\omega(t)) \big{|}\omega(\tau_{N-1,i})\right] 
= \EE_{\QQ^p}\left[f\restriction_{P^{N-1}_i}(\omega(t)) \big{|}\omega(\tau_{N-1, i})\right], 
$$
for all $1\leq i\leq d_{N-1}$. Combining the preceding two identities we obtain that
\begin{align*}
\EE_{\QQ^p}\left[\mathbf{1}_{\{t\geq\tau\}}\EE_{\QQ^p}\left[f(\omega(t)) \big{|}\omega(\tau)\right]\right]
&= \sum_{i=1}^{d_{N-1}} \EE_{\QQ^p}\left[\mathbf{1}_{\{t\geq \tau_{N-1,i}\}} 
\EE_{\QQ^p} \left[f\restriction_{P^{N-1}_i}(\omega(t)) \big{|}\omega(\tau_{N-1,i})\right]
\right].
\end{align*}
Theorem \ref{thm:Hitting_distribution_hypers} shows that there is a non-negative, measurable function,
$$
h_i(\cdot,p,\cdot):[0,\infty)\times P^{N-1}_i\rightarrow [0,\infty),\quad\forall\, 1\leq i\leq d_{N-1},
$$
such that we can express the right-hand side of the preceding equality in integral form as:
\begin{equation}
\label{eq:Exp_1}
\begin{aligned}
&\EE_{\QQ^p}\left[\mathbf{1}_{\{t\geq \tau_{N-1, i}\}} 
\EE_{\QQ^p}\left[f\restriction_{P^{N-1}_i}(\omega(t)) \big{|}\omega(\tau_{N-1, i})\right]\right] \\
&\quad\quad
= \int_0^t\int_{\Int(P^{N-1}_i)} 
h_i(s,p,w) 
\EE_{\QQ^w}\left[f\restriction_{P^{N-1}_i}(\omega(t-s))\right]\,d\mu_{P^{N-1}_i}(w)\, ds,
\end{aligned}
\end{equation}
where we used once again the strong Markov property of the canonical process under the measure $\QQ^p$ established in Corollary \ref{cor:Strong_Markov}. Because the operator $L$ is tangent to the boundary hypersurface $P^{N-1}_i$, it follows from Lemma 
\ref{lem:Absorption_tangent_boundary} that the process $\{\omega(t)\}$, for all $t\geq \tau_{N-1,i}$, has support in compact manifold $P^{N-1}_i$ with corners. Remarks \ref{rmk:Sato} and \ref{rmk:Sato_equation} imply that its law is the solution to the martingale problem associated to the restriction operator $L_{P^{N-1}_i}$, for all $1\leq i\leq d_{N-1}$. Because $P^{N-1}_i$ is a compact manifold with corners of dimension $N-1$, we can apply the induction hypothesis to conclude that, for all 
$w\in P^{N-1}\backslash \partial^T P^{N-1}_i$, we have that there are Borel measurable functions,
\begin{align*}
&k^{P^{N-1}_i}(\cdot,w,\cdot):(0,\infty)\times  P^{N-1}_i\rightarrow [0,\infty),\\
&k^{P^{N-1}_i;P^k_j}(\cdot,w,\cdot):(0,\infty)\times P^k_j\rightarrow [0,\infty),
\end{align*}
for $P^k_j$ with  $P^{N-1}_i\cap P^k_j \neq \emptyset$, for all $1\leq j\leq d_k$, $0\leq k\leq N-1$, 
and the following holds:
\begin{equation}
\label{eq:Exp_2}
\begin{aligned}
&\EE_{\QQ^w}\left[f\restriction_{P^{N-1}_i}(\omega(t-s))\right]\\
&\quad
= \int_{\Int(P^{N-1}_i)} f(\xi) k^{P^{N-1}_i}(t-s,w,\xi)\, d\mu_{P^{N-1}_i}(\xi)\\
&\quad\quad 
+ \sum_{k=1}^{N-2}\sum_{j=1}^{d_k}\mathbf{1}_{\{P^{N-1}_i\cap P^k_j\neq\emptyset\}}
\int_{\Int(P^k_j)} f(\xi) k^{P^{N-1}_i;P^k_j}(t-s,w,\xi)\, d\mu_{P^k_j}(\xi)\\
&\quad\quad 
+ \sum_{j=1}^{d_0} \mathbf{1}_{\{P^{N-1}_i\cap P^0_j\neq\emptyset\}} 
f(\xi)k^{P^{N-1}_i;P^0_j}(t-s,w,\xi) \,\delta_{P^0_j}(\xi).
\end{aligned}
\end{equation}
We introduce the following notation, for all $1\leq i\leq d_{N-1}$, $1\leq j \leq d_k$, and $1\leq k \leq N-2$:
\begin{align*}
k^{P;P^{N-1}_i}(t,p,\xi)&:= \int_0^t\int_{\Int(P^{N-1}_i)} h_i(s,p,w)k^{P^{N-1}_i}(t-s,w,\xi)\,d\mu_{P^{N-1}_i}(w)\, ds,\\
k^{P;P^{N-1}_i;P^k_j}(t,p,\xi)&:= \int_0^t\int_{\Int(P^{N-1}_i)} h_i(s,p,w)k^{P^{N-1}_i;P^k_j}(t-s,w,\xi)
\,d\mu_{P^{N-1}_i}(w)\, ds,\\
k^{P;P^{N-1}_i;P^0_j}(t,p,\xi)&:= \int_0^t\int_{\Int(P^{N-1}_i)} h_i(s,p,w)k^{P^{N-1}_i;P^0_j}(t-s,w,\xi)
\,d\mu_{P^{N-1}_i}(w)\, ds\,\delta_{P^0_j}(\xi).
\end{align*}
Combining identities \eqref{eq:Exp_1} and \eqref{eq:Exp_2}, we obtain
\begin{align*}
&\EE_{\QQ^p}
\left[\mathbf{1}_{\{t\geq \tau_{N-1,i}\}} 
\EE_{\QQ^p}\left[f\restriction_{P^{N-1}_i}(\omega(t)) \big{|}\omega(\tau_{N-1,i})\right]\right] \\
&\quad
= \int_{\Int(P^{N-1}_i)} f(\xi) k^{P;P^{N-1}_i}(t,p,\xi) \, d\mu_{P^{N-1}_i}(\xi)\\
&\quad\quad 
+ \sum_{k=1}^{N-2}\sum_{j=1}^{d_k}\mathbf{1}_{\{P^{N-1}_i\cap P^k_j\neq\emptyset\}}
\int_{\Int(P^k_j)} f(\xi)  k^{P;P^{N-1}_i;P^k_j}(t,p,\xi)\, d\mu_{P^k_j}(\xi)\\
&\quad\quad 
+ \sum_{j=1}^{d_0} \mathbf{1}_{\{P^{N-1}_i\cap P^0_j\neq\emptyset\}} f(\xi) k^{P;P^{N-1}_i;P^0_j}(t,p,\xi)\,\delta_{P^0_j}(\xi).
\end{align*}
The preceding identity together with \eqref{eq:First_int_tran_prob}, \eqref{eq:Tran_prob_with_test_function_decomposition} and \eqref{eq:Tran_prob_with_test_function} give us
\begin{align*}
&\int_P f(w) \Gamma^P(t,p,dw) = \int_{\Int(P)} f(w)k^P(t,p,w)\, d\mu_P(w)\\
&\quad\quad
+ \sum_{i=1}^{d_{N-1}}\int_{\Int(P^{N-1}_i)} f(w) k^{P;P^{N-1}_i}(t,p,w) \, d\mu_{P^{N-1}_i}(w)\\
&\quad\quad 
+ \sum_{i=1}^{d_{N-1}}\sum_{k=1}^{N-2}\sum_{j=1}^{d_k}\mathbf{1}_{\{P^{N-1}_i\cap P^k_j\neq\emptyset\}}
\int_{\Int(P^k_j)} 
f(w)  k^{P;P^{N-1}_i;P^k_j}(t,p,w)\, d\mu_{P^k_j}(w)\\
&\quad\quad 
+ \sum_{i=1}^{d_{N-1}} \sum_{j=1}^{d_0} \mathbf{1}_{\{P^{N-1}_i\cap P^0_j\neq\emptyset\}} 
f(w) k^{P;P^{N-1}_i;P^0_j}(t,p,w) \delta_{P^0_i}(w).
\end{align*}
Letting now, for all $1\leq j\leq d_k$ and $1\leq k \leq N-2$,
\begin{align*}
k^{P;P^k_j}(t,p,w) &:= \sum_{i=1}^{d_{N-1}} \mathbf{1}_{\{P^{N-1}_i\cap P^k_j\neq\emptyset\}} k^{P;P^{N-1}_i;P^k_j}(t,p,w),\\
k^{P;P^0_j}(t,p,w) &:= \sum_{i=1}^{d_{N-1}} \mathbf{1}_{\{P^{N-1}_i\cap P^0_j\neq\emptyset\}} k^{P;P^{N-1}_i;P^0_j}(t,p,w),
\end{align*}
and using the fact that the test function $f$ was arbitrarily chosen in $C^{\infty}(P)$, the preceding identity implies \eqref{eq:Tran_prob}. This completes the proof of the induction step.
\end{step}

Combining Steps \ref{step:TP_base_case} and \ref{step:TP_induction_step}, we obtain that identity \eqref{eq:Tran_prob} holds, for all $N\geq 1$. This completes the proof.
\end{proof}

We conclude with the
\begin{proof}[Proof of Corollary \ref{cor:Zero_times_spent_on_trans_boundary}]
Property \eqref{eq:Zero_times_spent_on_trans_boundary} follows as soon as we prove that
\begin{equation}
\label{eq:Zero_times_spent_on_trans_boundary_aux}
\EE_{\QQ^p}\left[\int_0^T \mathbf{1}_{\{\omega(t)\in \partial^{\pitchfork P}\}}\, dt\right] = 0,\quad\forall\, p\in P,\quad\forall\, T>0.
\end{equation}
Theorem \ref{thm:Tran_prob} shows that the distribution of the generalized Kimura process is supported on 
$P\backslash\partial^{\pitchfork} P$, which is disjoint from the support of the function 
$\mathbf{1}_{\{\omega(t)\in \partial^{\pitchfork} P\}}$. This implies that property \eqref{eq:Zero_times_spent_on_trans_boundary_aux} holds and completes the proof. 
\end{proof}

\appendix

\section{Appendix}
\label{sec:Appendix}

We establish auxiliary results used in the proof of Lemma \ref{lem:Hitting_corner_tangent_transverse}. 
%CP 11.26.2016: M. Safonov's suggestion: ``In a Krylov's article, which is joint with another co-author, (see the attachment) the proof of Lemma 3.3 starts as follows: "We use the method of proof which we learned from M. Safonov" without including my name into References. Since Krylov was and  is my teacher, I suggest to use same pattern in the proof of Lemma A.3 in your article."
The first result is a Landis-type growth lemma which was proved by M.V. Safonov, \cite{Safonov_2016}, in the two-dimensional case (that is, we take $n=2$ and $m=0$ in the statement of Lemma \ref{lem:Growth_lemma}). Given $r,\mu>0$ and a bounded function $u$, we introduce the notation:
\begin{align}
\label{eq:M_max}
M(r;\mu) &:= \sup\{u(z):\, z\in \bar Q(r;\mu)\}\text{ where,}\\
\label{eq:Q_ball}
Q(r;\mu) &:= (0,\mu r)^n \times (-\mu\sqrt{r},\mu\sqrt{r})^m.
\end{align}
We consider a differential operator $A$ defined on $B^{\infty}_R\subset S_{n,m}$ by
\begin{equation}
\label{eq:A_operator}
\begin{aligned}
Au(z) 
&=\sum_{i,j=1}^n \sqrt{x_ix_j}a_{ij}(z) u_{x_ix_j} +\sum_{i=1}^n b_i(z)u_{x_i}+\sum_{i=1}^n \sum_{l=1}^m \sqrt{x_i}c_{il}(z)u_{x_iy_l}\\
&\quad + \sum_{l,k=1}^m d_{lk}(z) u_{y_ly_k} +\sum_{l=1}^m e_l(z) u_{y_l},\quad\forall\, u\in C^2(B^{\infty}_R),
\end{aligned}
\end{equation}
where we assume that the coefficients $\{b_i(z):1\leq i\leq n\}$ satisfy the cleanness condition,
\begin{align}
\label{eq:A_tangent}
&b_i = 0 \quad\hbox{ on } \partial B^{\infty}_R\cap\{x_i=0\},\\
\label{eq:A_transverse}
&\hbox{or } b_i > 0 \quad\hbox{ on } \partial B^{\infty}_R\cap\{x_i=0\}.
\end{align}
We denote by $I^T$ the set of indices $1\leq i\leq n$ such that
property \eqref{eq:A_tangent} holds and we denote by $I^{\pitchfork}$
the set of indices $1\leq i\leq n$ such that property
\eqref{eq:A_transverse}. In this appendix it is assumed that neither set
  is empty. We also denote:
\begin{align*}
F_i &:=\partial B^{\infty}_R\cap\{x_i=0\},\quad\forall\, 1\leq i\leq n,\\
\partial^T B^{\infty}_R &:= \bigcup_{i\in I^T} F_i.
\end{align*}
We can now state:

\begin{lem}[Growth lemma]
\label{lem:Growth_lemma}
\cite{Safonov_2016}
Let $n,m$ be non-negative integers such that $n\geq 2$ and let $0<R, \nu<1$. 
Suppose that the coefficients of the operator $A$ in \eqref{eq:A_operator} are bounded continuous functions on $\bar B^{\infty}_R$ and that there are positive constants, $\delta$ and $K$, such that for all $1\leq i, j \leq n$ and $1\leq l,k\leq m$, we have that
\begin{align}
\label{eq:A_boundedness}
&\|a_{ij}\|_{C(\barB^{\infty}_R)} +\|b_i\|_{C(\barB^{\infty}_R)}+ \|c_{il}\|_{C(\barB^{\infty}_R)} 
+ \|d_{lk}\|_{C(\barB^{\infty}_R)} + \|e_l\|_{C(\barB^{\infty}_R)} 
\leq K,
\\
\label{eq:A_uniform_ellipticity}
&\sum_{i,j=1}^n a_{ij}(z)\xi_i \xi_j + \sum_{l,k=1}^m d_{lk}(z)\eta_l\eta_k \geq \delta(|\xi|^2 + |\eta|^2),
\quad\forall\, z\in \barB^{\infty}_R, \forall\,\xi\in\RR^n, \forall\,\eta\in \RR^m,
\end{align}
and the coefficients $\{b_i(z):1\leq i\leq n\}$ satisfy one of conditions \eqref{eq:A_tangent} or \eqref{eq:A_transverse}.
Let $i \in I^T$ and $j \in I^{\pitchfork}$. Then there is a positive constant, $\theta \in (0,1)$, depending only on $b_0,\delta,K,\nu, m$, and $n$, such that if 
\begin{equation}
\label{eq:A_u_reg}
u\in C^2(\bar B^{\infty}_R\backslash\partial^T B^{\infty}_R)\cap C(\bar B^{\infty}_R\backslash (F_i\cap F_j))
\end{equation}
is a solution such that
\begin{align}
\label{eq:A_u}
A u = 0&\quad\hbox{ on } \quad\bar B^{\infty}_R\backslash\partial^T B^{\infty}_R,\\
\label{eq:A_u_boundary}
0\leq u \leq \nu &\quad\hbox{ on } \quad\partial^T B^{\infty}_R\backslash (F_i\cap F_j),\\
\label{eq:A_u_bounds}
0 \leq u \leq 1 &\quad\hbox{ on } \quad \barB^{\infty}_R\backslash (F_i\cap F_j),
\end{align}
then $u$ satisfies the growth property
\begin{equation}
\label{eq:Growth_property}
M(r;1/2) \leq \theta M(r; 1),\quad\forall\, r\in (0,R).
\end{equation}
\end{lem}

The proof of Lemma \ref{lem:Growth_lemma} is based on an induction argument, which applies a comparison principle, and on the following scaling property of the operator $A$ defined in \eqref{eq:A_operator}:

\begin{rmk}[Scaling property of the operator $A$]
\label{rmk:A_scaling}
For $\lambda>0$, we consider the rescaling described in \eqref{eq:Scaling}, and using \eqref{eq:A_u} we notice that the rescaled function $v(z')$ satisfies the equation $A' v(z') = 0$, for all 
$z'=(x',y')\in B^{\infty}_{R/\lambda}$, where the operator $A'$ is given by
\begin{align*}
A'v(z') 
&=\sum_{i,j=1}^n \sqrt{x'_ix'_j}a_{ij}(\lambda x', \sqrt{\lambda} y') v_{x'_ix'_j} 
+\sum_{i=1}^n b_i(\lambda x', \sqrt{\lambda} y') v_{x'_i}
+\sum_{i=1}^n \sum_{l=1}^m \sqrt{\lambda x'_i} c_{il}(\lambda x', \sqrt{\lambda} y') v_{x'_iy'_l}\\
&\quad + \sum_{l,k=1}^m d_{lk}(\lambda x', \sqrt{\lambda} y') v_{y'_ly'_k} 
+\sum_{l=1}^m \sqrt{\lambda} e_l(\lambda x', \sqrt{\lambda} y') v_{y'_l}.
\end{align*}
Thus, by rescaling the solutions as in \eqref{eq:Scaling}, the rescaled
functions are solutions to an equation defined by a new operator, $A'$, that
satisfies the same properties as the original operator, $A$, when we choose the
parameter $\lambda$ in a bounded set. In the proof of Lemma
\ref{lem:Growth_lemma}, we apply this scaling property for $\lambda\in
(0,1)$.

Note that in the $n=2, m=0$ case we would have a function $u$ defined on $[0,1]\times
[0,1],$ which is not known to be (and in general will not be) continuous at $(0,0).$
This precludes a simple application of the maximum principle. If we can show
that $u(1/2,z)\leq\theta,$ and $u(z,1/2)\leq\theta$ for $z\in(0,1/2),$ in a manner that only
depends on bounds on the coefficients of the operator and the other hypotheses
on $u,$ then the scaling argument shows that, for $r\in (0,1),$
\begin{equation}
  u(rz,r/2)\leq\theta\text{ and }u(r/2,rz)\leq\theta\text{ for }z\in (0,1/2),
\end{equation}
as well. In other words $u(x,y)\leq \theta,$ for $(x,y)\in (0,1/2)\times
(0,1/2),$ which provides an effective replacement for the maximum
principle. This is, in essence, how the argument below proceeds. The proof that
we provide is a small modification of an argument communicated to us by
M.V. Safonov, \cite{Safonov_2016}.
\end{rmk}

Remark \ref{rmk:A_scaling} implies that to establish the growth property \eqref{eq:Growth_property}, we can assume without loss of generality that $R=1$ and $M(1;1)=1$, and it is sufficient to prove that there is a constant, $\theta \in (0,1)$, such that
\begin{equation}
\label{eq:Growth_lemma_normalized}
M(1;1/2) \leq \theta.
\end{equation}
Because our proof of Lemma \ref{lem:Growth_lemma} relies on an induction argument, we begin with the proof of the base case in the induction. 

\begin{lem}[Growth lemma with $n=2$ and $m\geq 0$]
\label{lem:GL_base_case}
\cite{Safonov_2016}
In the statement of Lemma \ref{lem:Growth_lemma} assume that $n=2$, $m\geq 0$, $i=1$, and $j=2$. Then the conclusion of Lemma 
\ref{lem:Growth_lemma} holds.
\end{lem}

\begin{proof}
We establish property \eqref{eq:Growth_lemma_normalized} in several steps.

\setcounter{step}{0}

\begin{step}
\label{step:C_1_base}
In this step we prove that there are positive constants, $h,\theta_1\in (0,1)$, such that
\begin{equation}
\label{eq:C_1}
u(x,y) \leq \theta_1,\quad\forall\, (x,y) \in \left(\bigcup_{r\in (0,1)} [0,hr]\times\left\{\frac{r}{2}\right\}\right)
\times\left[-\frac{1}{2},\frac{1}{2}\right]^m.
\end{equation}
We let $H\in (0,1)$ and we consider the set
$
W_2 := (0,H)\times (1/4,3/4)\times(-1,1)^m,
$
and the barrier function
$$
w_2(z) = \nu+16\left(x_2-\frac{1}{2}\right)^2 + \sqrt{\frac{x_1}{H}} + \sum_{l=1}^m y_l^2,
$$
where we recall that $\nu$ belongs to $(0,1);$  it is the constant appearing in \eqref{eq:A_u_boundary}.
Direct calculations give us that
$$
A w_2 = 32\left(x_2a_{22} + b_2\left(x_2-\frac{1}{2}\right)\right) + \frac{2b_1-a_{11}}{4\sqrt{Hx_1}} + \sum_{l=1}^m (d_{ll}+e_ly_l),
$$
and, using the assumption that $b_1=0$ along $\partial B^{\infty}_1\cap\{x_1=0\}$ together with \eqref{eq:A_boundedness} and \eqref{eq:A_uniform_ellipticity}, we see that we can find positive constants, $C$ and $H$, such that
$$
A w_2 \leq -\frac{C}{\sqrt{Hx_1}} + C < 0\quad\hbox{ on } W_2.
$$ 
We also see that on the boundary of the set $W_2$ the following hold. By property \eqref{eq:A_u_boundary}, we have that $u\leq\nu \leq w_2$ on $\partial W_2\cap\{x_1=0\}$. By property \eqref{eq:A_u_bounds}, we have that $u\leq 1\leq w_2$ on $\partial W_2\cap S_{n,m}$. 
%CP 11.26.2016: \partial W_2 does not intersect \{x_2=0\}, and so this sentence is not needed
%We note that by adapting the argument of the proof of \cite[Proposition 3.1.1]{Epstein_Mazzeo_annmathstudies} to the case of the elliptic problem for the operator $A$, we do not need to have that $u\leq w_2$ on the portion of the boundary $\partial W_2\cap\{x_2=0\}$ due to the regularity assumption \eqref{eq:A_u_reg} and to the fact that the operator $A$ is degenerate as we approach this boundary portion of $W_2$. 
Applying the comparison principle, it follows that $u\leq w_2$ on $W_2$. From the choice of the barrier function $w_2$, we see that for all $\theta_1\in (\nu,1)$, there is a positive constant $h\in (0,1)$ such that $u(x,y) \leq \theta_1$, for all points $(x,y)\in \bar W_2$ with the property that $x_2=1/2$, $x_1\in [0,h]$, and $y_l\in [-h,h]$, for all $1\leq l\leq m$. 

Applying Remark \eqref{rmk:A_scaling} with $\lambda := r\in (0,1)$ arbitrarily chosen, we obtain that $u(x,y) \leq \theta_1$, for all points $(x,y)$ with the property that there is $r\in (0,1)$ such that $x_2=r/2$, $x_1\in [0,hr]$, and $y_l\in [-h\sqrt{r},h\sqrt{r}]$, for all 
$1\leq l\leq m$. We note that the scaling property described in Remark \ref{rmk:A_scaling} continues to hold when we apply translations in the $y$-coordinates of solutions. From here we deduce that $u(x,y) \leq \theta_1$, for all points $(x,y)$ with the property that 
$y_l\in [-1/2,1/2]$, for all $1\leq l\leq m$, and there is $r\in (0,1)$ such that $x_2=r/2$ and $x_1\in [0,hr]$. This proves that 
\eqref{eq:C_1} holds.
\end{step}

\begin{step}
\label{step:Extension_base}
In this step, we extend property \eqref{eq:C_1} in the sense that we prove that for all $k\in (0,1/2)$ there is a constant, $\theta_2=\theta_2(k,\theta_1)\in (0,1)$, such that
\begin{equation}
\label{eq:C_1_extension}
u(x,y)\leq \theta_2,\quad\forall\, 
(x,y)\in \left[0,\frac{1}{2}\right]\times\left[k,\frac{1}{2}\right]\times \left[-\frac{1}{2},\frac{1}{2}\right]^m.
\end{equation}
For $k\in (0,1/2)$, we consider the sets $D_1:=[hk,1/2]\times[k,1/2]$ and $D_2 := [hk/2,1]\times[k/2,1]$, which have the property that $D_1\subset D_2$ and the following hold:
\begin{align*}
0\leq u\leq \theta_1<1 &\quad\hbox{ on } \partial D_2\cap\{x_1=k/2\},\quad\hbox{ (by Step \ref{step:C_1_base} applied with $r=k$)}\\
0\leq u\leq 1 &\quad\hbox{ on } D_2,\quad\hbox{ (by \eqref{eq:A_u_bounds})}\\
Au = 0 &\quad\hbox{ on } D_2,\quad\hbox{ (by \eqref{eq:A_u})}.
\end{align*}
Because the operator $A$ is strictly elliptic on $D_2$, the preceding properties of $u$ show that we can apply the strong maximum principle to conclude that there is a positive constant, $\theta_3=\theta_3(A,k)\in (0,1)$, such that $u\leq \theta_3$ on $D_1$. Combining this property with inequality \eqref{eq:C_1} and letting $\theta_2:=\theta_1\vee\theta_3$, we obtain 
\eqref{eq:C_1_extension}.  
\end{step}

\begin{step}
\label{step:C_2_base}
We now prove that there are positive constants, $k\in (0,1/2)$ and $\theta_4\in (0,1)$, such that
\begin{equation}
\label{eq:C_2}
u(x,y)\leq \theta_4,\quad\forall (x,y)\in\left\{\frac{1}{2}\right\}\times (0,k)\times (-k,k)^m.
\end{equation}
For $k\in (0,1/2)$, we let $\theta_2=\theta_2(A,k)\in (0,1)$ be the constant in inequality \eqref{eq:C_1_extension}. We consider the set
$W_1 := (1/4, 3/4) \times (0,k) \times (-1,1)^m$ and the barrier function
$$
w_1(z) = \theta_2 + (1-\theta_2)\left[16\left(x_1-\frac{1}{2}\right)^2 + \beta\left(k-x_2\right) 
+ \frac{1}{2}+\frac{1}{2}\sum_{l=1}^m y_l^2\right], 
$$
where the positive constant $\beta$ is suitably chosen below.
Because $b_1>0$ along $\partial W_1\cap\{x_2=0\}$, we can choose $k$ small enough such that there is a positive constant, $b_0$, with the property that $b_1\geq b_0$ on $\bar W_1$.
Direct calculations give us that
\begin{align*}
Aw_1 &\leq (1-\theta_2)\left[32\left(x_2a_{22}+b_2\left(x_2-\frac{1}{2}\right)\right) -\beta b_0
+ \frac{1}{2} \sum_{l=1}^m(d_{ll}+e_ly_l)\right],
\end{align*}
and so we can choose $\beta=\beta(b_0,K)$ large enough so that $Aw_1 <0$ on $W_1$. We next prove that $u\leq w_1$ on $\partial W_1 \cap S_{n,m}$. We recall that by adapting the argument of the proof of 
\cite[Proposition 3.1.1]{Epstein_Mazzeo_annmathstudies} to the case of elliptic problem for the operator $A$, we do not need to have that $u\leq w_1$ on the portion of the boundary $\partial W_1\cap\{x_2=0\}$ due to the regularity assumption \eqref{eq:A_u_reg} and to the fact that the operator $A$ is degenerate as we approach this boundary portion of $W_1$. We see that on $\partial W_1\cap\{x_2=k\}$, we have that $w_1\geq \theta_2$ and $u\leq \theta_2$ by inequality \eqref{eq:C_1_extension}, and so $u\leq w_1$ on 
$\partial W_1\cap\{x_2=k\}$. On the portion of the boundary $\partial W_1\cap \{x_1=1/4\hbox{ or } x_1=3/4\}$, it is clear that 
$u\leq w_1$ because on this set we have that $w_1\geq 1$ by construction and $u\leq 1$ by \eqref{eq:A_u_bounds}. Thus, the comparison principle implies that $u\leq w_1$ on $W_1$. In particular, when $z=(x,y)$ has the property that $x_1=1/2$, $x_2\in (0,k)$, and 
$y_l\in (0,k)$, for all $1\leq l\leq m$, we have that
$$
u(z) \leq \theta_2+(1-\theta_2)\left(\beta k+ \frac{1}{2}+\frac{1}{2} mk^2\right).
$$  
We can choose the positive constant $k$ small enough such that there is a positive constant $\theta_4\in (0,1)$ with the property that
$u(x,y)\leq \theta_4$, for all $(x,y)$ such that $x_1=1/2$ and $x_2,y_l\in (0,k)$, for all $1\leq l\leq m$. We recall that the scaling property in Remark \ref{rmk:A_scaling} continues to hold when we apply translations in the $y$-coordinates of solutions. This concludes the proof of inequality \eqref{eq:C_2}.
\end{step}

Let $k\in (0,1/2)$ and $\theta_4$ be chosen as in Step \ref{step:C_2_base} and let $\theta_2$ be chosen as in Step \ref{step:Extension_base}. Let $\theta:=\theta_2\vee\theta_4$. Inequalities \eqref{eq:C_1_extension} and \eqref{eq:C_2} give us that $u(x,y)\leq \theta$, for all $(x,y)$ such that $y\in (0,k)^m$, $x_1=1/2$ and $x_2\in (0,1/2)$, or $x_2=1/2$ and $x_1\in (0,1/2)$. Remark \ref{rmk:A_scaling} implies that $u(x,y)\leq \theta$, for all $(x,y)$ such that there is $r\in (0,1)$ with the property that $y\in (0,\sqrt{r}k)^m$, $x_1=r/2$ and $x_2\in (0,r/2)$, or $x_2=r/2$ and $x_1\in (0,r/2)$. We recall that the scaling property in Remark \ref{rmk:A_scaling} continues to hold when we apply translations in the $y$-coordinates of solutions. From here we deduce that $u(x,y) \leq \theta$, for all points $(x,y)\in Q(1;1/2)$, where we recall the definition of $Q(1;1/2)$ in \eqref{eq:Q_ball}. This concludes the proof of inequality \eqref{eq:Growth_lemma_normalized}, when $n=2$ and $m\geq 0$.
\end{proof}

We can now give

\begin{proof}[Proof of Lemma \ref{lem:Growth_lemma}]
We assume without loss of generality that $i=1$, $j=2$, and $R=1$. We prove inequality \eqref{eq:Growth_lemma_normalized} by an induction argument on $n$. The base case, $n=2$, was established in Lemma \ref{lem:GL_base_case}. We next consider the induction step. We assume that \eqref{eq:Growth_lemma_normalized} holds with $n$ replaced by $n-1$ and we want to establish it for $n$. We prove this assertion in several steps, which are adaptations of the steps of the proof of Lemma \ref{lem:GL_base_case} to the multi-dimensional case ($n>2$). For the multi-dimensional case, we do not need an adaptation of Step \ref{step:C_1_base} in Lemma \ref{lem:GL_base_case}. 

\setcounter{step}{0}

\begin{step}
\label{step:Extension}
Analogously to Step \ref{step:Extension_base} in the proof of Lemma \ref{lem:GL_base_case}, we prove that for all $k\in (0,1/2)$ there is a constant, $\theta_1=\theta_1(A,k)\in (0,1)$, such that
\begin{equation}
\label{eq:Comparison_1_extension}
u(x,y)\leq \theta_1,\quad\forall\, 
(x,y)\in \left[0,\frac{1}{2}\right]\times
\left(\left[0,\frac{1}{2}\right]^{n-1}\backslash \left[0,k\right]^{n-1}\right)
\times \left[-\frac{1}{2},\frac{1}{2}\right]^m.
\end{equation}
For $k\in (0,1/2)$, we consider the sets 
\begin{align*}
D_1 & := \left[0,\frac{1}{2}\right]\times
\left(\left[0,\frac{1}{2}\right]^{n-1}\backslash \left[0,k\right]^{n-1}\right)
\times \left[-\frac{1}{2},\frac{1}{2}\right]^m,\\
D_2 & := \left[0, 1\right]\times
\left([0,1]^n\backslash\left[0,k/2\right]^{n-1}\right)
\times \left[-1,1\right]^m,
\end{align*}
which have the property that $D_1\subset D_2$ and the following hold:
\begin{align*}
0\leq u\leq \nu<1 &\quad\hbox{ on } \partial^T D_2\backslash (F_1\cap F_2),
\quad\hbox{ (by \eqref{eq:A_u_boundary} with $i=1$ and $j=2$)}\\
0\leq u\leq 1 &\quad\hbox{ on } D_2\backslash (F_1\cap F_2),\quad\hbox{ (by \eqref{eq:A_u_bounds})}\\
Au = 0 &\quad\hbox{ on } D_2\backslash \partial^T D_2,\quad\hbox{ (by \eqref{eq:A_u})},
\end{align*}
where we let $\partial^T D_2 := \partial^T B^{\infty}_1\cap D_2$. Because on the set $D_1$ the operator $A$ can be viewed as a degenerate operator of the form \eqref{eq:A_operator} defined on $S_{n-1,m+1}$, as opposed to $S_{n,m}$, we can apply the induction hypothesis and by covering $D_1$ by a finite number of balls, we obtain that  there is a positive constant, $\theta_1=\theta_1(A,k)\in (0,1)$, such that $u\leq \theta_1$ on $D_1$. This completes the proof of inequality \eqref{eq:Comparison_1_extension}.  
\end{step}

\begin{step}
\label{step:Comparison_2}
Analogously to Step \ref{step:C_2_base} in Lemma \ref{lem:GL_base_case}, we next prove that there are positive constants, $k\in (0,1/2)$ and $\theta_2\in (0,1)$, such that
\begin{equation}
\label{eq:Comparison_2}
u(x,y)\leq \theta_2,\quad\forall (x,y)\in\left\{\frac{1}{2}\right\}\times (0,k)^{n-1}\times (-k,k)^m.
\end{equation}
We fix $k\in (0,1/2)$ and let $\theta_1=\theta_1(A,k)\in (0,1)$ be chosen as in Step \ref{step:Extension}.
We consider the set
$W_1 := (1/4, 3/4) \times (0,k) \times (-1,1)^m$ and the barrier function
$$
w_1(z) = \theta_2 + (1-\theta_2)\left[16\left(x_1-\frac{1}{2}\right)^2 + \beta\left(k-x_2\right) 
+ \frac{1}{2}+\frac{1}{2}\left(\sum_{i=3}^n x_i + \sum_{l=1}^m y_l^2\right)\right], 
$$
where $\beta$ is a positive constant. The argument of the proof of Step \ref{step:C_2_base} Lemma \ref{lem:GL_base_case} immediately adapts to the present choice of the set $W_1$ and of the barrier function $w_1$, and we obtain that we can choose $\beta$ and $k$ small enough so that there is a constant, $\theta_2\in (0,1)$, such that inequality \eqref{eq:Comparison_2} holds. This completes the argument of Step \ref{step:Comparison_2}.
\end{step} 

Let $k\in (0,1/2)$ and $\theta_2$ be chosen as in Step \ref{step:Comparison_2} and let $\theta_1$ be chosen as in Step 
\ref{step:Extension}. Let $\theta:=\theta_1\vee\theta_2$. Inequalities \eqref{eq:Comparison_1_extension} and \eqref{eq:Comparison_2} gives us that $u(x,y)\leq \theta$, for all $(x,y)\in  S_{n,m}\cap \partial (0,1/2)^n\times(-k,k)^m$. Remark \ref{rmk:A_scaling} implies that $u(x,y)\leq\theta$, for all $(x,y)\in (0,1/2)^n\times(-k,k)^m$. We recall that the scaling property in Remark \ref{rmk:A_scaling}  holds when we apply translations in the $y$-coordinates of solutions. From here we deduce that $u(x,y) \leq \theta$, for all points $(x,y)\in Q(1;1/2)$, where we recall the definition of $Q(1;1/2)$ in \eqref{eq:Q_ball}. This completes the proof of inequality \eqref{eq:Growth_lemma_normalized}.
\end{proof}

We next establish that the probability of hitting the intersection of a tangent and a transverse boundary component defines a function 
(by formula \eqref{eq:Function_hitting_corner}) that satisfies the hypotheses of Lemma \ref{lem:Growth_lemma}.

\begin{lem}[Regularity of the hitting probability]
\label{lem:Regularity_hitting_probability}
Suppose that the generalized Kimura operator satisfies the standard assumptions.
Let $i\in I^T$ and $j\in I^{\pitchfork}$. Then the hitting probability,
\begin{equation}
\label{eq:Function_hitting_corner}
u(p) := \QQ^p(\omega(\tau_{\partial^T P}) \in H_i\cap H_j),\quad\forall\, p \in P, 
\end{equation}
belongs to the space of functions 
\begin{equation}
\label{eq:Nice_space_elliptic}
C^{\infty}(P\backslash \partial^T P) \cap C(P \backslash (H_i\cap H_j))
\end{equation}
and satisfies properties 
\begin{align}
\label{eq:L_u}
L u = 0&\quad\hbox{ on } \quad P\backslash \partial^T P,\\
\label{eq:L_u_boundary}
u=0 &\quad\hbox{ on } \quad \partial^T P\backslash (H_i\cap H_j),\\
\label{eq:L_u_bounds}
0 \leq u \leq 1 &\quad\hbox{ on } \quad P.
\end{align}
\end{lem}

\begin{proof}
From definition \eqref{eq:Function_hitting_corner} of $u$, it is clear that property \eqref{eq:L_u_bounds} holds. Applying Lemma 
\ref{eq:Absorption_tangent_boundary} with $p\in \partial^T P\backslash (H_i\cap H_j)$, we also have that $u$ satisfies property \eqref{eq:L_u_boundary}. To prove the remaining assertions of Lemma \ref{lem:Regularity_hitting_probability}, we use an approximation argument. Let $k\in\NN$, $\varphi_k:P\rightarrow[0,1]$, and $\psi_k:[0,\infty]\rightarrow [0,1]$ be smooth functions such that
\begin{equation}
\label{eq:Definition_eta_k_elliptic_Dirichlet}
\begin{aligned}
&\varphi_k = 1\quad\hbox{ on }\quad P\cap\{p\in P: \hbox{ dist}(p, H_i\cap H_j) \leq 1/k\},\\
&\varphi_k = 0\quad\hbox{ on }\quad P \backslash \{p\in P: \hbox{ dist}(p, H_i\cap H_j) \geq 2/k\},\\
&\psi_k(t) = 1\quad\hbox{ for all } t\geq 2/k,
\quad\hbox{ and }\quad
\psi_k(t) = 0\quad\hbox{ for all } t\leq 1/k,
\end{aligned}
\end{equation} 
and let $\zeta_k(t,p) := \psi_k(t)\varphi_k(p)$, for all $(t,p)\in [0,\infty)\times P$. Let $v_k$ be the unique solution in the space of functions \eqref{eq:Nice_space} to the non-homogeneous Dirichlet problem \eqref{eq:Parabolic_Dirichlet} with boundary condition $\zeta=\zeta_k$. The stochastic representation \eqref{eq:Parabolic_Dirichlet_rep_stoch} gives us that
\begin{equation}
\label{eq:Stoch_rep_v_k}
v_k(t,p) = \EE_{\QQ^p}\left[\psi_k(t-t\wedge\tau_{\partial^T P})\varphi_k(\omega(t\wedge\tau_{\partial^T P}))\right],
\quad\forall\, (t,p)\in [0,\infty)\times P,
\end{equation}
where we recall that the stopping time $\tau_{\partial^T P}$ is defined in \eqref{eq:tau_tangent} and the probability measure $\QQ^p$ is the unique solution to the martingale problem in Definition \ref{defn:Martingale_problem}.

We divide the proof into several steps. 

\setcounter{step}{0}
\begin{step}[Convergence as $t\rightarrow\infty$ and $k\rightarrow\infty$]
\label{step:Convergence_seq}
From the definition of the function $\psi_k$ we have that
\begin{equation}
\label{eq:psi_tau}
\begin{aligned}
\psi_k(t-t\wedge\tau_{\partial^TP})&=1\quad\hbox{ if }\tau_{\partial^TP} < t-1/k,\\
\psi_k(t-t\wedge\tau_{\partial^TP})&=0\quad\hbox{ if }t-2/k <\tau_{\partial^TP},
\end{aligned}
\end{equation}
from which it follows that
\begin{align*}
\left|v_k(t,p) - \EE_{\QQ^p}\left[\varphi_k(\omega(t\wedge\tau_{\partial^T P}))\mathbf{1}_{\{\tau_{\partial^TP} < t-1/k\}}\right]\right|
\leq \QQ^p(t-1/k \leq \tau_{\partial^TP} < t-2/k).
\end{align*}
Letting $t$ tend to $\infty$, we see that the right-hand side of the preceding inequality tends to $0$ because
$$
\sum_{a=0}^{\infty} \QQ^p(a/k \leq \tau_{\partial^TP} < (a+1)/k) = \QQ^p(\tau_{\partial^TP} < \infty) <\infty.
$$
Thus, we have that the sequence of functions $v_k(t,\cdot)$ converges pointwise on $P$, as $t\rightarrow\infty$, for each fixed $k\in \NN$, and we denote
\begin{equation}
\label{eq:Stoch_rep_u_k}
u_k(p):=\lim_{t\rightarrow\infty} v_k(t,p) = \EE_{\QQ^p}\left[\varphi_k(\omega(\tau_{\partial^T P}))\right],
\quad\forall\, k\in \NN,\quad\forall\, p \in P.
\end{equation}
From the definition of the cutoff functions $\varphi_k$ in \eqref{eq:Definition_eta_k_elliptic_Dirichlet} and of the hitting probability $u$ in \eqref{eq:Function_hitting_corner}, we have that
$$
\lim_{k\rightarrow\infty} u_k(p) = u(p),\quad\forall\, p \in P.
$$
We use this construction of the function $u$ to prove that it belongs to the space of functions \eqref{eq:Nice_space_elliptic}
and satisfies properties \eqref{eq:L_u} and \eqref{eq:L_u_boundary} in the following steps.
\end{step}

\begin{step}[Proof of $u\in C^2(P\backslash\partial^T P)$ and $u$ satisfies \eqref{eq:L_u}]
\label{step:Smoothness_seq}
Note that $|v_k(t,p)|\leq 1$, for all $(t,p)\in [0,\infty)\times P$ and for all $k\in\NN$. Let $U$ be a relatively open set in $P$ such that  
$\hbox{dist}(\bar U, P\backslash\partial^TP)>0$. We can apply \cite[Theorem 1.2]{Pop_2013a} to conclude that, for all $l\in\NN$, there is a positive constant, $C=C(L,l)$, such that
$$
\|v_k(t,\cdot)\|_{C^l(\bar U)} \leq C,\quad\forall\, t\in [0,\infty),\quad\forall\, k\in\NN.
$$
The construction of the function $u$ in Step \ref{step:Convergence_seq}, the
preceding estimate which is uniform in $t\in [0,\infty)$ and $k\in\NN$, and an
application of the Arzel\`a-Ascoli Theorem imply that $u$ belongs to
$C^{\infty}(P\backslash\partial^T P)$, since the relatively open set $U\subset P\backslash\partial^T P$ such that $\hbox{dist}(\bar U,\partial^TP)>0$ was arbitrarily chosen. Moreover, because the functions $v_k$ are solutions to the parabolic equation \eqref{eq:Parabolic_Dirichlet}, the preceding observation implies that the hitting probability $u$ satisfies equation \eqref{eq:L_u}.
\end{step}

\begin{step}[Proof of $u\in C(P\backslash(H_i\cap H_j))$ and of \eqref{eq:L_u_boundary}]
\label{step:Continuity_seq}
To establish that $u\in C(P\backslash(H_i\cap H_j))$ and satisfies \eqref{eq:L_u_boundary}, it is sufficient to prove that
$$
\lim_{q\rightarrow p} u(q) =0,\quad\forall\, p\in\partial^T P\backslash(H_i\cap H_j).
$$
Let $p\in \partial^T P \backslash(H_i\cap H_j)$. Then there is a positive
constant $r$ such that $B_r(p)\subset P\backslash(H_i\cap H_j)$. Let $\varsigma$
be the first hitting time of the Kimura diffusion on the set 
$(\partial B_r(p)\cap\partial^T P)\cup(\partial B_r(p)\cap\Int(P))$, where we recall that $B_r(p)$ is the relatively open ball centered at $p$ of radius $r$ with respect to the Riemannian metric induced by the generalized Kimura operator on $P$. The strong Markov property of Kimura diffusions established in Corollary \ref{cor:Strong_Markov} and the stochastic representation \eqref{eq:Function_hitting_corner} of the function $u$ give us that
\begin{equation}
\label{eq:stoch_rep_local}
u(q) = \EE_{\QQ^q}\left[\xi(\omega(\varsigma))\right],\quad\forall\, q\in \bar B_r(p). 
\end{equation}
where $\xi=0$ on $\partial^T P\cap \partial B_r(p)$ and $\xi=u$ on $\partial B_r(p)\cap\Int(P)$. Without loss of generality we can choose an adapted system of coordinates in a neighborhood of $p$ such that $p$ is equivalent to the origin in $\bar S_{n,m}$ and the generalized Kimura operator takes the form \eqref{eq:Operator} on $B^{\infty}_R$, for some $R>0$ small enough, such that $b_1=0$ on $\partial B^{\infty}_R \cap\{x_1=0\}$. Then, in the adapted system of coordinates, the local stochastic representation \eqref{eq:stoch_rep_local} becomes
\begin{equation}
\label{eq:stoch_rep_local_u}
u(z) = \EE_{\QQ^z} \left[\xi(\omega(\nu))\right],\quad\forall\, z\in\bar B^{\infty}_{\rho},
\end{equation}
where $0<\rho<R$, $\nu$ is the first hitting time of the boundary of the set 
$(\partial^TB^{\infty}_{\rho}\cap \partial B^{\infty}_{\rho}) \cup (\partial B^{\infty}_{\rho}\cap S_{n,m})$, and $\xi=0$ on 
$\partial^T B^{\infty}_{\rho}\cap \partial B^{\infty}_{\rho}$ and $\xi=u$ on $\partial
B^{\infty}_{\rho}\cap S_{n,m}$. The constant $\rho$ is suitably chosen
below. Our goal is to prove that 
\begin{equation}
\label{eq:lim_origin}
\lim_{z\rightarrow 0} u(z) = 0.
\end{equation}
Similarly to Step \ref{step:C_1_base} in the proof of Lemma \ref{lem:GL_base_case}, we choose the barrier function:
$$
w(z) = \sqrt{\frac{x_1}{\rho}} +\sum_{i=2}^n x_i + \sum_{l=1}^m y_l,\quad\forall\, z=(x,y)\in\barB^{\infty}_{\rho}.
$$
We clearly have that $w(0)=0$, $w(z)>0$, for all $z\in\barB^{\infty}_{\rho}\backslash\{0\}$, and direct calculations give us that
$$
Lw(z) = \frac{-a_{11}(z)+2b_1(z)}{4\sqrt{\rho x_1}} + \sum_{i=2}^n b_i(z) +\sum_{l=1}^m 2(d_{ll}(z)+e_l(z)y_l),\quad\forall\, z\in\barB^{\infty}_{\rho}.
$$
Using the fact that $b_1=0$ along $\partial B^{\infty}_{\rho}\cap \{x_1=0\}$ together with the continuity and boundedness of the coefficients, and the fact that $a_{11}\geq \delta>0$ on $\barB^{\infty}_{\rho}$ (by Assumption \ref{assump:Coeff}), it follows that we can choose $\rho>0$ small enough so that $Lw < 0$ on $\barB^{\infty}_{\rho}$. Thus, applying It\^o's rule, we have that
\begin{equation}
\label{eq:stoch_rep_w}
\EE_{\QQ^z}\left[w(\omega(\nu))\right] \leq w(z),\quad\forall\, z\in\barB^{\infty}_{\rho}.
\end{equation}
Using the definition of the function $\xi$ and of the stopping time $\nu$, together with fact that $w(0)=0$ and $w(z)>0$, for all 
$z\in\barB^{\infty}_{\rho}\backslash\{0\}$, we find that there is a positive constant, $K$, such that $\xi(\omega(\nu)) \leq Kw(\omega(\nu))$. Identity \eqref{eq:stoch_rep_local_u} together with inequality \eqref{eq:stoch_rep_w} yield
$$
u(z) \leq K w(z),\quad\forall\, z\in\barB^{\infty}_{\rho},
$$
which implies that
$$
\limsup_{z\rightarrow 0} u(z) \leq K\limsup_{z\rightarrow 0} w(z) = 0.
$$
Because the function $u$ is non-negative, we see that the preceding inequality implies \eqref{eq:lim_origin}. This completes the proof of the fact that $u$ is continuous at $p$, for all $p\in\partial^T P\backslash (H_i\cap H_j)$.
\end{step}

Combining Steps \ref{step:Convergence_seq}, \ref{step:Smoothness_seq}, and \ref{step:Continuity_seq} completes the proof.
\end{proof}

%%%%%%%%%%%%%%%%%%%%%%%%%%%%%%%%%%%%%%%%%%%%%%%%%%%%%%%%%%%%%%%%%%%%%%%%%%%%%%%
%
%                                bibliography
%
%%%%%%%%%%%%%%%%%%%%%%%%%%%%%%%%%%%%%%%%%%%%%%%%%%%%%%%%%%%%%%%%%%%%%%%%%%%%%%%

\bibliography{mfpde}
\bibliographystyle{amsplain}

\end{document}

%% file: latexformat.tex
\newtheorem{thm}{Theorem}[section]
\newtheorem*{thm*}{Theorem}
\newtheorem{lem}[thm]{Lemma}
\newtheorem*{lem*}{Lemma}

\newtheorem{cor}[thm]{Corollary}
\newtheorem{claim}[thm]{Claim}
\newtheorem{prop}[thm]{Proposition}

\theoremstyle{definition}

\newtheorem{assump}[thm]{Assumption}
\newtheorem{case}{Case}\renewcommand{\thecase}{}
\newtheorem*{case*}{Case}

\newtheorem{defn}[thm]{Definition}
\newtheorem*{defn*}{Definition}

\newtheorem*{exmp*}{Example}

\newtheorem{rmk}[thm]{Remark}
\newtheorem*{rmk*}{Remark}
\newtheorem{step}{Step}\renewcommand{\thestep}{}

\theoremstyle{remark}

\makeatletter
\def\alphenumi{
  \def\theenumi{\alph{enumi}}
  \def\p@enumi{\theenumi}
  \def\labelenumi{(\@alph\c@enumi)}}
\makeatother

\makeatletter
\def\thecase{\@arabic\c@case}
\makeatother

\makeatletter
\def\thestep{\@arabic\c@step}
\makeatother

%% file: latexnewmacros.tex
\newcount\hh
\newcount\mm
\mm=\time
\hh=\time
\divide\hh by 60
\divide\mm by 60
\multiply\mm by 60
\mm=-\mm
\advance\mm by \time
\def\hhmm{\number\hh:\ifnum\mm<10{}0\fi\number\mm}

\renewcommand\emptyset{\varnothing}

\newcommand\barB{{\bar{B}}}

\newcommand\pa{\partial}
\newcommand\restrictedto{\upharpoonright}

\newcommand\EE{\mathbb{E}}

\newcommand\NN{\mathbb{N}}

\newcommand\PP{\mathbb{P}}
\newcommand\QQ{\mathbb{Q}}
\newcommand\RR{\mathbb{R}}

\newcommand\cA{{\mathcal{A}}}
\newcommand\cB{{\mathcal{B}}}

\newcommand\cF{{\mathcal{F}}}
\newcommand\cG{{\mathcal{G}}}

\newcommand\fa{{\mathfrak{a}}}

\newcommand\fb{{\mathfrak{b}}}

\newcommand\fc{{\mathfrak{c}}}

\newcommand\fe{{\mathfrak{e}}}

\newcommand\ff{{\mathfrak{f}}}

\newcommand\eps{\varepsilon}

\DeclareMathOperator{\Int}{int}